\documentclass[graybox]{svmult}

\usepackage{type1cm}
\usepackage{makeidx} 
\usepackage{graphicx}
\usepackage{multicol}
\usepackage[bottom]{footmisc}
\usepackage{mathtools}
\usepackage{newtxtext}
\usepackage{newtxmath}
\makeindex 
\usepackage{mathrsfs}
\usepackage{amsmath}
\usepackage{paralist, url}
\usepackage[colorlinks=true,hypertexnames=false]{hyperref}
\hypersetup{urlcolor=blue, citecolor=red}
\usepackage{bookmark} 
\usepackage{enumitem}
\makeatletter\providecommand*{\toclevel@titlech}{0}\edef\toclevel@authorch{\the\numexpr\toclevel@titlech+1}\makeatother

\newcommand{\strattwo}{Strategy \hyperlink{strat2}{2}}
\newcommand{\stratone}{Strategy \hyperlink{strat1}{1}}

\newcommand{\msc}[1]{\href{https://mathscinet.ams.org/mathscinet/search/mscbrowse.html?sk=default&sk=#1&submit=Chercher}{#1}}
\renewcommand{\qed}{\hfill\ $\square$}
\newcommand\C{{\mathbb C}}
\newcommand\R{{\mathbb R}}
\newcommand\Z{{\mathbb Z}}
\newcommand{\be}[1]{\begin{equation}\label{#1}}
\newcommand{\ee}{\end{equation}}
\renewcommand{\(}{\left(}
\renewcommand{\)}{\right)}

\newcommand{\re}{\mathrm{Re}}
\newcommand{\hh}{h_\star}
\newcommand{\h}{{}}
\newcommand{\T}{\mathbb{T}}
\newcommand{\ip}{\partial_x^{-1}}
\newcommand{\TT}{\mathsf{T}}
\newcommand{\LL}{\mathsf{L}}
\renewcommand{\AA}{\mathsf{A}}
\newcommand{\HH}{\mathsf{H}}
\newcommand{\Pt}{\tilde{P}}

\newcommand{\omu}{\overline{\mu}}
\newcommand{\diag}{\operatorname{diag}}
\newcommand{\norm}[1]{\lVert#1\rVert}

\newcommand{\pa}[1]{\left( #1 \right)}

\newcommand{\eps}{\varepsilon}
\newcommand{\M}{\mathcal M}
\DeclareMathOperator{\avg}{{\mathrm{avg}}}

\newcommand\numberthis{\addtocounter{equation}{1}\tag{\theequation}} 

\makeatletter 
\providecommand*{\toclevel@author}{0} 
\providecommand*{\toclevel@title}{0} 
\makeatother

\title*{Sharpening of decay rates in Fourier based hypocoercivity methods}
\titlerunning{Rates in Fourier based hypocoercivity methods}
\author{Anton Arnold, Jean Dolbeault, Christian Schmeiser, and Tobias W\"ohrer}
\authorrunning{A.~Arnold, J.~Dolbeault, C.~Schmeiser, and T.~W\"ohrer}
\institute{Anton Arnold \at Technische Universit\"at Wien, Institut f\"{u}r Analysis und Scientific Computing, Wiedner Hauptstr.~8, A-1040 Wien, \"Osterreich, \email{anton.arnold@tuwien.ac.at},\newline \url{https://www.asc.tuwien.ac.at/~arnold/}
\and
Jean Dolbeault \at CEREMADE (CNRS UMR n$^\circ$ 7534), PSL university, Universit\'e Paris-Dauphine, Place de Lattre de Tassigny, 75775 Paris 16, France, \email{dolbeaul@ceremade.dauphine.fr},\newline \url{https://www.ceremade.dauphine.fr/~dolbeaul/}
\and
Christian Schmeiser \at Fakult\"at f\"{u}r Mathematik, Universit\"at Wien, Oskar-Morgenstern-Platz 1, 1090 Wien, Austria, \email{Christian.Schmeiser@univie.ac.at},\newline \url{https://homepage.univie.ac.at/christian.schmeiser/}
\and
Tobias W\"ohrer \at Technische Universit\"at Wien, Institut f\"{u}r Analysis und Scientific Computing, Wiedner Hauptstr.~8, A-1040 Wien, \"Osterreich, \email{tobias.woehrer@asc.tuwien.ac.at},\newline \url{https://www.asc.tuwien.ac.at/~arnold/}}

\newcounter{partcounter}
\renewcommand{\part}[2]{\phantomsection\stepcounter{partcounter}\addtocontents{toc}{\par\medskip Part~\Roman{partcounter}.~{#1}}\newpage~\vspace*{1cm}\begin{center}\Large{\bf Part \Roman{partcounter}. #2}\end{center}\vspace*{0.5cm}\setcounter{section}{0}}
\renewcommand{\thesection}{\Roman{partcounter}.\arabic{section}}

\begin{document}
\maketitle\vspace*{-1.5cm}
\thispagestyle{empty}


\abstract{This paper is dealing with two $\mathrm L^2$ hypocoercivity methods based on Fourier decomposition and mode-by-mode estimates, with applications to rates of convergence or decay in kinetic equations on the torus and on the whole Euclidean space. The main idea is to perturb the standard $\mathrm L^2$ norm by a twist obtained either by a nonlocal perturbation build upon diffusive macroscopic dynamics, or by a change of the scalar product based on Lyapunov matrix inequalities. We explore various estimates for equations involving a Fokker--Planck and a linear relaxation operator. We review existing results in simple cases and focus on the accuracy of the estimates of the rates. The two methods are compared in the case of the Goldstein--Taylor model in one-dimension.
\\[4pt]
}

\begin{flushright}\emph\today\end{flushright}

\newpage
\newcommand{\nocontentsline}[3]{}
\newcommand{\tocless}[2]{\bgroup\let\addcontentsline=\nocontentsline#1{#2}\egroup}
\tocless\tableofcontents
\newpage

\subsection*{Introduction}\label{Sec:Introduction}

We consider dynamical systems involving a degenerate dissipative operator and a conservative operator, such that the combination of both operators implies the convergence to a uniquely determined equilibrium state. In the typical case encountered in kinetic theory, the dissipative part is not coercive and has a kernel which is unstable under the action of the conservative part. Such dynamical systems are called \emph{hypocoercive} according to~\cite{Mem-villani}. We are interested in the decay rate of a natural dissipated functional, the \emph{entropy}, in spite of the indefiniteness of the entropy dissipation term. In a linear setting, the functional typically is quadratic and can be interpreted as the square of a Hilbert space norm. Classical examples are evolutions of probability densities for Markov processes with positive equilibria. Over the last~15 years, various hypocoercivity methods have been developed, which rely either on Fisher type functionals (the $\mathrm H^1$ approach) or on entropies which are built upon weighted $\mathrm L^2$ norms, or even weaker norms as in~\cite{armstrong2019variational}. In the $\mathrm L^2$ approach, it is very natural to introduce spectral decompositions and handle the free transport operator, for instance in Fourier variables, as a simple multiplicative operator. In the appropriate functional setting, the problem is then reduced to the study of a system of ODEs, which might be finite or infinite. This is the point of view that we adopt here, with the purpose of comparing several methods and benchmarking them on some simple examples.

\medskip Decay rates are usually obtained by adding a twist to the entropy or squared Hilbert space norm. In hyperbolic systems with dissipation, early attempts can be traced back to the work of Kawashima and Shizuta~\cite{shizuta1985,MR1057534}, where the twist is defined in terms of a \emph{compensating function}. The similarities between hypocoercivity and hypoellipticity are not only motivated the creation of the latter terminology, as explained in~\cite{Mem-villani}, but also serve as a guideline for proofs of hypocoercivity~\cite{Herau,Mouhot-Neumann,Mem-villani} and in particular for the construction of the twist. Here we shall focus on two approaches to $\mathrm L^2$-hypo\-coer\-ci\-vity.\\
$\rhd$ In~\cite{Dolbeault2009511,DMS-2part}, an abstract method motivated by~\cite{Herau} and by the compensating function approach has been formulated, which provides constructive hypocoercivity estimates. The twist is built upon a non-local term associated with the spectral gap of the diffusion operator obtained in the \emph{diffusion limit} and controls the relaxation of the macroscopic part in the limiting diffusion equation, that is, the projection of the distribution function on the orthogonal of the kernel of the dissipative part of the evolution operator. The motivating applications are kinetic transport models with diffusive macroscopic dynamics, see, \emph{e.g.},~\cite{Dolbeault22102012,CRS,goudon:hal-01421710,NeuSch,addala2019l2hypocoercivity,FavSch}, where the results yield decay estimates in an $\mathrm L^2$ setting.\\
$\rhd$ The goal of the second approach is to find sharp decay estimates in special situations, where sufficient explicit information about the dynamics is available. Examples are ODE systems~\cite{AAC20} as well as problems where a spectral decomposition into ODE problems exists~\cite{Arnold2014,ArnEinSigWoe,ASS20,AESW}. In these situations, sharp decay estimates can be derived by employing \emph{Lyapunov matrix inequalities}.

In the standard definitions of hypocoercivity, a spectral gap and an exponential decay to equilibrium are required. This, however, can be expected only in sufficiently confined situations, \emph{i.e.}, in bounded domains or for sufficiently strong confining forces. Problems without or with too weak confinement have been treated either by regaining spectral gaps pointwise in frequency after Fourier transformation as in~\cite{BDMMS,MR2927622} or by employing specially adapted functional inequalities in~\cite{BDMMS,BDS-very-weak,BDLS}, with the Nash inequality~\cite{Nash58} as the most prominent example.

The aim of this work is to present a review and a comparison of the two approaches mentioned above, executed for both confined and unconfined situations, where for the former a periodic setting is chosen, such that the Fourier decomposition method can be used in all cases. A special emphasis is put on optimizing the procedures with the ultimate goal of proving sharp decay rates. Attention is restricted to abstract linear hyperbolic systems with linear relaxation, where 'abstract' means that infinite systems such as kinetic transport equations are allowed. Note that in the finite dimensional case, the setting is as in~\cite{MR2927622}. 

In Part~\hyperlink{PartI}{I} of this work, both methods are presented in an abstract framework. Concerning the method of~\cite{DMS-2part,BDMMS}, the setting is abstract linear ODEs, where the dynamics is driven by the sum of a dissipative and a conservative operator such that the dissipation rate is indefinite, but the conservative operator provides enough mixing to create hypocoercivity. Then the approach based on Lyapunov matrix inequalities is discussed at the hand of hyperbolic systems with relaxation. By Fourier decomposition the problem is reduced to ODE systems and Lyapunov functionals with optimal decay rates are built. The results of this section can be seen as a sharpening of the abstract decay estimates in~\cite{MR2927622}.

Part~\hyperlink{PartII}{II} is concerned with sharpening the approach of~\cite{DMS-2part,BDMMS} applied to linear kinetic equations with centered Maxwellian equilibria (for sake of simplicity). It contains results on the optimal choice of parameters in the abstract setting, on the mode-by-mode application of the method after Fourier transformation, on the convergence of an optimized rate estimate to the sharp rate in the macroscopic diffusion limit and, finally, on the derivation of global convergence or decay rates for the cases of small tori and of the Euclidean space without confinement.

Part~\hyperlink{PartIII}{III} is devoted to a comparison of both approaches for a particular example, the Goldstein--Taylor model with constant exchange rate, a hyperbolic system of two equations with an exchange term in one space dimension, which can be interpreted as a discrete velocity model with two velocities. It has already been used as a model problem in~\cite{DMS-2part}, and the sharp decay rate on the one-dimensional torus has been derived by the Lyapunov matrix inequality approach in~\cite{ArnEinSigWoe}. The challenging problem of finding the sharp decay rate for a position dependent exchange rate has been treated in~\cite{BS,BScorr}. It is shown that the mode-by-mode Lyapunov functionals derived by both methods, the Lyapunov matrix inequality approach and the modal optimization of the abstract framework outlined in Part~\hyperlink{PartII}{II}, coincide for the Goldstein--Taylor model. On the torus the mode-by-mode Lyapunov functionals can be combined into a global Lyapunov functional which provides the sharp decay rate (see Theorem~\ref{thm:AESW}). On the real line, the modal results combine into a global estimate with sharp algebraic decay rate. Due to the presence of a defective eigenvalue in the modal equations, the standard approach requires modifications to obtain reasonable multiplicative constants.

\part{Review of two hypocoercivity methods}{Review of two hypocoercivity methods\hypertarget{PartI}{}}

We consider the abstract evolution equation
\be{EqnEvol}
\frac{dF}{dt}+\mathsf TF=\mathsf LF\,,\quad t>0\,,
\ee
with initial datum $F(t=0,\cdot)=F_0$. Applied to kinetic equations, $\mathsf T$ and $\mathsf L$ are respectively the \emph{transport} and the \emph{collision} operators, but the abstract result of this section is not restricted to such operators. We shall assume that $\mathsf T$ and $\mathsf L$ are respectively anti-Hermitian and Hermitian operators defined on a complex Hilbert space $\big(\mathcal H,\langle\cdot,\cdot\rangle\big)$ with corresponding norm denoted by~$\|\cdot\|$.

\section{An abstract hypocoercivity result based on a twisted \texorpdfstring{$\mathrm L^2$}{L2} norm}\label{Sec:Abstract1}

Let us start by recalling the basic method of~\cite{DMS-2part}. This technique is inspired by diffusion limits and we invite the reader to consider~\cite{DMS-2part} for detailed motivations. We define
\begin{equation}\label{A}
\mathsf A:=\Big(\mathrm{Id}+(\mathsf T\Pi)^*\mathsf T\Pi\Big)^{-1}(\mathsf T\Pi)^*
\end{equation}
where ${}^*$ denotes the adjoint with respect to $\langle\cdot,\cdot\rangle$ and $\Pi$ is the orthogonal projection onto the null space of $\mathsf L$. We assume that positive constants $\lambda_m$, $\lambda_M$, and $C_M$ exist, such that, for any $F\in\mathcal H$, the following properties hold:\\
$\rhd$ \emph{microscopic coercivity}
\begin{equation*}
-\,\langle\mathsf LF,F\rangle\ge\lambda_m\,\|(\mathrm{Id}-\Pi)F\|^2\,,\label{H1}\tag{H1}
\end{equation*}
$\rhd$ \emph{macroscopic coercivity}
\begin{equation*}
\|\mathsf T\Pi F\|^2\ge\lambda_M\,\|\Pi F\|^2\,,\label{H2}\tag{H2}
\end{equation*}
$\rhd$ \emph{parabolic macroscopic dynamics}
\begin{equation*}
\Pi\mathsf T\Pi\,F=0\,,\label{H3}\tag{H3}
\end{equation*}
$\rhd$ \emph{bounded auxiliary operators}
\begin{equation*}
\|\mathsf{AT}(\mathrm{Id}-\Pi)F\|+\|\mathsf{AL}F\|\le C_M\,\|(\mathrm{Id}-\Pi)F\|\,.\label{H4}\tag{H4}
\end{equation*}
A simple computation shows that a solution $F$ of~\eqref{EqnEvol} is such that
\[
\frac12\,\frac d{dt}\|F\|^2=\langle\mathsf LF,F\rangle\le-\,\lambda_m\,\|(\mathrm{Id}-\Pi)F\|^2\,.
\]
We assume that~\eqref{EqnEvol} has, up to normalization, a unique steady state $F_\infty$. By linearity, we can replace $F_0$ by $F_0-\langle F_0,F_\infty\rangle\,F_\infty$ or simply $F_0-F_\infty$, with $F_\infty$ appropriately normalized. With no loss of generality, we can therefore assume that $F_\infty=0$. This is however not enough to conclude that $\|F(t,\cdot)\|^2$ decays exponentially with respect to $t\ge0$. As in the \emph{hypocoercivity} method introduced in~\cite{DMS-2part} for real valued operators and extended in~\cite{BDMMS} to complex Hilbert spaces, we consider the Lyapunov functional
\be{H}
\mathsf H_1[F]:=\frac12\,\|F\|^2+\delta\,\re\langle\mathsf AF,F\rangle
\ee
for some $\delta>0$ to be determined later. If $F$ solves~\eqref{EqnEvol}, then
\be{D}\begin{array}{rl}\hspace*{-7pt}
-\,\frac d{dt}\mathsf H_1[F]=\mathsf D[F]:=&-\,\langle\mathsf LF,F\rangle+\delta\,\langle\mathsf{AT}\Pi F,F\rangle\\[4pt]
&-\,\delta\,\re\langle\mathsf{TA}F,F\rangle+\delta\,\re\langle\mathsf{AT}(\mathrm{Id}-\Pi)F,F\rangle-\delta\,\re\langle\mathsf{AL}F,F\rangle\,.
\end{array}\ee
The following result has been established in~\cite{DMS-2part,BDMMS}.
\begin{theorem}\label{theo:DMS2015} Let $\mathsf L$ and $\mathsf T$ be closed linear operators in the complex Hilbert space $\big(\mathcal H,\langle\cdot,\cdot\rangle\big)$. We assume that $\mathsf L$ is Hermitian and $\mathsf T$ is anti-Hermitian, and that~\eqref{H1}--\eqref{H4} hold for some positive constants $\lambda_m$, $\lambda_M$, and $C_M$. Then for some $\delta>0$, there exists $\lambda>0$ and $C>1$ such that, if $F$ solves~\eqref{EqnEvol} with initial datum $F_0\in\mathcal H$, then
\be{Decay:BDMMS}
\mathsf H_1[F(t,\cdot)]\le\mathsf H_1[F_0]\,e^{-\lambda\,t}\quad\mbox{and}\quad\|F(t,\cdot)\|^2\le C\,e^{-\lambda\,t}\,\|F_0\|^2\quad\forall\,t\ge0\,.
\ee
\end{theorem}
Here we assume that the unique steady state is $F_\infty=0$ otherwise we have to replace $F(t,\cdot)$ by $F(t,\cdot)-F_\infty$ and $F_0$ by $F_0-F_\infty$ in~\eqref{Decay:BDMMS}. The strategy of~\cite{DMS-2part}, later extended in~\cite{BDMMS}, is to prove that for any $\delta>0$ small enough, we have
\be{Decay:H}
\lambda\,\mathsf H_1[F]\le\mathsf D[F]
\ee
for some $\lambda>0$, and
\be{H-norm}
c_-\,\|F\|^2\le\mathsf H_1[F]\le c_+\,\|F\|^2
\ee
for some constants $c_-$ and $c_+$ such that $0\le c_-\le1/2\le c_+$. As a consequence, if \hbox{$c_->0$}, we obtain the estimate $C\le c_+/c_-$. We learn from~\cite[Proposition~4]{BDMMS} that~Theorem~\ref{theo:DMS2015} holds with $c_\pm=(1\pm\delta)/2$,
\be{BDMScomp}
\hspace*{-5pt}\lambda=\frac{\lambda_M}{3\,(1+\lambda_M)}\min\left\{1,\lambda_m,\tfrac{\lambda_m\,\lambda_M}{(1+\lambda_M)\,C_M^2}\right\}\;\mbox{and}\;\delta=\frac12\,\min\left\{1,\lambda_m,\tfrac{\lambda_m\,\lambda_M}{(1+\lambda_M)\,C_M^2}\right\}\,.
\ee
Our primary goal of Part~\hyperlink{PartII}{II} is to obtain sharper estimates of $\lambda$, $c_\pm$ and $C$ for an appropriate choice of~$\delta$ in specific cases. Notice that it is convenient to work in an Hilbert space framework because this allows us to use Fourier transforms.

\section{An abstract hypocoercivity result based on Lyapunov matrix inequalities}\label{Sec:Abstract2}

Here we review our second hypocoercivity method, as developed on various examples in~\cite{Arnold2014, AAC, Achleitner2018}, before comparing it with the method of~\S~\ref{Sec:Abstract1}. 

As in~\S~\ref{Sec:Abstract1}, without loss of generality we assume that~\eqref{EqnEvol} has the unique steady state $F_\infty = 0$. We are interested in explicit decay rates for $\|F(t,\cdot)\|_\h^2 \to0$ as $t\to+\infty$. To fix the ideas we start with some prototypical examples:

\begin{enumerate}[leftmargin=0pt,itemsep=6pt]

\item \label{Ex1} Although almost trivial, the \emph{stable ODEs} with constant-in-$t$ coefficients
\begin{equation}\label{stable-ODE}
\frac{dF}{dt} = -\,C\,F
\end{equation}
is at the core of the method. Here $F(t)\in\C^n$, $\TT:=C_{AH}\in\C^{n\times n}$ is an anti-Hermitian matrix, and $\LL:=-\,C_{H}\in\C^{n\times n}$ is a Hermitian negative semi-definite matrix, where $C_{AH}$ and $C_{H}$ denote the anti-Hermitian and Hermitian parts of $C=C_{AH}+C_{H}$. Several other examples will be reduced to~\eqref{stable-ODE}, mostly via Fourier transformation in~$x$. We shall use the same index notation (`$AH$' and `$H$') for matrix $B$ in Example~\ref{Ex4} and matrix~$C$ in Example~\ref{Ex5}. The hypocoercivity structure of~\eqref{stable-ODE} is discussed in~\cite{AAC20}.

\item \label{Ex2} \emph{Discrete velocity BGK models}, \emph{i.e.}\ transport-relaxation equations (see~\S~2.1 and~\S~4.1 in~\cite{AAC}) can be written in the form of~\eqref{EqnEvol} where $F(t,x)=\big(f_1(t,x),...,f_n(t,x)\big)^\top$, $x\in \mathcal X\subset \R$, $\TT:=V\,\partial_x$ with the diagonal matrix $V\in\R^{n\times n}$ representing the velocities, and the collision operator $\LL:=\sigma\,B$ with $\sigma>0$. Here, the matrix $B\in\R^{n\times n}$ is in \emph{BGK form}
\begin{equation*} \label{matrix:A:n}
B = \begin{pmatrix} b_1 \\ \vdots \\ b_n \end{pmatrix} \otimes (1,\,\ldots,\,1) - \mathrm{Id}
\end{equation*}
with $b=(b_1,\ldots,b_n)^\top \in(0,1)^n$ such that $\sum_{j=1}^n b_j = 1$, and $\mathrm{Id}$ denotes the identity matrix. The collision operator $\LL$ is symmetric on the velocity-weighted $\mathrm L^2$-space $\mathcal{H} = \mathrm L^2(\mathcal X \times \{1,\ldots,n\}; \{b_j^{-1}\})$. Due to this structure,~$B$ has a simple eigenvalue 0 with corresponding left eigenvector $l_1=(1,\ldots,1)$ associated with the mass conservation of the system. The corresponding right eigenvector $b$ spans the local-in-$x$ steady states, which are of the form $\rho(x)\,b$ for some arbitrary scalar function $\rho(x)$. The case with only two velocities, or \emph{Goldstein--Taylor model}, is dealt with in Part~\hyperlink{PartIII}{III}.

\item \label{Ex3} A linear \emph{kinetic BGK model} is analyzed in~\cite{AAC}, where $F=f(t,x,v)\in\R$, $x\in\T$ (the $1$-dimensional torus of length $2\pi$), and $v\in\R$. The kinetic transport operator is $\TT:=v\,\partial_x$, and the BGK operator $\mathsf Lf:=\M_\upvartheta(v)\,\int_\R f\,dv-f$ is symmetric in the weighted space $\mathcal{H}=\mathrm L^2\big(\T\times\R;dx\,dv/(2\pi\,\mathcal M_\upvartheta(v))\big)$, where $\mathcal M_\upvartheta(v)$ denotes the centered Maxwellian with variance (or temperature) $\upvartheta$. The kernel of $\LL$ is spanned by $\mathcal M_\upvartheta(v)$, which is also the global steady state $F_\infty(v)$, due to the setting on the torus.

\item \label{Ex4} The (degenerate) \emph{reaction-diffusion systems} of~\cite{FePrTa17} can also be written as in~\eqref{EqnEvol}, with $F(t,x)=(f_1(t,x),...,f_n(t,x))^\top$, $\TT:=-B_{AH}$, $\LL:=D\,\Delta+B_H$. Here, $0\le D\in\R^{n\times n}$ is a diagonal matrix, $B\in\R^{n\times n}$ is an essentially non-negative matrix, \emph{i.e.}, $b_{ij}\ge0$ for any $i\ne j$, and $b_{ii}=-\sum_{i\ne j}b_{ij}$, and $B_{AH}$ and $B_{H}$ are its anti-symmetric and symmetric parts.

\item \label{Ex5} As a final example, let us mention (possibly degenerate) \emph{Fokker--Planck equations} with linear-in-$x$ drift for $F=f(t,x)$, $x\in\R^d$. After normalization (in the sense of~\cite{ASS20}), they can be identified with~\eqref{EqnEvol}, where $\TT f:=-\,\mbox{div}\big(f_\infty\, C_{AH}\,\nabla(f/f_\infty)\big)$, $\LL f:=\mbox{div}\big(f_\infty\, C_{H}\,\nabla(f/f_\infty)\big)$, with a positive stable drift matrix $C\in\R^{d\times d}$ such that $C_H\ge0$, and $f_\infty=(2\pi)^{-d/2}\exp\big(-|x|^2/2\big)$ is the unique normalized steady state. As shown in~\cite{ASS20}, these Fokker--Planck equations are equivalent to~\eqref{stable-ODE} and tensorized versions of it.

\end{enumerate}

In the articles cited above for Examples~\ref{Ex1}--\ref{Ex3}, an $\mathrm L^2$-based hypocoercive entropy method has been used to derive sharp decay estimates for the solution $F(t)$ towards its steady state $F_\infty$, and the same strategy can also be applied to Example~\ref{Ex4}. In~\cite{Arnold2014} an $\mathrm H^1$-based hypocoercive entropy method was developed for the Fokker--Planck equations in Example~\ref{Ex5}. But in view of its subspace decomposition given in~\cite{ASS20}, an $\mathrm L^2$-analysis is also feasible.

\medskip In our second hypocoercive entropy method, we construct a problem adapted Lypunov functional that is able to reveal the sharp decay behavior as $t\to+\infty$. We shall illustrate this strategy for Examples~\ref{Ex2} and~\ref{Ex3}, where the anti-Hermitian operator~$\TT$ is either $V\partial_x$ (for discrete velocities) or $v\,\partial_x$ (for continuous velocities). In order to establish the \emph{mode-by-mode hypocoercivity}, we Fourier transform~\eqref{EqnEvol} w.r.t.\ $x\in\mathcal X$, with either $\mathcal X=\T^1=:\T$ or $\mathcal X=\R^d$. In the torus case, we assume $d=1$ for simplicity, but the method extends to higher dimensions (see~\cite{Achleitner2018}). With the abuse of notations of keeping $F$ for the distribution function written in the variables $(t,\xi,v)$, this yields
\begin{equation}\label{abstr-eq-FT}
\frac{dF}{dt} = -\,i\,\xi\,VF+\LL F =: -\,C(\xi)\,F\,,
\end{equation}
with a discrete modal variable $\xi\in\Z$ for the torus and $\xi\in\R^d$ in the whole space case. In~\eqref{abstr-eq-FT}, $V$ is a diagonal matrix for Example~\ref{Ex2}, and for Example~\ref{Ex3} it either represents the multiplication operator by $v$ or, when using a basis in the $v$-variable, a symmetric, real-valued ``infinite matrix'' (\emph{cf.}~\cite[\S~4]{AAC}).

For each fixed mode $\xi$,~\eqref{abstr-eq-FT} is now an ODE with constant coefficients (of dimension $n<\infty$ for Example~\ref{Ex2}, and infinite dimensional for Example~\ref{Ex3}). For finite $n$, we define the \emph{modal spectral gap} of $C(\xi)$ as
\begin{equation}\label{def-gap}
\mu(\xi):=\min_{0\neq\lambda_j\in\sigma(C(\xi))} \re(\lambda_j)\,.
\end{equation}
If no eigenvalue of $C(\xi)$ with $\re(\lambda_j)=\mu(\xi)$ is defective (\emph{i.e.}, all eigenvalues have matching algebraic and geometric multiplicities), see \emph{e.g.}~\cite{HJ}, then the exponential decay of $\|F(t,\xi)\|^2$ with the sharp rate $2\,\mu(\xi)$ is shown using a Lyapunov functional obtained as a twisted Euclidean norm on $\C^n$. To this end we use the following algebraic result.
\begin{lemma}[{\cite[Lemma 2]{AAC}}] \label{lemma:Pdefinition}
For a given matrix $C\in\C^{n\times n}$, let $\mu$ be defined as in~\eqref{def-gap}. Assume that $0\not\in\sigma(C)$ and that $C$ has no defective eigenvalues with $\re(\lambda_j)=\mu$. Then there exists a positive definite Hermitian matrix $P\in\C^{n\times n}$ such that
\begin{align} \label{matrixestimate1}
C^*P+P\,C &\geq 2\,\mu\,P\,.
\end{align}
Moreover, if all eigenvalues of $C$ are non-defective, any matrix
\begin{align} \label{simpleP1} 
P:= \sum\limits_{j=1}^n c_j \, w_j \otimes w_j^*
\end{align}
satisfies~\eqref{matrixestimate1}, where $w_j\in\C^{n}$ denote the normalized (right) eigenvectors of $C^*$ and, for all $j=1,\dots,n$, the coefficient $c_j\in(0,+\infty)$ is an arbitrary weight.
\end{lemma}
For the extension of this lemma to the case $0\in\sigma(C)$ we refer to~\cite[Lemma~3]{AAC}, but, anyhow, this is typically relevant only for $\xi=0$. The more technical case when~$C$ has defective eigenvalues was analyzed in~\cite[Lemma 4.3(i)]{Arnold2014}. In the case $n=\infty$ (occuring in the kinetic BGK models of Example~\ref{Ex3}), the eigenfunction construction of the operator (or ``infinite matrix'') $P$ via~\eqref{simpleP1} is, in general, not feasible. A systematic construction of \emph{approximate} matrices $P$ with a suboptimal value compared with $\mu$ in~\eqref{matrixestimate1} was presented in~\cite[\S~4.3-4.4]{AAC} and~\cite[\S~2.3]{Achleitner2018}.

Using the deformation matrix $P$, we define the ``twisted Euclidean norm'' in~$\C^n$~as
$$
\|F\|_P^2 := \langle F,P\,F\rangle\,,
$$
which is equivalent to the Euclidean norm $\|\cdot\|$ through the estimate 
\begin{equation}\label{norm-equiv}
\lambda_1^P\,\|F\|^2 \le \|F\|_P^2 \le \lambda_n^P\,\|F\|^2\,,
\end{equation}
where $\lambda_1^P$ and $\lambda_n^P$ are the smallest and largest eigenvalues of $P$, respectively.
From~\eqref{stable-ODE} and~\eqref{matrixestimate1} follows that
\begin{equation*}
\frac{d}{dt}\|F\|_P^2 = -\,\langle F, (C^*P + PC)F\rangle \leq -\,2\,\mu\,\|F\|_P^2\,.
\end{equation*}
This shows that solutions to~\eqref{stable-ODE} satisfy
\begin{equation*}\label{F-decay}
\|F(t)\|_P^2 \le e^{-2\,\mu\,t}\,\|F_0\|_P^2\quad\forall\,t\ge0\,,
\end{equation*}
and hence in the Euclidean norm:
\begin{equation}\label{eq:eucliddecay}
\|F(t)\|^2 \le \mbox{cond}(P) \,e^{-2\,\mu\,t}\,\|F_0\|^2\quad\forall\,t\ge0\,,
\end{equation}
where $\mbox{cond}(P):=\lambda_n^P / \lambda_1^P$ denotes the \emph{condition number} of $P$. We recall from~\cite{AAS19} that $\mbox{cond}(P)$ is in general not the minimal multiplicative constant for~\eqref{eq:eucliddecay}. In fact, in general it is impossible to obtain that optimal constant from a Lyapunov functional, even for $n=2$, see~\cite[Theorem 4.1]{AAS19}. We also remark that the matrix $P$ from~\eqref{simpleP1} is not uniquely determined (even beyond trivial multiples). As a consequence, $\operatorname{cond}(P)$ may be different for different admissible choices of $P$. For an example with $n=3$, we refer to~\cite[\S~3]{AAS19}.

Analogous decay estimates hold for solutions $F(t,\xi)$ to the modal ODEs~\eqref{abstr-eq-FT}, and they involve the deformation matrices $P(\xi)$ and the modal spectral gaps $\mu(\xi)$:
\begin{equation}\label{eq:modaldecay}
\|F(t,\xi)\|_{P(\xi)}^2 \leq e^{-2\,\mu(\xi)\,t}\,\|F_0(\xi)\|_{P(\xi)}^2\quad\forall\,t\geq 0\,.
\end{equation}
This motivates the definition of a \emph{modal-based Lyapunov functional} by assembling the modal functionals. We present two variants of this approach.

\medskip\noindent\hypertarget{strat1}{\textit{\textbf{Strategy 1.}}} We consider the global Lyapunov functional 
\begin{equation}\label{def-H}
\HH_2[F] := \sum_{\xi\in\Z} \|F(\xi)\|^2_{P(\xi)}\,,
\end{equation}
which is written here for the case of discrete modes, \emph{i.e.}, $\mathcal X=\T$.

We recall that the matrix $P(\xi)$ is not unique. In the kinetic BGK examples studied so far (\emph{cf.}~\cite{AAC, Achleitner2018}) it was convenient to choose $P$ depending continuously on $\xi$ (for $\xi\in \R^d$) and such that $P(\xi)\to \mathrm{Id}$ as $|\xi|\to+\infty$. For kinetic equations with a local-in-$x$ dissipative operator $\LL$, the matrix $C(\xi)$ has the form given in~\eqref{abstr-eq-FT}. Under the assumption of a uniform spectral gap $\overline{\mu}:= \inf_{\xi}\mu (\xi)>0$, the form $P(\xi)= \mathrm{Id} + O\(1/|\xi|\)$ is very natural (see $P^{(1)}(\xi)$ in~\eqref{eq:P12} for an example) in view of the matrix inequality~\eqref{matrixestimate1}.

The modal decay~\eqref{eq:modaldecay} implies the following decay estimate for the solution to~\eqref{EqnEvol}:
\begin{equation*}\label{eq:H2position}
\HH_2[F(t)] \le e^{-2\,\bar\mu\,t}\,\HH_2[F_0]\quad\forall\,t\ge0\,, \quad \text{for any } F_0 \perp F_\infty\,.
\end{equation*}
Using Parseval's identity and the norm equivalence from~\eqref{norm-equiv}, this yields 
\begin{equation}\label{F-decay1}
\|F(t)\|_\h^2 \le \bar c_P\,e^{-2\,\bar\mu\,t}\,\|F_0\|_\h^2\quad\forall\,t\ge0\,,\quad \text{for any } F_0 \perp F_\infty\,,
\end{equation}
where $\bar c_P:=\sup_\xi\mbox{cond}(P(\xi))$.

\medskip\noindent\hypertarget{strat2}{\textit{\textbf{Strategy 2.}}} If all modes $\xi$ have the same spectral gap $\mu(\xi)$, then the estimate~\eqref{F-decay1} clearly yields the minimal multiplicative constant $\bar c_P$ (obtainable by Lyapunov methods). This is the case when the relaxation rate $\sigma<2$ in the Goldstein--Taylor model, which is studied in~\cite{AESW} and in Part~\hyperref[sec:GTgen]{III} below. But faster decaying modes may have a ``too large'' condition number $\mbox{cond}(P(\xi))$, as it is the case for $\sigma>2$ in~\cite{AESW}. Then, the matrices $P(\xi)$ from~\eqref{matrixestimate1} have to be modified in order to reduce $\mbox{cond}(P(\xi))$ by lowering $\mu=\mu(\xi)$ in~\eqref{matrixestimate1}. For simplicity we detail this strategy only for the case that the infimum $\bar\mu$ is actually attained. Main steps are:
\begin{itemize}[leftmargin=0pt]
\item Let $\Xi:=\{\xi\,:\,\mu(\xi)=\bar\mu\}$ be the set of the modes with slowest decay. Set $c_\Xi:=\sup_{\xi\in\Xi}\mbox{cond}(P(\xi))$, \emph{i.e.}\ the worst common multiplicative constant for these slow modes.
\item For all modes $\xi\not\in\Xi$, we distinguish several cases:
\begin{itemize}
\item If $\mbox{cond}(P(\xi))\le c_\Xi$, set $\tilde P(\xi):=P(\xi)$.
\item If $\mbox{cond}(P(\xi))>c_\Xi$, then replace $P(\xi)$ by $\tilde P(\xi)\in\C^{n\times n}$, which is a positive definite Hermitian solution to the matrix inequality 
\begin{equation*}\label{eq:barmu}
C(\xi)^*P+P\,C(\xi)\ge 2\,\bar\mu\,P\,.
\end{equation*}
In particular, $P$ should be either chosen as any such solution that satisfies $\mbox{cond}(P(\xi))\le c_\Xi$ or, if this is impossible, then by a solution $P$ 
having the least condition number.
\end{itemize}
\end{itemize}
\begin{itemize} 
\item Let $\tilde c_\Xi:=\sup_{\xi\not\in\Xi}\mbox{cond}\(\tilde P(\xi)\)$ be the best multiplicative constant for the faster modes.
\item Set $\tilde c_P:=\max\{c_\Xi,\,\tilde c_\Xi\}$. With this construction we define a second, refined Lyapunov functional (again written for the case $\mathcal X=\T$) by
\begin{equation}\label{def-H2}
\widetilde \HH_2[F] := \sum_{\xi\in\Xi} \|F(\xi)\|^2_{P(\xi)} + \sum_{\xi\in\Xi^c} \|F(\xi)\|^2_{\tilde P(\xi)}\,,
\end{equation}
where $\Xi^c := \Z\setminus\Xi $.
\end{itemize}
This yields the improved decay estimate (\emph{w.r.t.}\ the multiplicative constant):
\begin{equation}\label{F-decay2}
\|F(t)\|_\h^2 \le \tilde c_P\,e^{-2\,\bar\mu\,t}\,\|F_0\|_\h^2\quad\forall\,t\ge0\,,\quad \text{for any } F_0 \perp F_\infty\,.
\end{equation}
Note that, by construction, $\tilde c_P \le \bar c_P$. Altogether, our estimates on a solution to the evolution equation~\eqref{EqnEvol} rewritten as~\eqref{abstr-eq-FT} in Fourier variables can be summarized into the following result.
\begin{proposition}\label{prop:F-decay} On $\T$, let us consider an operator $C$ such that, in Fourier variables, $C(\xi)$ takes values in $\C^{n\times n}$ for any $\xi\in\Z$. Assume the existence of a uniform spectral gap $\overline{\mu}:= \inf_{\xi\in\Z}\mu(\xi)>0$ where $\mu(\xi)$ is defined by~\eqref{def-gap}.
\begin{enumerate}
\item[a)] If the corresponding modal deformation matrices $P(\xi)$ satisfy $\bar c_P<\infty$, then the solutions of~\eqref{EqnEvol} satisfy the decay estimate~\eqref{F-decay1}.
\item[b)] If the modified deformation matrices $\tilde P(\xi)$ satisfy $\tilde c_P<\infty$, then the solutions of~\eqref{EqnEvol} satisfy the decay estimate~\eqref{F-decay2}.
\end{enumerate}
\end{proposition}
The above procedure was applied in~\cite{AESW} to the Goldstein--Taylor model, and in~\cite{AAC} to Examples~\ref{Ex2}--\ref{Ex3}, considered on $\T$.

\medskip The hypocoercivity results based on the Lyapunov matrix inequalities~\eqref{matrixestimate1} and mode-by-mode estimates as in~\eqref{eq:modaldecay} have the advantage that, in simple cases, it is possible to identify the optimal decay rates. They are less flexible than the hypocoercivity results based on the twisted $\mathrm L^2$ norm inspired by diffusion limits of~\S~\ref{Sec:Abstract1}. Our purpose of Part~\hyperlink{PartIII}{III} is to detail several variants of these methods in simple cases, draw a few consequences and compare the estimates of the two methods.

\part{Optimization of twisted $\mathrm L^2$ norms}{Optimization of twisted \texorpdfstring{$\mathrm L^2$}{L2} norms\hypertarget{PartII}{}}

This part is devoted to accurate hypocoercivity estimates in Fourier variables based on our \emph{first abstract method}, for two simple kinetic equations with Gaussian local equilibria. It is a refined version of the paper~\cite{BDMMS} devoted to a larger class of equilibria, but to the price of weaker bounds. Here we underline some key ideas of mode-by-mode hypocoercivity and perform more accurate and explicit computations. New estimates are obtained, which numerically improve upon known ones. Rates and constants are discussed and numerically illustrated, with the purpose of establishing benchmarks for the \emph{$\mathrm L^2$-hypocoercivity theory based upon a twist inspired by diffusion limits}. Exponential rates are obtained on the torus, with a discussion on high frequency estimates. On the whole space case, low frequencies are involved in the computation of the asymptotic decay rates. We also detail how spectral estimates of the mode-by-mode $\mathrm L^2$ hypocoercivity method can be systematically turned into rates of decay using the ideas of the original proof of Nash's inequality.

\section{A detailed mode-by-mode approach}\label{part:II}

\subsection{Introduction}\label{Sec:Intro}

We consider the Cauchy problem
\be{eq:model}
\partial_tf+v\cdot\nabla_xf=\mathsf Lf\,,\quad f(0,x,v)=f_0(x,v)\,,
\ee
for a distribution function $f(t,x,v)$, where $x\in\R^d$ denotes the position variable, $v\in\R^d$ is the velocity variable, and $t\ge 0$ is the time. Concerning the collision operator, $\mathsf L$ denotes the \emph{Fokker--Planck operator} $\mathsf L_1$ or, as in~\cite{Dolbeault2009511}, the \emph{linear BGK operator} $\mathsf L_2$, which are defined respectively by
\[
\mathsf L_1 f:=\Delta_v f+\nabla_v\cdot(v\,f)\quad\mbox{and}\quad\mathsf L_2 f:=\rho_f\,\M-f\,.
\]
Here $\M$ is the normalized Gaussian function
\[
\M(v)=\frac{e^{-\,\frac12\,|v|^2}}{(2\,\pi)^{d/2}}\quad\forall\,v\in\R^d
\]
and $\rho_f:=\int_{\R^d}f\,dv$ is the spatial density. Notice that $\M$ spans the kernel of $\mathsf L$. We introduce the \emph{weight}
\[
d\gamma:=\gamma(v)\,dv\quad\mbox{where}\quad\gamma:=\frac 1{\M}
\]
and the weighted norm
\[
\|f\|_{\mathrm L^2(dx\,d\gamma)}^2:=\iint_{\mathcal X\times\R^d}|f(x,v)|^2\,dx\,d\gamma\,,
\]
where $\mathcal X$ denotes either the cube $[0,L)^d$ with periodic boundary conditions or \hbox{$\mathcal X=\R^d$}, that is, the whole Euclidean space. 

Let us consider the Fourier transform of $f$ in $x$ defined by
\be{Fourier}
\hat f(t,\xi,v)=\int_{\mathcal X}e^{-i\,x\cdot\xi}\,f(t,x,v)\,dx\,,
\ee
where either $\mathcal X=[0,L)^d$ (with periodic boundary conditions), or $\mathcal X=\R^d$. We denote by $\xi\in(2\pi/L)^d\,\Z^d\subset\R^d$ or $\xi\in\R^d$ the Fourier variable. Details will be given in~\S~\ref{Sec:Mode-by-mode}. Next, we rewrite Equation~\eqref{eq:model} for $F=\hat f$ as
\be{eq:hat}
\partial_tF+\mathsf TF=\mathsf LF\,,\quad F(0,\xi,v)=\hat{f_0}(\xi,v)\,,\quad\mathsf TF=i\,(v\cdot\xi)F\,.
\ee
Here we abusively use the same notation $\mathsf T$ for the transport operator in the original variables and after the Fourier transform, where it is a simple multiplication operator. We shall also consider $\xi$ as a given, fixed parameter and omit it whenever possible, so that we shall write that $F$ is a function of $(t,v)$, for sake of simplicity. Let us define
\be{HnrmPi}
\mathcal H=\mathrm L^2\(d\gamma\)\,,\quad\|F\|^2=\int_{\R^d}|F|^2\,d\gamma\,,\quad\Pi F=\M\,\int_{\R^d}F\,dv=\M\,\rho_F\,.
\ee
Our goal is to obtain decay estimates of $\|F\|$ parameterized by $\xi$ and this is why such an approach can be qualified as a \emph{mode-by-mode hypocoercivity method.}

\subsection{A first optimization in the general setting}\label{Sec:GeneralSetting}

The estimates of~\cite[Proposition~4]{BDMMS} are rough and it is possible to improve upon the choice for $\delta$ and $\lambda$. On the triangle
\[
\mathcal T_m:=\Big\{(\delta,\lambda)\in(0,\lambda_m)\times(0,2\,\lambda_m)\,:\,\lambda<2\,(\lambda_m-\delta)\Big\}\,,
\]
let us define
\[
\hh(\delta,\lambda):=\delta^2\(C_M+\frac\lambda2\)^2-4\(\lambda_m-\,\delta-\frac\lambda2\)\(\frac{\delta\,\lambda_M}{1+\lambda_M}-\frac\lambda2\)\,,
\]
\[
\lambda_\star(\delta):=\sup\Big\{\lambda\in(0,2\,\lambda_m)\,:\,\hh(\delta,\lambda)\le0\Big\}\quad\mbox{and}\quad C_\star(\delta):=\frac{2+\delta}{2-\delta}\,.
\]
We will also need later
\[
K_M:=\frac{\lambda_M}{1+\lambda_M}<1\quad\mbox{and}\quad\delta_\star:=\frac{4\,K_M\,\lambda_m}{4\,K_M+C_M^2}<\lambda_m\,.
\]
Our first result provides us with the following refinement of~\eqref{Decay:BDMMS}.
\begin{proposition}\label{prop:DMS2015} Under the assumptions of Theorem~\ref{theo:DMS2015}, we have
\[
\mathsf H_1[F(t,\cdot)]\le\mathsf H_1[F_0]\,e^{-\lambda\,t}\quad\forall\,t\ge0
\]
with $\lambda=\max\Big\{\lambda_\star(\delta)\,:\,\delta\in(0,\delta_\star)\Big\}$. Moreover, for any $\delta<\min\{2,\delta_\star\}$, if $F$ solves~\eqref{EqnEvol} with initial datum $F_0\in\mathcal H$, then
\[
\|F(t)\|^2\le C_\star(\delta)\,e^{-\,\lambda_\star(\delta)\,t}\,\|F_0\|^2\quad\forall\,t\ge0\,.
\]\end{proposition}
On the boundary of the triangle $\mathcal T_m$, we notice that
\[
\hh(0,\lambda)=\lambda\,(2\,\lambda_m-\lambda)>0\quad\forall\,\lambda\in(0,2\,\lambda_m)\,,
\]
\[
\hh\Big(\delta,2\,(\lambda_m-\delta)\Big)=(C_M+\lambda_m-\delta)^2\,\delta^2>0\quad\forall\,\delta\in(0,\lambda_m)\,,
\]
and $\hh(\delta,0)/\delta=\(C_M^2+4\,K_M\)\,\delta-4\,K_M\,\lambda_m$ is negative if $0<\delta<\delta_\star$. As a consequence, the set $\{(\delta,\lambda)\in\mathcal T_m\,:\,\hh(\delta,\lambda)\le0\}$ is non-empty. The functions $\lambda\mapsto \hh(\delta,\lambda)$ for a fixed $\delta\in(0,\lambda_m)$ and $\delta\mapsto \hh(\delta,\lambda)$ for a fixed $\lambda\in(0,2\,\lambda_m)$ are both polynomials of second degree. The expression of $\lambda_\star(\delta)$ is explicitly computed as the smallest root of $\lambda\mapsto \hh(\delta,\lambda)$ but has no interest by itself. It is also elementary to check that $\hh$ is positive if $(\delta,\lambda)\in\mathcal T_m$ with $\delta>\delta_\star$.
\begin{proof} The method is the same as in~\cite{DMS-2part} and~\cite[Proposition~4]{BDMMS}, except that we use sharper estimates.

Since $\mathsf{AT}\Pi$ can be interpreted as $z\mapsto(1+z)^{-1}\,z$ applied to $(\mathsf T\Pi)^*\mathsf T\Pi$, the spectral theorem and conditions~\eqref{H1} and~\eqref{H2} imply that
\be{Est1}
-\,\langle\mathsf LF,F\rangle+\delta\,\langle\mathsf{AT}\Pi F,F\rangle\ge\lambda_m\,\|(\mathrm{Id}-\Pi)F\|^2+\frac{\delta\,\lambda_M}{1+\lambda_M}\,\|\Pi F\|^2\,.
\ee
{}From that point, one has to prove that $-\,\langle\mathsf LF,F\rangle+\delta\,\langle\mathsf{AT}\Pi F,F\rangle$ controls the other terms in the expression of $\mathsf D[F]$. By~\eqref{H4}, we know that
\be{Est2}
\left|\re\langle\mathsf{AT}(\mathrm{Id}-\Pi)F,F\rangle+\,\re\langle\mathsf{AL}F,F\rangle\right|\le C_M\,\|\Pi F\|\,\|(\mathrm{Id}-\Pi)F\|\,.
\ee
As in~\cite[Lemma~1]{DMS-2part}, if $G=\mathsf AF$, \emph{i.e.}, if $(\mathsf T\Pi)^*F=G+(\mathsf T\Pi)^*\,\mathsf T\Pi\,G$, then
\[
\langle\mathsf{TA}F,F\rangle=\langle G,(\mathsf T\Pi)^*\,F\rangle=\|G\|^2+\|\mathsf T\Pi G\|^2=\|\mathsf AF\|^2+\|\mathsf{TA}F\|^2\,.
\]
By the Cauchy-Schwarz inequality, we know that
\begin{multline*}
\langle G,(\mathsf T\Pi)^*\,F\rangle=\langle\mathsf{TA}F,(\mathrm{Id}-\Pi)F\rangle\\
\le\|\mathsf{TA}F\|\,\|(\mathrm{Id}-\Pi)F\|\le\frac1{2\,\mu}\,\|\mathsf{TA}F\|^2+\frac\mu2\,\|(\mathrm{Id}-\Pi)F\|^2
\end{multline*}
for any $\mu>0$. Hence
\[\label{A-bound}
2\,\|\mathsf AF\|^2+\(2-\frac1\mu\)\|\mathsf{TA}F\|^2\le\mu\,\|(\mathrm{Id}-\Pi)F\|^2\,,
\]
which, by taking either $\mu=1/2$ or $\mu=1$, proves that
\[
\|\mathsf AF\|\le\frac12\,\|(\mathrm{Id}-\Pi)F\|\,,\quad\|\mathsf{TA}F\|\le\|(\mathrm{Id}-\Pi)F\|
\]
and establishes~\eqref{H-norm}. Incidentally, this proves that
\be{Est3}
\left|\langle\mathsf{TA}F,F\rangle\right|=\left|\langle\mathsf{TA}F,(\mathrm{Id}-\Pi)F\rangle\right|\le\|(\mathrm{Id}-\Pi)F\|^2\,,
\ee
and also that
\be{Est4}
|\langle\mathsf AF,F\rangle|\le\frac12\,\|\Pi F\|\,\|(\mathrm{Id}-\Pi)F\|\le\frac14\,\|F\|^2\,.
\ee
As a consequence of this last identity, we obtain
\[
\left|\mathsf H_1[F]-\tfrac12\,\|F\|^2\right|=\delta\,\big|\langle\mathsf AF,F\rangle\big|\le\frac\delta4\,\|F\|^2\,,
\]
which, under the condition $\delta<2$, is a proof of~\eqref{H-norm} with the improved constant
\be{cpm}
c_\pm=\frac{2\pm\,\delta}4\,.
\ee

Now let us come back to the proof of~\eqref{Decay:H}. Collecting~\eqref{Est1},~\eqref{Est2}, and~\eqref{Est3} with the definition of $\mathsf D[F]$, we find that
\[
\mathsf D[F]\ge(\lambda_m-\,\delta)\,X^2+\frac{\delta\,\lambda_M}{1+\lambda_M}\,Y^2-\,\delta\,C_M\,X\,Y
\]
with $X:=\|(\mathrm{Id}-\Pi)F\|$ and $Y:=\|\Pi F\|$. Using~\eqref{Est4}, we observe that
\[
\mathsf H_1[F]\le\frac12\(X^2+Y^2\)+\frac\delta2\,X\,Y\,.
\]
Hence the largest value of $\lambda$ for which
\[
\mathsf D[F]\ge\lambda\,\mathsf H_1[F]
\]
can be estimated by the largest value of $\lambda$ for which
\begin{multline*}
\mathcal Q(X,Y):=(\lambda_m-\,\delta)\,X^2+\frac{\delta\,\lambda_M}{1+\lambda_M}\,Y^2-\,\delta\,C_M\,X\,Y-\frac\lambda2\(X^2+Y^2\)-\frac\lambda2\,\delta\,X\,Y\\
=\(\lambda_m-\,\delta-\frac\lambda2\)X^2-\,\delta\(C_M+\frac\lambda2\)X\,Y+\(\frac{\delta\,\lambda_M}{1+\lambda_M}-\frac\lambda2\)Y^2
\end{multline*}
is a nonnegative quadratic form. It is characterized by the discriminant condition $\hh(\delta,\lambda)\le0$, and the condition $\lambda_m-\,\delta-\lambda/2>0$ which determines $\mathcal T_m$ with the two other conditions: $\delta>0$ and $\lambda>0$. From~\eqref{Decay:H}, we deduce the decay of $\mathsf H_1[F(t,\cdot)]$ and the decay of $\|F(t)\|^2$ by~\eqref{H-norm} using~\eqref{cpm}.
\qed\end{proof}

\begin{remark}
The estimate~\eqref{BDMScomp} of~\cite[Proposition~4]{BDMMS} is easily recovered as follows. Using
\begin{multline*}
\mathsf{D}[F]\ge(\lambda_m-\delta)\,X^2+\frac{\delta\,\lambda_M}{1+\lambda_M}\,Y^2-\delta\,C_M\,X\,Y\\
\ge(\lambda_m-\delta)\,X^2+\frac{\delta\,\lambda_M}{1+\lambda_M}\,Y^2-\frac\delta2\(C_M^2\,X^2+Y^2\)
\end{multline*}
and
\[
\mathsf H_1[F]\le\frac{2+\delta}4\(X^2+Y^2\)\,,
\]
with $\delta$ defined as in~\eqref{BDMScomp}, we obtain
\begin{multline*}
\mathsf{D}[F]\ge\frac{\lambda_m}4\,X^2+\frac{\delta\,\lambda_M}{2\,(1+\lambda_M)}\,Y^2\\
\ge\frac14\,\min\left\{\lambda_m,\frac{2\,\delta\,\lambda_M}{1+\lambda_M}\right\}\|F\|^2\ge\frac{2\,\delta\,\lambda_M}{3\,(1+\lambda_M)}\,\mathsf{H}[F]\,.
\end{multline*}
Hence we have that $\frac14\,\|F\|^2\ge\frac13\,\mathsf{H}[F]$ because $4/(2+\delta)\ge8/5>4/3$ if \hbox{$\delta<1/2$}. This estimate is non-optimal and it is improved in the proof of Proposition~\ref{prop:DMS2015}.
\end{remark}
\begin{remark} In the discussion of the positivity of $\mathcal Q$, we can observe that $(X,Y)$ is restricted to the upper right quadrant corresponding to $X>0$ and $Y>0$. The discriminant condition $\hh(\delta,\lambda)\le0$ and the condition $\lambda_m-\,\delta-\lambda/2>0$ guarantee that $\mathcal Q(X,Y)\ge0$ \emph{for any} $X$, $Y\in\R$, which is of course a sufficient condition. It is also necessary because the coefficient of $Y^2$ is positive (otherwise one can find some $X>0$ and $Y>0$ such that $\mathcal Q(X,Y)<0$) and then by solving a second degree equation, one could again find a region in the upper right quadrant such that $\mathcal Q$ takes negative values.

Hence we produce a necessary and sufficient condition for $\mathcal Q$ to be a nonnegative quadratic form. This does not mean that the condition of Proposition~\ref{prop:DMS2015} is necessary because we have made various estimates, which are not generically optimal, in order to reduce the problem to the discussion of the sign of $\mathcal Q$. In special cases, we can indeed improve upon Proposition~\ref{prop:DMS2015}. We will discuss such improvements in the next section.\end{remark}

\subsection{Mode-by-mode hypocoercivity}\label{Sec:Mode-by-mode}

\subsubsection{Fourier representation and mode-by-mode estimates}\label{Sec:Fourier}

Let us consider the Fourier transform in $x$, take the Fourier variable $\xi\in\R^d$ as a parameter, and study, \emph{for a given $\xi$}, Equation~\eqref{eq:hat}. For a given $\xi\in\R^d$, let us implement the strategy of Theorem~\ref{theo:DMS2015} and Proposition~\ref{prop:DMS2015} applied to $(t,v)\mapsto F(t,\xi,v)$, with the choices~\eqref{HnrmPi}. The operator $\mathsf A$ is defined by
\[
(\mathsf AF)(v)=-\frac{\,i\,\xi}{1+|\xi|^2}\cdot\int_{\R^d}w\,F(w)\,dw\,\M(v)\,.
\]
Taking advantage of the explicit form of $\mathsf A$, we can reapply the method of~\S~\ref{Sec:GeneralSetting} with explicit numerical values, and actually improve upon the previous results. Let us give some details, which will be useful for benchmarks and numerical computations. Again we aim at relating the Lyapunov functional
\[
\mathsf H_1[F]:=\frac12\,\|F\|^2+\delta\,\re\langle\mathsf AF,F\rangle
\]
defined as in~\eqref{H} with $\mathsf D[F]$ defined by~\eqref{D}, \emph{i.e.},
\begin{multline*}
\mathsf D[F]:=-\,\langle\mathsf LF,F\rangle+\delta\,\langle\mathsf{AT}\Pi F,F\rangle\\
-\,\delta\,\re\langle\mathsf{TA}F,F\rangle+\delta\,\re\langle\mathsf{AT}(\mathrm{Id}-\Pi)F,F\rangle-\delta\,\re\langle\mathsf{AL}F,F\rangle\,.
\end{multline*}
In other words, we want to estimate the optimal constant $\lambda(\xi)$ in the \emph{entropy -- entropy production inequality}
\be{Opt}
\mathsf D[F]\ge\lambda(\xi)\,\mathsf H_1[F]
\ee
corresponding to the best possible choice of $\delta$, for a given $\xi\in\R^d$.

If $\mathsf L=\mathsf L_1$, $\lambda_m=1$ is given by the Gaussian Poincar\'e inequality. If $\mathsf L=\mathsf L_2$, it is straightforward to check that $\lambda_m=1$. In both cases, it follows from the definition of~$\mathsf T$ that $\lambda_M=|\xi|^2$. With $X:=\|(\mathrm{Id}-\Pi)F\|$ and $Y:=\|\Pi F\|$, using~\eqref{Est1} we have
\be{Est1F}
-\,\langle\mathsf LF,F\rangle+\delta\,\langle\mathsf{AT}\Pi F,F\rangle\ge X^2+\frac{\delta\,|\xi|^2}{1+|\xi|^2}\,Y^2\,.
\ee
By a Cauchy-Schwarz estimate, we know that
\[
\left|\xi\cdot\int_{\R^d}w\,F(w)\,dw\right|=|\xi|\,\left|\int_{\R^d}\frac\xi{|\xi|}\cdot w\,\sqrt{\M}\;\frac{(\mathrm{Id}-\Pi)F}{\sqrt{\M}}\,dw\right|\le|\xi|\,\|(\mathrm{Id}-\Pi)F\|
\]
and therefore obtain that
\be{AF-ALF}
\|\mathsf AF\|\le\frac{|\xi|}{1+|\xi|^2}\,\|(\mathrm{Id}-\Pi)F\|\quad\mbox{and}\quad\|\mathsf{AL}F\|\le\frac{|\xi|}{1+|\xi|^2}\,\|(\mathrm{Id}-\Pi)F\|\,,
\ee
where the second estimate is a consequence of $\mathsf{AL}F=-\,\mathsf AF$ when $\mathsf L=\mathsf L_1$ or $\mathsf L=\mathsf L_2$. Notice that the estimate of $\|\mathsf AF\|$ is sharper than the one used in the introduction. 

Using~\eqref{AF-ALF}, we have that
\be{Equivxi}
|\re\langle\mathsf AF,F\rangle|\le\frac{|\xi|}{1+|\xi|^2}\,\|\Pi F\|\,\|(\mathrm{Id}-\Pi)F\|\le\frac12\,\frac{|\xi|}{1+|\xi|^2}\,\|F\|^2
\ee
and obtain an improved version of~\eqref{H-norm} given by
\be{H-norm-xi}
\frac12\(1-\frac{\delta\,|\xi|}{1+|\xi|^2}\)\|F\|^2\le\mathsf H_1[F]\le\frac12\(1+\frac{\delta\,|\xi|}{1+|\xi|^2}\)\|F\|^2\,.
\ee
We also deduce from~\eqref{AF-ALF} that
\be{Est0F}
\mathsf H_1[F]\le\frac12\(X^2+Y^2\)+\frac{\delta\,|\xi|}{1+|\xi|^2}\,X\,Y
\ee
and, using $\mathsf{AL}F=-\,\mathsf AF$ and~\eqref{Equivxi},
\be{IntermF}
|\re\langle\mathsf{AL}F,F\rangle|\le\frac{|\xi|}{1+|\xi|^2}\,\|\Pi F\|\,\|(\mathrm{Id}-\Pi)F\|\,.
\ee
As for estimating $\|\mathsf AF\|$, by a Cauchy-Schwarz estimate we obtain
\[
\|\mathsf{TA}F\|\le\frac{|\xi|^2}{1+|\xi|^2}\,\|(\mathrm{Id}-\Pi)F\|\,,
\]
so that
\be{Est3F}
\delta\,|\re\langle\mathsf{TA}F,F\rangle|\le\frac{\delta\,|\xi|^2}{1+|\xi|^2}\,X^2\,.
\ee
As in~\cite{BDMMS}, we can also estimate
\begin{align*}
\|\mathsf{AT}(\mathrm{Id}-\Pi)F\|&=\tfrac{\left|\int_{\R^d}\(v'\cdot\xi\)^2\,(\mathrm{Id}-\Pi)F(v')\,dv'\right|}{1+|\xi|^2}\\
&\le\tfrac{\(\int_{\R^d}\(v'\cdot\xi\)^4\M(v')\,dv'\)^{1/2}}{1+|\xi|^2}\,\|(\mathrm{Id}-\Pi)F\|=\frac{\sqrt3\,|\xi|^2}{1+|\xi|^2}\,\|(\mathrm{Id}-\Pi)F\|\,.
\end{align*}
This inequality and~\eqref{AF-ALF} establish that~\eqref{H4} holds with $C_M=\frac{|\xi|\,(1+\sqrt3\,|\xi|)}{1+|\xi|^2}$. Let us finally notice that
\be{H4F}
\delta\,|\re\langle\mathsf{AT}(\mathrm{Id}-\Pi)F,F\rangle|+\delta\,|\re\langle\mathsf{AL}F,F\rangle|\le\delta\,\frac{|\xi|\(1+\sqrt3\,|\xi|\)}{1+|\xi|^2}\,X\,Y\,.
\ee

\subsubsection{Improved estimates with some plots}\label{Sec:Plots}

In this section, our purpose is to provide constructive estimates of the rate $\lambda$ in Theorem~\ref{theo:DMS2015} and get improved estimates using various refinements in the mode-by-mode approach. Let us start with the one given in~\eqref{BDMScomp}.

\medskip With $s:=|\xi|$, we read from~\S~\ref{Sec:Fourier} that
\be{Ex}
\lambda_m=1\,,\quad\lambda_M=s^2\quad\mbox{and}\quad C_M=\frac{s\(1+\sqrt3\,s\)}{1+s^2}\,.
\ee
In that case, the estimate~\eqref{BDMScomp} becomes $\lambda\ge\lambda_0(s)$ for $\delta=\delta_0(s)$ with
\[
\lambda_0(s):=\frac13\,\frac{s^2}{\(1+\sqrt3\,s\)^2}\quad\mbox{and}\quad\delta_0(s):=\frac12\,\frac{1+s^2}{\(1+\sqrt3\,s\)^2}\,.
\]

With~\eqref{Ex} in hand, we can also apply the result of Proposition~\ref{prop:DMS2015}. In order to take into account the dependence on $s$, the function $\hh$ has to be replaced by a function~$h_1$ defined by
\[
h_1(\delta,\lambda,s):=\delta^2\(\frac{s\(1+\sqrt3\,s\)}{1+s^2}+\frac\lambda2\)^2-4\(1-\,\delta-\frac\lambda2\)\(\frac{\delta\,s^2}{1+s^2}-\frac\lambda2\)\,,
\]
so that the whole game is now reduced, for a given value of $s>0$, to study the conditions on $(\delta,\lambda)\in\mathcal T_m$ such that $h_1(\delta,\lambda,s)\le0$. In particular, we are interested in computing the largest value $\lambda_1(s)$ of $\lambda$ for which there exists $\delta>0$ for which $h_1(\delta,\lambda,s)\le0$ with $(\delta,\lambda)\in\mathcal T_m$, and denote it by $\delta_1(s)$. The triangle $\mathcal T_m$ is shown in Fig.~\ref{Fig:Triangle} and the curves $s\mapsto\lambda_1(s)$ and $s\mapsto\delta_1(s)$ in Figs.~\ref{Fig:F1} and~\ref{fig:lambda}. Solutions are numerically contained in $\mathcal T_m$ in the sense that $s\mapsto\big(\delta_1(s),\lambda_1(s)\big)\in\mathcal T_m$ for any $s>0$.
\setlength\unitlength{1cm}
\begin{figure}[ht]
\begin{center}\hspace*{-10pt}\begin{picture}(12,9)
\put(2,0){\includegraphics[width=8cm]{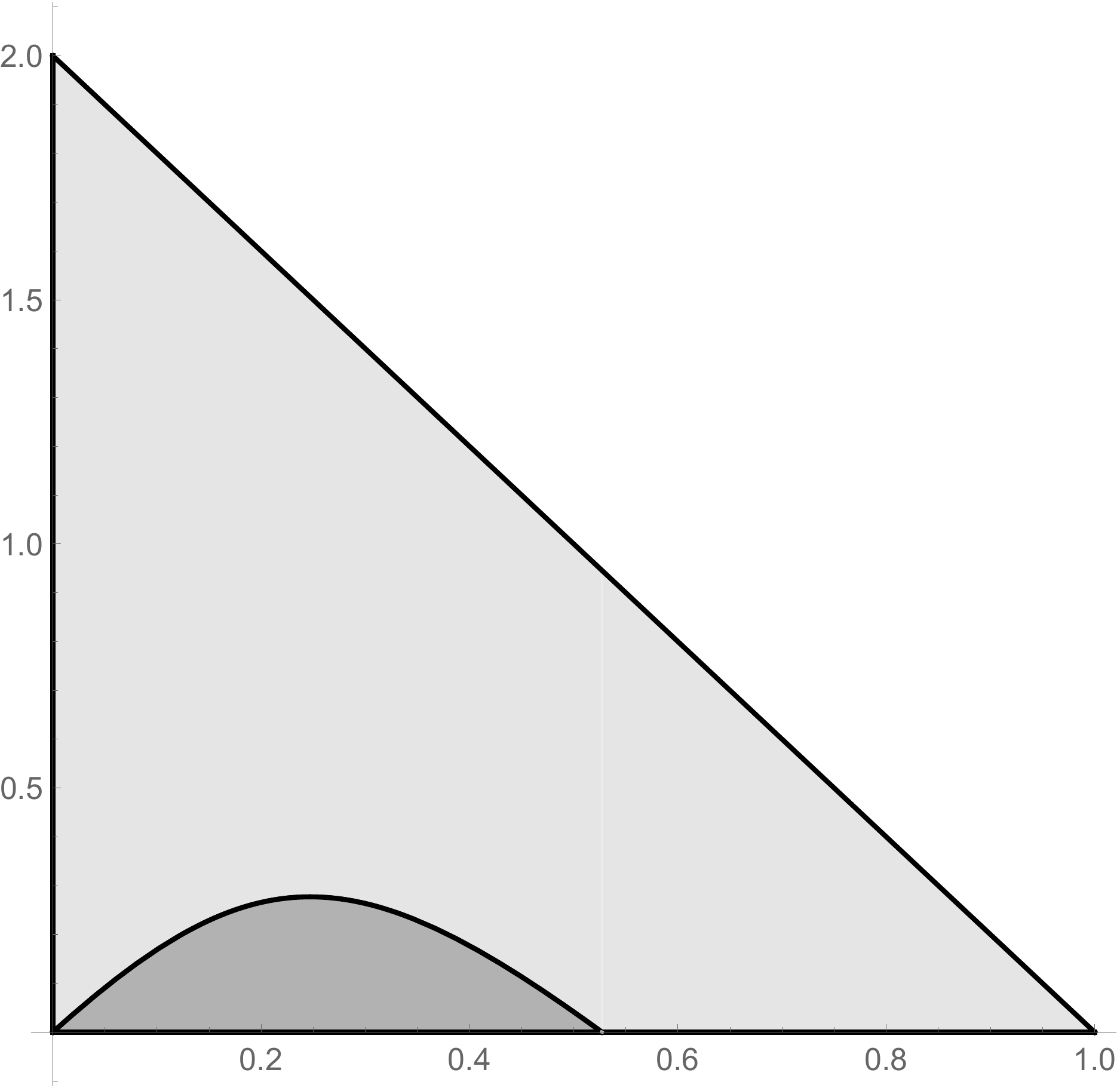}}
\put(10.9,0.35){$\delta$}
\put(2.75,7.75){$\lambda$}
\put(8,6){$s=5$}
\put(3.4,0.6){\small$h_1(\delta_1,\lambda_1)<0$}
\put(3.2,1.6){\small$(\delta_1(s),\lambda_1(s))$}
\thicklines
\put(2,0.415){\vector(1,0){8.75}}
\put(2.375,0){\vector(0,1){8}}
\put(4.23,1.28){\line(0,1){0.2}}
\end{picture}
\caption{\label{Fig:Triangle} With $\lambda_m$, $\lambda_M$ and $C_M$ given by~\eqref{Ex}, the admissible range $\mathcal T_m$ of the parameters $(\delta,\lambda)$ is shown in grey for $s=5$. The darker area is the region in which $h_1(\delta,\lambda,s)$ takes negative values, and $(\delta_1(s),\lambda_1(s))$ are the coordinates of the maximum point of the curve which separates the two regions in the triangle $\mathcal T_m$.}
\end{center}\vspace*{-10pt}
\end{figure}
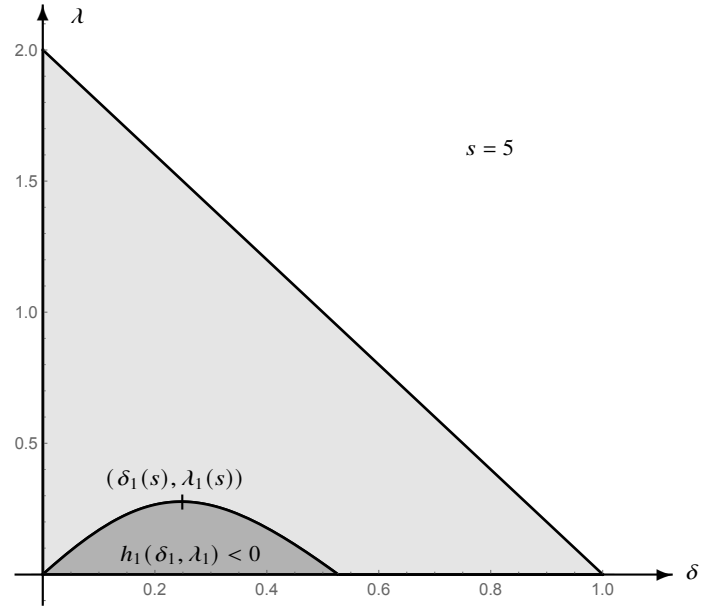
As already noted, some estimates in~\S~\ref{Sec:Fourier} (namely~\eqref{AF-ALF},~\eqref{H-norm-xi},~\eqref{Est0F},~\eqref{IntermF} and~\eqref{Est3F}) are slightly more accurate then the estimates of the proof of Proposition~\ref{prop:DMS2015}. By collecting~\eqref{Est1F},~\eqref{Est0F},~\eqref{Est3F} and~\eqref{H4F}, we obtain
\begin{multline*}
\mathsf D[F]-\lambda\,\mathsf H_1[F]\\
\ge\(1-\frac{\delta\,s^2}{1+s^2}-\frac\lambda2\)X^2-\frac{\delta\,s}{1+s^2}\(1+\sqrt3\,s+\lambda\)X\,Y+\(\frac{\delta\,s^2}{1+s^2}-\frac\lambda2\)Y^2\numberthis \label{eq:Q2}
\end{multline*}
is nonnegative for any $X$ and $Y$ under the discriminant condition which amounts to the nonpositivity of
\begin{equation*}\label{eq:h2}
h_2(\delta,\lambda,s):=\delta^2\,s^2\(\frac{1+\sqrt3\,s+\lambda}{1+s^2}\)^2-4\(1-\frac{\delta\,s^2}{1+s^2}-\frac\lambda2\)\(\frac{\delta\,s^2}{1+s^2}-\frac\lambda2\)\,,
\end{equation*}
in the triangle
\[
\mathcal T_m(s):=\Big\{(\delta,\lambda)\in\(0,\lambda_m\,\tfrac{1+s^2}{s^2}\)\times(0,2\,\lambda_m)\,:\,\lambda<2\(\lambda_m-\tfrac{\delta\,s^2}{1+s^2}\)\Big\}
\]
with $\lambda_m=1$. Exactly the same discussion as for $s\mapsto\lambda_1(s)$ and $s\mapsto\delta_1(s)$ determines the curves $s\mapsto\lambda_2(s)$ and $s\mapsto\delta_2(s)$ shown in Figs.~\ref{Fig:F1} and~\ref{fig:lambda}. Solutions satisfy $s\mapsto\big(\delta_2(s),\lambda_2(s)\big)\in\mathcal T_m(s)$ for any $s>0$.
\begin{figure}[ht]
 \vspace*{-20pt}
\begin{center}\hspace*{-10pt}\begin{picture}(12,9)
\put(2,0){\includegraphics[width=8cm]{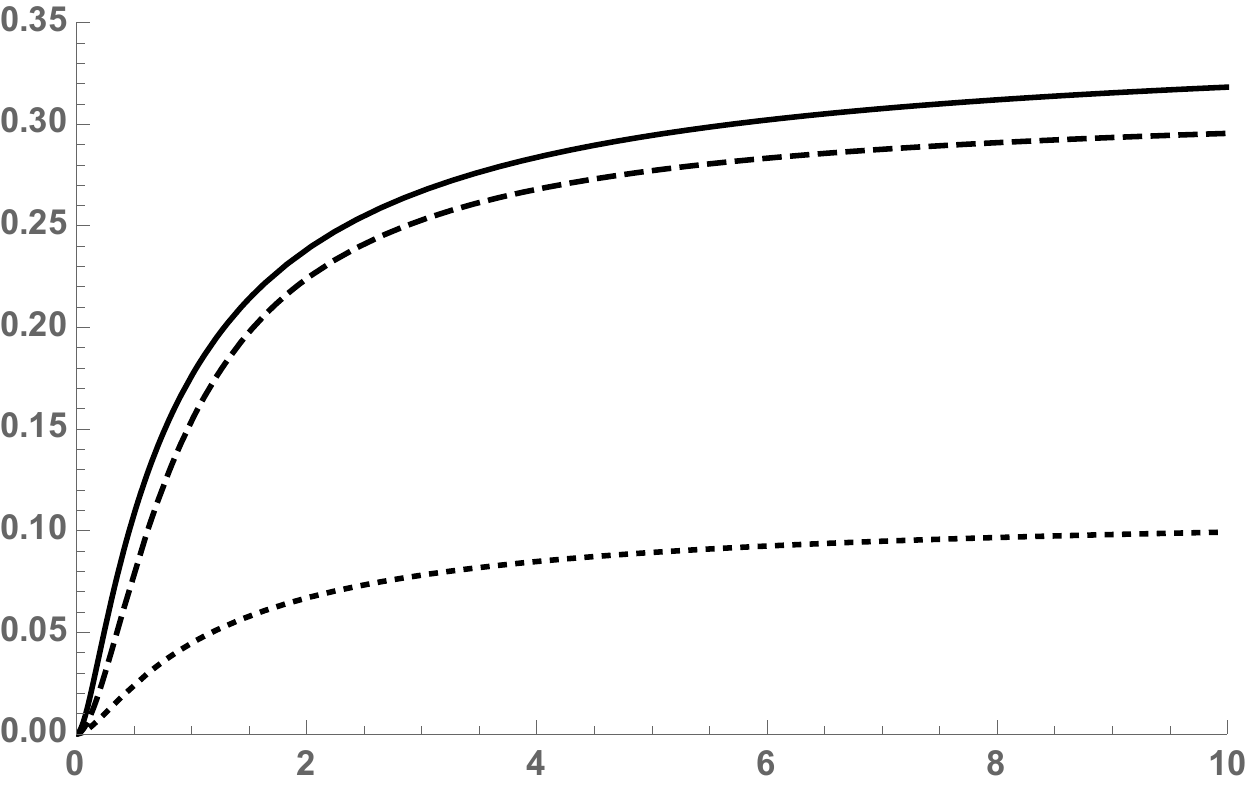}}
\put(11.2,0.7){$s$}
\put(7,1.8){$\lambda_0(s)$}
\put(5.5,3.5){$\lambda_1(s)$}
\put(4.25,4.12){$\lambda_2(s)$}
\thicklines
\put(2.48,0.38){\vector(1,0){8.75}}
\put(2.5,0.3){\vector(0,1){5.35}}
\end{picture}
\caption{\label{Fig:F1} With $\lambda_m$, $\lambda_M$ and $C_M$ given by~\eqref{Ex}, curves $s\mapsto\lambda_i(s)$ with $i=0$, $1$ and $2$ are shown. The improvement of $\lambda_2$ upon $\lambda_0$ is of the order of a factor $5$.}
\end{center}\vspace*{-40pt}
\end{figure}
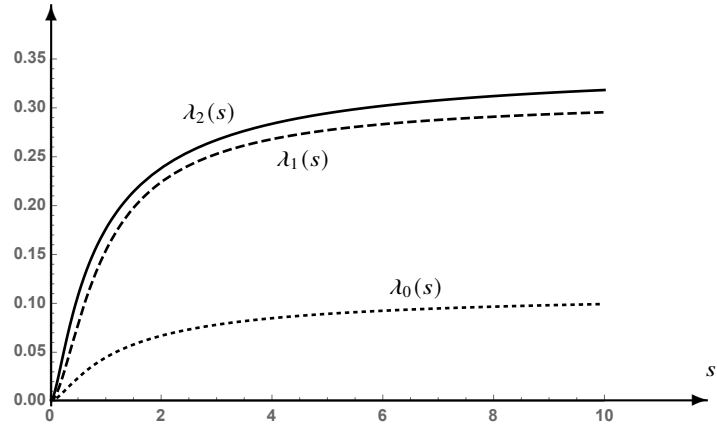
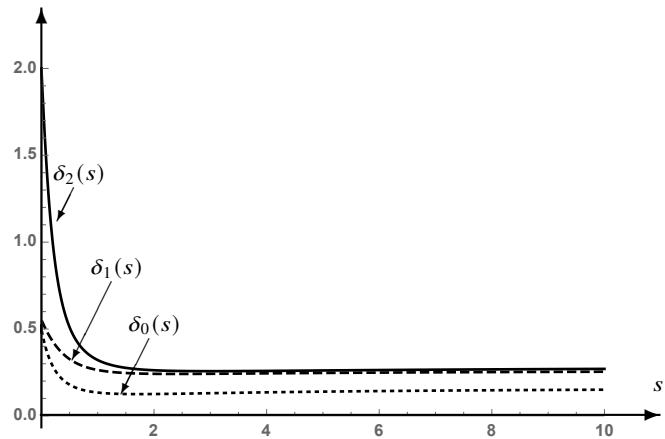
\begin{figure}[hb]
\begin{center}
\hspace*{-10pt}\begin{picture}(12,9)
\put(2,0){\includegraphics[width=8cm]{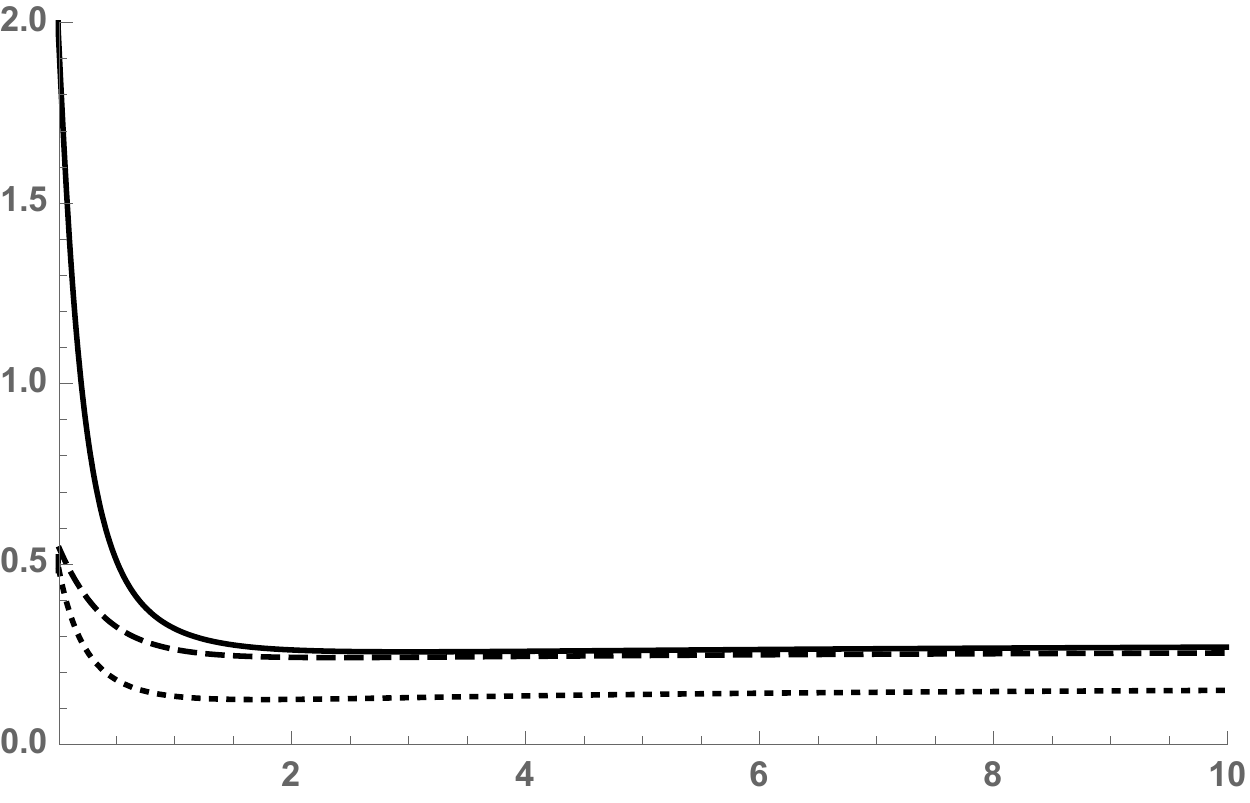}}
\put(10.5,0.7){$s$}
\put(2.5,3.5){$\delta_2(s)$}
\put(2.8,3.4){\vector(-1,-2){0.22}}
\put(3,2.25){$\delta_1(s)$}
\put(3.3,2.15){\vector(-1,-2){0.52}}
\put(3.5,1.5){$\delta_0(s)$}
\put(3.8,1.4){\vector(-1,-2){0.37}}
\thicklines
\put(2.35,0.38){\vector(1,0){8.25}}
\put(2.37,0.2){\vector(0,1){5.6}}
\end{picture}
\caption{\label{fig:lambda} With $\lambda_m$, $\lambda_M$ and $C_M$ given by~\eqref{Ex}, curves $s\mapsto\delta_i(s)$ with $i=0$, $1$ and $2$ are shown. The dotted curve $s\mapsto\delta_0(s)$ shows the estimate~\eqref{BDMScomp} of~\cite[Proposition~4]{BDMMS}. It can be checked numerically that the numerical curves $s\mapsto\big(\delta_i(s),\lambda_i(s)\big)$ with $i=1$, $2$ satisfy the constraints, \emph{i.e.}, stay in their respective triangles for all $s>0$, as shown in Fig.~\ref{Fig:Triangle}.}
\end{center}\vspace*{-20pt}
\end{figure}
\clearpage

\subsection{Further observations}\label{Sec:Further}

In this section, we collect various observations, which are of practical interest, and rely all on the same computations as the ones of Sections~\ref{Sec:GeneralSetting} and~\ref{Sec:Mode-by-mode}.

\subsubsection{Explicit estimates}\label{Sec:Explicit}

The explicit computation of $\delta_2$ and $\lambda_2$ is delicate as it involves finding the roots of high degree polynomials, but it is possible to obtain a very good approximation as follows. After estimating $\lambda\,X\,Y$ by $\lambda\(X^2+Y^2\)/2$, we obtain that
\begin{multline*}
\mathsf D[F]-\lambda\,\mathsf H_1[F]\\
\ge\(1-\frac{\delta\,s^2}{1+s^2}-\frac\lambda2\)X^2-\frac{\delta\,s}{1+s^2}\(1+\sqrt3\,s+\lambda\)X\,Y+\(\frac{\delta\,s^2}{1+s^2}-\frac\lambda2\)Y^2\\
\ge\(1-\frac{\delta\,s^2}{1+s^2}-\frac\lambda2\(1+\frac{\delta\,s}{1+s^2}\)\)X^2-\frac{\delta\,s}{1+s^2}\(1+\sqrt3\,s\)X\,Y\\
+\(\frac{\delta\,s^2}{1+s^2}-\frac\lambda2\(1+\frac{\delta\,s}{1+s^2}\)\)Y^2=:\widetilde{\mathcal Q}(X,Y)
\end{multline*}
is nonnegative for any $X$ and $Y$, under the discriminant condition which amounts to the nonpositivity of
\begin{multline}\label{eq:h2t}
\tilde h_2(\delta,\lambda,s):=\delta^2\,s^2\(\frac{1+\sqrt3\,s}{1+s^2}\)^2\\
-4\(1-\frac{\delta\,s^2}{1+s^2}-\frac\lambda2\(1+\frac{\delta\,s}{1+s^2}\)\)\(\frac{\delta\,s^2}{1+s^2}-\frac\lambda2\(1+\frac{\delta\,s}{1+s^2}\)\)\,.
\end{multline}
By doing a computation as in~\S~\ref{Sec:Plots}, we can find an explicit result, which goes as follows.
\begin{proposition}\label{Prop:tilde} Assume~\eqref{Ex}. The largest value of $\lambda>0$ for which there is some $\delta>0$ such that the quadratic form $\widetilde{\mathcal Q}$ is nonnegative is
\[\textstyle
\tilde\lambda_2(s):=\frac{7\,s^2-\sqrt{21\,s^4+4\,(3+5\,\sqrt3)\,s^3+\,(22+8\,\sqrt 3)\,s^2+4\,(1+\sqrt3)\,s+1}+2\,(1+\sqrt3)\,s+1}{7\,s^2+2\,(2+\sqrt3)\,s+2}
\]
with corresponding $\delta$ given by 
\[\textstyle
\tilde\delta_2(s):=\frac{s^2+1}s\,\frac{\tilde\lambda_2(s)^2-\tilde\lambda_2(s)+2\,s}{7\,s^2+2\,\sqrt3\,s+1-\tilde\lambda_2(s)^2}\,.
\]
\end{proposition}
The proof is tedious but elementary and we shall skip it. By construction, we know that
\[
\lambda_2(s)\ge\tilde\lambda_2(s)\quad\forall\,s>0
\]
and the approximation of $\lambda_2(s)$ by $\tilde\lambda_2(s)$ is numerically quite good (with a relative error of the order of about 10 \%), with exact asymptotics in the limits as $s\to0_+$, in the sense that $\lambda_2(s)/s^2\sim\tilde\lambda_2(s)/s^2$, and $s\to+\infty$. See Figs.~\ref{Fig:F3approx} and~\ref{Fig:F3}. The approximation of $\delta_2(s)$ by $\tilde\delta_2(s)$ is also very good.
\begin{figure}[ht]
\begin{center}\hspace*{-10pt}
\includegraphics[width=8cm]{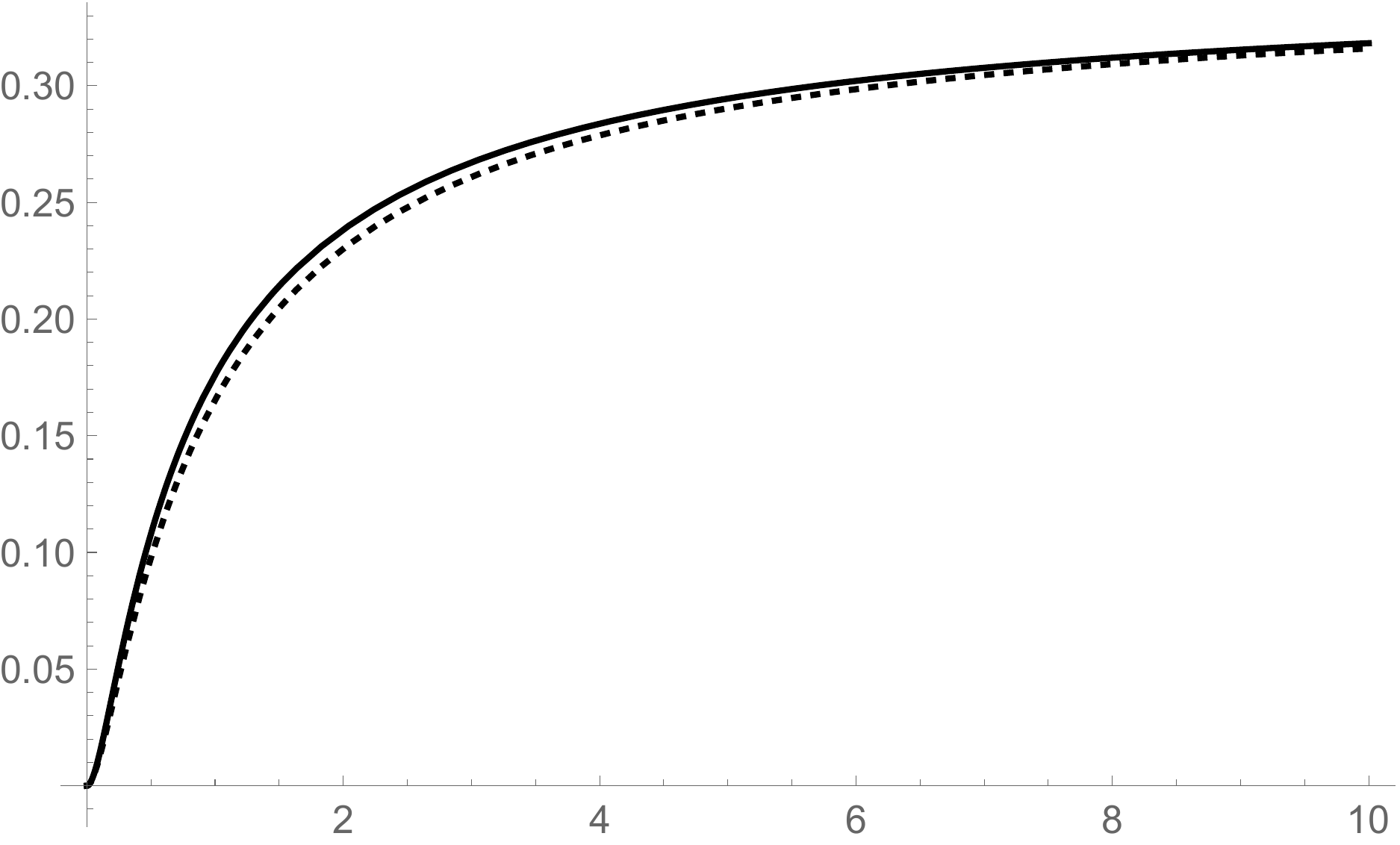}
\caption{\label{Fig:F3approx} Plot of $s\mapsto\lambda_2(s)$ and of of $s\mapsto\tilde\lambda_2(s)$, represented, respectively, by the plain and by the dotted curves.}
\end{center}\vspace*{-20pt}
\end{figure}
\begin{figure}[hb]\vspace*{-20pt}
\begin{center}
\hspace*{-10pt}\begin{picture}(12,9)
\put(2,0){\includegraphics[width=8cm]{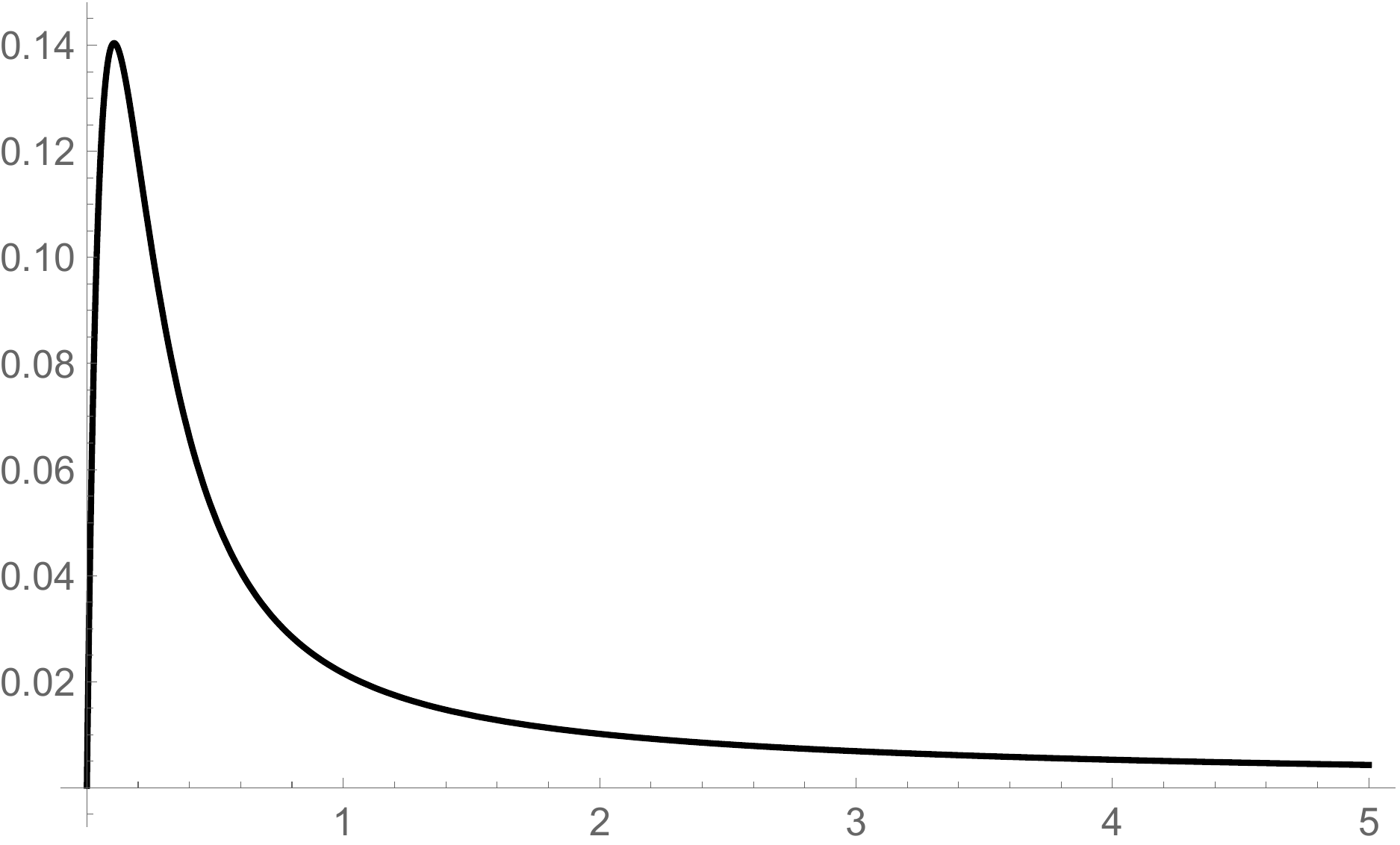}}
\put(10.5,0.7){$s$}
\put(3.5,3.5){$s\mapsto\big(\lambda_2(s)-\tilde\lambda_2(s)\big)\,\big(1+s^{-2}\big)$}
\thicklines
\put(2.35,0.34){\vector(1,0){8.25}}
\put(2.5,0){\vector(0,1){5.2}}
\end{picture}
\caption{\label{Fig:F3} The curves $s\mapsto\lambda_2(s)$ and $s\mapsto\tilde\lambda_2(s)$ have the same asymptotic behaviour as $s\to0_+$ and as $s\to+\infty$.}
\end{center}
\end{figure}
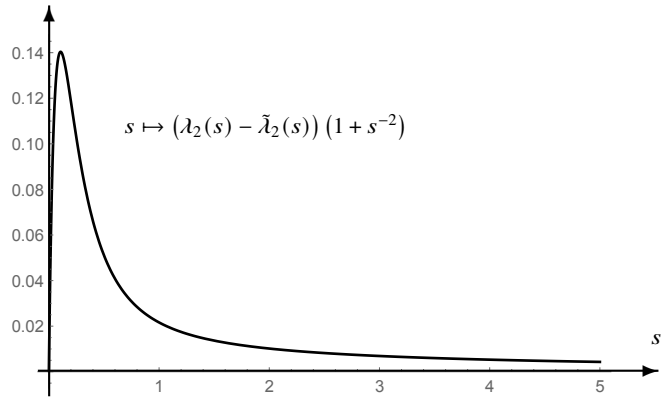

Let us summarize some properties which, as a special case, are of interest for Sections~\ref{Sec:DiffusionLimit} and~\ref{Sec:Nash}.
\begin{lemma}\label{Lem:tilde} With the notation of Proposition~\ref{Prop:tilde}, the function $\tilde\lambda_2$ is monotone increasing, the function $\tilde\delta_2$ is monotone increasing for $s>0$ large enough, and
\[
\lim_{s\to0_+}\frac{\tilde\lambda_2(s)}{s^2}=2\,,\quad\lim_{s\to+\infty}\tilde\lambda_2(s)=1-\sqrt{3/7}\approx0.345346\quad\mbox{and}\quad\lim_{s\to+\infty}\tilde\delta_2(s)=2/7\,.
\]
\end{lemma}
The proof of this result is purely computational and will be omitted here.

\subsubsection{Mode-by-mode diffusion limit}\label{Sec:DiffusionLimit}

We consider the diffusion limit which corresponds to the parabolic scaling applied to the abstract equation~\eqref{EqnEvol}, that is, the limit as $\varepsilon\to0_+$ of
\[\label{EqnEvoleps}
\varepsilon\,\frac{dF}{dt}+\mathsf TF=\frac1\varepsilon\,\mathsf LF\,.
\]
We will not go to the details and should simply mention that this amounts to replace~$\lambda$ by $\lambda\,\varepsilon$ when we look for a rate $\lambda$ which is asymptotically independent of $\varepsilon$. We also have to replace~\eqref{H4} by the assumption
\begin{equation*}
\|\mathsf{AT}(\mathrm{Id}-\Pi)F\|+\frac1\varepsilon\,\|\mathsf{AL}F\|\le C_M^\varepsilon\,\|(\mathrm{Id}-\Pi)F\|\label{H4eps}\tag{H4$_\varepsilon$}
\end{equation*}
in order to clarify the dependence on $\varepsilon$. Since $\mathsf A$, $\mathsf T$, $\Pi$ and $\mathsf L$ do not depend on $\varepsilon$, this simply means that we can write $C_M^\varepsilon=C_M^{(1)}+\frac1\varepsilon\,C_M^{(2)}$ where $C_M^{(1)}$ and $C_M^{(2)}$ are the bounds corresponding to
\[
\|\mathsf{AT}(\mathrm{Id}-\Pi)F\|\le C_M^{(1)}\,\|(\mathrm{Id}-\Pi)F\|\quad\mbox{and}\quad\|\mathsf{AL}F\|\le C_M^{(2)}\,\|(\mathrm{Id}-\Pi)F\|\,.
\]
With these considerations taken into account, proving an entropy - entropy production inequality is equivalent to proving the nonnegativity of
\begin{multline*}
\mathsf D[F]-\lambda\,\varepsilon\,\mathsf H_1[F]\\
\ge\(\frac1\varepsilon-\frac{\delta\,s^2}{1+s^2}-\frac{\lambda\,\varepsilon}2\)X^2-\frac{\delta\,s}{1+s^2}\(\frac1\varepsilon+\sqrt3\,s+\lambda\,\varepsilon\)X\,Y+\(\frac{\delta\,s^2}{1+s^2}-\frac{\lambda\,\varepsilon}2\)Y^2
\end{multline*}
for any $X$ and $Y$, and the discriminant condition amounts to the nonpositivity of
\[
\frac{\delta^2\,s^2}{\varepsilon^2}\(\frac{1+\sqrt3\,\varepsilon\,s+\lambda\,\varepsilon^2}{1+s^2}\)^2-4\(\frac1\varepsilon-\frac{\delta\,s^2}{1+s^2}-\frac{\lambda\,\varepsilon}2\)\(\frac{\delta\,s^2}{1+s^2}-\frac{\lambda\,\varepsilon}2\)\,.
\]
In the limit as $\varepsilon\to0_+$, we find that the optimal choice for $\lambda$ is given by $(\lambda_\varepsilon(s),\delta_\varepsilon(s))$ with
\[
\lim_{\varepsilon\to0_+}\lambda_\varepsilon(s)=2\,s^2\quad\mbox{and}\quad\delta_\varepsilon(s)=2\,(1+s^2)\,\varepsilon\,\big(1+o(1))\big)\,.
\]
Notice that $\lambda(s)=2\,s^2$ corresponds to the expected value of the spectrum associated with the heat equation obtained in the diffusion limit. This also corresponds to the limiting behaviour as $s\to0_+$ of $\tilde\lambda_2$ obtained in Lemma~\ref{Lem:tilde}.

\subsubsection{Towards an optimized mode-by-mode hypocoercivity approach ?}\label{Sec:Mode-by-modeOptimized}

In our method, the essential property of the operator $\mathsf A:=\Big(\mathrm{Id}+(\mathsf T\Pi)^*\mathsf T\Pi\Big)^{-1}(\mathsf T\Pi)^*$ is the equivalence of $\langle\mathsf{AT}\Pi F,F\rangle$ with $\|\Pi F\|^2$ given by the estimate
\[
\frac{\lambda_M}{1+\lambda_M}\,\|\Pi F\|^2\le\langle\mathsf{AT}\Pi F,F\rangle\le\|\Pi F\|^2\,.
\]
These inequalities arise from the \emph{macroscopic coercivity} condition~\eqref{H2} and, using the spectral theorem, from the elementary estimate $z/(1+z)\le1$ for any $z\ge0$. On the one hand the Lyapunov functional $\mathsf H_1[F]:=\tfrac12\,\|F\|^2+\delta\,\re\langle\mathsf AF,F\rangle$ is equivalent to $\|F\|^2$ for $\delta>0$ small enough because $\mathsf A$ is a bounded operator. On the other hand, $\mathsf D[F]=-\,\frac d{dt}\mathsf H_1[F]$ can be compared directly with $\|F\|^2$ because, up to terms that can be controlled, as in the proof of Proposition~\ref{prop:DMS2015}, in the limit as $\delta\to0_+$, $\mathsf D[F]$ is bounded from below by
\[
-\,\langle\mathsf LF,F\rangle+\delta\,\langle\mathsf{AT}\Pi F,F\rangle\ge\lambda_m\,\|(\mathrm{Id}-\Pi)F\|^2+\frac{\delta\,\lambda_M}{1+\lambda_M}\,\|\Pi F\|^2
\]
by Assumptions~\eqref{H1} and~\eqref{H2}. Notice that this estimate holds for any $\delta>0$. The choice of $z/(1+z)$ picks a specific scale and one may wonder if $z/(\varepsilon+z)$ would not be a better choice for some value of $\varepsilon>0$ to be determined. By ``better'', we simply have in mind to get a larger decay rate as $t\to+\infty$, without trying to optimize on the constant $C$ in~\eqref{Decay:BDMMS}. It turns out that \emph{the answer is negative}, as $\varepsilon$ can be scaled out. Let us give some details. 

\medskip Let us replace $\mathsf A$ by
\[
\mathsf A_\varepsilon:=\Big(\varepsilon^2\,\mathrm{Id}+(\mathsf T\Pi)^*\mathsf T\Pi\Big)^{-1}(\mathsf T\Pi)^*
\]
for some $\varepsilon>0$ that can be adjusted, without changing the general strategy, and consider the Lyapunov functional
\[
\mathsf H_{1,\varepsilon}[F]:=\tfrac12\,\|F\|^2+\delta\,\re\langle\mathsf A_\varepsilon F,F\rangle
\]
for some $\delta>0$, so that
\be{45a}
\begin{array}{rl}\displaystyle
\mathsf D_\varepsilon[F]:=&\displaystyle-\,\frac d{dt}\mathsf H_{1,\varepsilon}[F]\\[6pt]
=&-\,\langle\mathsf LF,F\rangle+\delta\,\langle\mathsf{A_\varepsilon T}\Pi F,F\rangle\\[4pt]
&-\,\delta\,\re\langle\mathsf{TA_\varepsilon}F,F\rangle+\delta\,\re\langle\mathsf{A_\varepsilon T}(\mathrm{Id}-\Pi)F,F\rangle-\delta\,\re\langle\mathsf{A_\varepsilon L}F,F\rangle\,.
\end{array}
\ee
For a given $\xi\in\R^d$ considered as a parameter, if $f$ solves~\eqref{eq:model} and if $F=\hat f$, then we are back to the framework of~\S~\ref{Sec:Mode-by-mode}. In this framework, the operator $\mathsf A_\varepsilon$ is given~by
\[
(\mathsf A_\varepsilon F)(v)=-\frac{\,i\,\xi}{\varepsilon^2+|\xi|^2}\cdot\int_{\R^d}w\,F(w)\,dw\,\M(v)\,.
\]
We have to adapt the computations of~\S~\ref{Sec:Fourier} to $\varepsilon\neq1$.

As a first remark, we notice that we do not need any estimate of $\|\mathsf A_\varepsilon F\|$: all quantities in~\eqref{45a} involving $\mathsf A_\varepsilon$ are directly computed except of $\re\langle\mathsf{TA_\varepsilon}F,F\rangle$. Estimating $\re\langle\mathsf{TA_\varepsilon}F,F\rangle$ provides a bound which is independent of $\varepsilon$ for the following reason. When we solve $G=\mathsf A_\varepsilon F$, \emph{i.e.}, if $(\mathsf T\Pi)^*F=\varepsilon^2\,G+(\mathsf T\Pi)^*\,\mathsf T\Pi\,G$, then
\[
\langle\mathsf{TA_\varepsilon}F,F\rangle=\langle G,(\mathsf T\Pi)^*\,F\rangle=\varepsilon^2\,\|G\|^2+\|\mathsf T\Pi G\|^2=\varepsilon^2\,\|\mathsf A_\varepsilon F\|^2+\|\mathsf{TA}_\varepsilon F\|^2\,.
\]
By the Cauchy-Schwarz inequality, we know that
\[
\langle G,(\mathsf T\Pi)^*F\rangle=\langle\mathsf{TA}_\varepsilon F,(\mathrm{Id}-\Pi)F\rangle\le\|\mathsf{TA}_\varepsilon F\|\,\|(\mathrm{Id}-\Pi)F\|\,,
\]
which proves that $\|\mathsf{TA}_\varepsilon F\|\le\|(\mathrm{Id}-\Pi)F\|$ and, as a consequence,
\[
|\re\langle\mathsf{TA_\varepsilon}F,F\rangle|\le\|(\mathrm{Id}-\Pi)F\|^2\,.
\]
It is clear that the right-hand side is independent of $\varepsilon>0$. A better estimate is obtained by computing as in~\eqref{Est3F}. By doing so, we obtain
\[
|\re\langle\mathsf{TA_\varepsilon}F,F\rangle|\le\frac{|\xi|^2}{\varepsilon^2+|\xi|^2}\,\|(\mathrm{Id}-\Pi)F\|^2\,.
\]

As in~\S~\ref{Sec:Fourier}, we have $\lambda_m=1$, $\lambda_M=|\xi|^2$ and the same computations show that
\begin{align*}
&|\re\langle\mathsf A_\varepsilon F,F\rangle|\le\frac{|\xi|}{\varepsilon^2+|\xi|^2}\,\|\Pi F\|\,\|(\mathrm{Id}-\Pi)F\|\,,\\
&|\re\langle\mathsf{A_\varepsilon L}F,F\rangle|\le\frac{|\xi|}{\varepsilon^2+|\xi|^2}\,\|\Pi F\|\,\|(\mathrm{Id}-\Pi)F\|\,,\\
&|\re\langle\mathsf{A_\varepsilon T}(\mathrm{Id}-\Pi)F,F\rangle|\le\frac{\sqrt3\,|\xi|^2}{\varepsilon^2+|\xi|^2}\,\|\Pi F\|\,\|(\mathrm{Id}-\Pi)F\|\,.
\end{align*}
This establishes that~\eqref{H4} holds with
\[
C_M=\frac{|\xi|\(1+\sqrt3\,|\xi|\)}{\varepsilon^2+|\xi|^2}\,,
\]
\[\label{Equiv}
\left|\mathsf H_{1,\varepsilon}[F]-\frac12\,\|F\|^2\right|\le\delta\,\left|\re\langle\mathsf A_\varepsilon F,F\rangle\right|\le\frac{\delta\,|\xi|}{\varepsilon^2+|\xi|^2}\,\|\Pi F\|\,\|(\mathrm{Id}-\Pi)F\|\,,
\]
and as a consequence it yields an improved version of~\eqref{H-norm} which reads
\be{36b}
\frac12\(1-\frac{\delta\,|\xi|}{\varepsilon^2+|\xi|^2}\)\|F\|^2\le\mathsf H_{1,\varepsilon}[F]\le\frac12\(1+\frac{\delta\,|\xi|}{\varepsilon^2+|\xi|^2}\)\|F\|^2\,.
\ee
Notice that the lower bound holds with a positive left-hand side for any $\xi$ only under the additional condition that
\[
\delta<2\,\varepsilon\,.
\]
Anyway, if we allow $\delta$ to depend on $\xi$, the whole method still applies, including for proving the hypocoercive estimate on $\|F\|^2$, if the condition
\[\label{Condition-xi}
\varepsilon^2-\delta\,|\xi|+|\xi|^2\ge0
\]
is satisfied for every $\xi$. 

With $X:=\|(\mathrm{Id}-\Pi)F\|$, $Y:=\|\Pi F\|$, and $s=|\xi|$, we look for the largest value of~$\lambda$ for which the right-hand side in
\begin{multline}\label{rhs}
\mathsf D_\varepsilon[F]-\lambda\,\mathsf H_{1,\varepsilon}[F]\\
\ge\(1-\frac{\delta\,s^2}{\varepsilon^2+s^2}-\frac\lambda2\)X^2-\frac{\delta\,s}{\varepsilon^2+s^2}\(1+\sqrt3\,s+\lambda\)X\,Y+\(\frac{\delta\,s^2}{\varepsilon^2+s^2}-\frac\lambda2\)Y^2
\end{multline}
is nonnegative for any $X$ and $Y$. Recall that $s$ is fixed and $\delta$ is a parameter to be adjusted. If we change the parameter $\delta$ into $\delta_*$ such that 
\be{delta1}
\frac{\delta\,s}{\varepsilon^2+s^2}=\frac{\delta_*\,s}{1+s^2}\,,
\ee
then the nonnegativity problem of the r.h.s.~in~\eqref{rhs} is reduced to the same problem with $\varepsilon=1$, provided that no additional constraint is added. Let us define
\[
h_3(\delta,\lambda,\varepsilon,s):=\delta^2\,s^2\(\frac{1+\sqrt3\,s+\lambda}{\varepsilon^2+s^2}\)^2-4\(1-\frac{\delta\,s^2}{\varepsilon^2+s^2}-\frac\lambda2\)\(\frac{\delta\,s^2}{\varepsilon^2+s^2}-\frac\lambda2\)
\]
with $(\delta,\lambda)$ in the triangle
\[
\mathcal T_m^\varepsilon(s):=\Big\{(\delta,\lambda)\in\(0,\big(1+\varepsilon^2\,s^{-2}\big)\,\lambda_m\)\times(0,2\,\lambda_m)\,:\,\lambda<2\(\lambda_m-\tfrac{\delta\,s^2}{\varepsilon^2+s^2}\)\Big\}
\]
and $\lambda_m=1$. Exactly the same method as in~\S~\ref{Sec:Plots} determines the curves $s\mapsto\lambda_3(s)$ and $s\mapsto\delta_3(s,\eps)$, but we have $\lambda_3(s)=\lambda_2(s)$ for any $s>0$ while $\delta_2(s)$ and $\delta_3(s,\eps)$ can be deduced from each other using~\eqref{delta1}. Solutions have to satisfy the constraint $s\mapsto\big(\delta_3(s,\eps),\lambda_3(s)\big)\in\mathcal T_m^\varepsilon(s)$ for any $s>0$. It is straightforward to check that $(\delta,\lambda)\in\mathcal T_m^\varepsilon(s)$ if and only if $(\delta_*,\lambda)\in\mathcal T_m^1(s)=\mathcal T_m(s)$, where $\delta_*$ is determined by~\eqref{delta1}. Altogether, our observations can be reformulated as follows.
\begin{lemma} Assume~\eqref{Ex}. Then for any $s>0$, we have
\[
\max\Big\{\lambda>0\,:\,(\delta,\lambda)\in\mathcal T_m^\varepsilon(s)\,,\;h_3(\delta,\lambda,\varepsilon,s)\le0\Big\}
\]
is independent of $\varepsilon>0$.\end{lemma}
To conclude this subsection, we note that, while the Lyapunov functionals $\HH_{1,\eps}$ are clearly different for different values of $\eps>0$, mode-by-mode, \emph{i.e.}, for a given value of $s=|\xi|$, they all yield the same exponential decay rate $\lambda=\lambda_2(s)$, when choosing the best parameter $\delta=\delta_3(s,\eps)$. Similarly, no improvement on the constant $C$ as in~\eqref{Decay:BDMMS} is achieved by adjusting $\eps>0$ when $\varepsilon$ is taken into account in~\eqref{36b}, for proving the equivalence of $\HH_{1,\eps}[F]$ and $\|F\|^2$. This reflects a deep scaling invariance of the method. 

\section{Convergence rates and decay rates}\label{Sec:Consequences}

In this section, we come back to the study of~\eqref{eq:model} and consider two situations. A periodic solution on a \emph{small torus} has a behaviour driven by high frequencies corresponding to $|\xi|$ large, while the decay rate of a solution on the whole Euclidean space is asymptotically determined by the low frequency regime with $\xi\to0$. In the latter case, we use estimates as in Nash type inequalities and relate the time decay with the behaviour of $\lambda(\xi)$ in a neighbourhood of $\xi=0$.

\subsection{Exponential convergence rate on a small torus}

In this section, let us assume that $\mathcal X=[0,L)^d$ (with periodic boundary conditions) and consider the limit as \hbox{$L\to0_+$}. With the notation of~\S~\ref{Sec:Intro} and the Fourier transform~\eqref{Fourier}, the periodicity implies that $\xi\in(2\,\pi/L)\,\mathbb Z^d$ and in particular, for any fixed $j\in\Z^d$ and $\xi=2\,\pi\,j/L$, we have $|\xi|\to+\infty$ as $L\to0_+$, unless $j=0$. Let us denote by~ $\lambda_L(\xi)$ the optimal constant in~\eqref{Opt} when $\xi$ is limited to $(2\,\pi/L)\,\mathbb Z^d\setminus\{0\}$. We recall~that
\[
\lambda_\star:=\liminf_{L\to0_+}\inf_{\xi\in(2\,\pi/L)\,\mathbb Z^d\setminus\{0\}}\lambda_L(\xi)\ge1-\sqrt{3/7}
\]
according to Lemma~\ref{Lem:tilde}. As a consequence, we have the following result.
\begin{proposition}\label{Prop:Torus} For any $\varepsilon>0$, small, there exists some $L_\varepsilon>0$ such that, if $\mathcal X=[0,L)^d$ for an arbitrary $L\le L_\varepsilon$, if $f$ solves~\eqref{eq:model} with $f_0\in\mathrm L^2(\mathcal X\times\R^d,dx\,d\gamma)$ and $\mathsf L=\mathsf L_1$ or $\mathsf L=\mathsf L_2$, then we have
\[
\|f(t,\cdot,\cdot)-\bar f\,\M\|_{\mathrm L^2(dx\,d\gamma)}^2\le(1+\eps)\,\|f_0-\bar f\,\M\|_{\mathrm L^2(dx\,d\gamma)}^2\,e^{-\min\{2,\lambda_\star-\varepsilon\!\}\,t}\quad\forall\,t\ge0\,,
\]
with $\bar f:=\frac1{L^d}\iint_{\mathcal X\times\R^d}f_0(x,v)\,dx\,dv$.\end{proposition}
\begin{proof} Let us notice that $g(t,v):=\hat f(t,0,v)=\int_{\mathcal X}f(t,x,v)\,dx$ solves
\[
\partial_tg=\mathsf Lg\,.
\]
As a consequence either of the definition of $\mathsf L=\mathsf L_2$, or of the Gaussian Poincar\'e inequality
\[
\|g-\bar f\,\M\|_{\mathrm L^2(d\gamma)}^2\le\|\nabla g\|_{\mathrm L^2(d\gamma)}^2
\]
if $\mathsf L=\mathsf L_1$, we know that
\[
\|g(t,\cdot)-\bar f\,\M\|_{\mathrm L^2(d\gamma)}^2\le\|g(0,\cdot)-\bar f\,\M\|_{\mathrm L^2(d\gamma)}^2\,e^{-2t}\quad\forall\,t\ge0\,.
\]
By the Plancherel formula, we have
\[
\|f(t,\cdot,\cdot)-\bar f\,\M\|_{\mathrm L^2(dx\,d\gamma)}^2=\|g(t,\cdot)-\bar f\,\M\|_{\mathrm L^2(d\gamma)}^2+(2\,\pi)^{-d}\kern-18pt\sum_{\xi\in(2\,\pi/L)\,\mathbb Z^d\setminus\{0\}}\kern-18pt\|\hat f(t,\xi,\cdot)\|_{\mathrm L^2(d\gamma)}^2\,.
\]
The conclusion follows from $C(s)=\(1+s^2+\tilde\delta_2(s)\,s\)/\(1+s^2-\tilde\delta_2(s)\,s\)$,
\[
\|\hat f(t,\xi,\cdot)\|_{\mathrm L^2(d\gamma)}^2\le C(|\xi|)\,\|\hat f_0(\xi,\cdot)\|_{\mathrm L^2(d\gamma)}^2\,e^{-\lambda_L(\xi)\,t}
\]
for any $(t,\xi)\in\R^+\times(2\,\pi/L)\,\mathbb Z^d\setminus\{0\}$, and the estimates of Lemma~\ref{Lem:tilde}.
\qed\end{proof}

\subsection{Algebraic decay rate in the whole Euclidean space}\label{Sec:Nash}

As a refinement of~\cite{BDMMS}, we investigate the decay estimates for the solution to~\eqref{eq:model} on $\mathcal X=\R^d$. Here we rely on \emph{Nash type estimates}.

\medskip To start with, let us consider a model problem. Assume that $s\mapsto\lambda(s)$ is a positive non-decreasing bounded function on $(0,+\infty)$ and, for any $s>0$, let
\[
h_\lambda(M,R,s):=\lambda(R)\(\omega_d\,R^d\,M^2-s\)\,,\quad\lambda^*(M,s):=-\min_{R>0}h_\lambda(M,R,s)\,,
\]
where $M$ is a positive parameter and $\omega_d=|\mathbb S^{d-1}|/d$. Since $h_\lambda(M,R,s)\sim-\lambda(R)\,s$ as $R\to0_+$ and $h_\lambda(M,R,s)\ge c\,R^d$ for some $c>0$ as $R\to+\infty$, there is indeed some $R>0$ such that $\lambda^*(M,s)=-\,h_\lambda(M,R,s)$ and $\lambda^*(M,s)$ is positive for any $(M,s)\in(0,+\infty)^2$. We also define the monotone decreasing function
\[
\psi_{\!\lambda,M}(s):=-\int_1^s\frac{dz}{\lambda^*(M,z)}\quad\forall\,s\ge0\,.
\]
Our first result is a decay rate on $\R^d$ for a solution of $\partial_tu=\mathcal L u$ where the operator~$\mathcal L$ acts on the Fourier space as the multiplication of $\xi\mapsto\hat u(\xi)$ with some scalar function $-\lambda(\xi)/2$, for any~\hbox{$\xi\in\R^d$}. With just a spectral inequality, we obtain the following estimate.
\begin{lemma}\label{Lem:Nash} Assume that $s\mapsto\lambda(s)$ is a positive non-decreasing bounded function on $(0,+\infty)$ such that, with the above notation, $\lim_{s\to0_+}\psi_{\!\lambda,\mu}(s)=+\infty$ for all $\mu>0$. If $u\in C(\R^+,\mathrm L^1\cap\mathrm L^2(dx))$ is such that $M:=\|u(t,\cdot)\|_{\mathrm L^1(dx)}$ does not depend on $t$ and
\[
\frac d{dt}|\hat u(t,\xi)|^2\le-\,\lambda(|\xi|)\,|\hat u(t,\xi)|^2\quad\forall\,(t,\xi)\in\R^+\times\R^d\,,
\]
then
\[
\|u(t,\cdot)\|_{\mathrm L^2(dx)}^2\le\psi_{\!\lambda,M}^{-1}\(t+\psi_{\!\lambda,M}\(\|u(0,\cdot)\|_{\mathrm L^2(dx)}^2\)\)\quad\forall\,t\in\R^+\,.
\]
\end{lemma}
Here $\hat u$ denotes the Fourier transform of $u$ in $x$.
\begin{proof} The inspiration for the proof comes from~\cite[page~935]{Nash58}. Let
\begin{align*}
y(t):=\int_{\R^d}|\hat u(t,\xi)|^2\,d\xi&\le\int_{|\xi|\le R}\kern-3pt\|\hat u(t,.)\|_{\mathrm L^\infty(d\xi)}^2\,d\xi+\frac1{\lambda(R)}\int_{\R^d}\lambda(|\xi|)\,|\hat u(t,\xi)|^2d\xi\\
&\le\omega_d\,R^d\,\|u(t,.)\|_{\mathrm L^1(dx)}^2-\frac1{\lambda(R)}\,\frac d{dt}\int_{\R^d}|\hat u(t,\xi)|^2\,d\xi \,.
\end{align*}
Hence,
\[
y'\le h_\lambda(M,R,y)\,,
\]
for any $R>0$. Taking the minimum of the r.h.s.~over $R>0$, we obtain
\[
y'\le-\,\lambda^*(M,y)\,,
\]
and the conclusion follows after elementary computations.
\qed\end{proof}
\begin{example} Let us consider a case that we have already encountered in Sections \ref{Sec:DiffusionLimit} and~\ref{Sec:Nash}. If $\lambda(s)=2\,s^2$ for $s\in(0,1)$, we find that
\[
\lambda^*(M,s)=2\,d\(\frac2{\omega_d\,M^2}\)^\frac2d\(\frac s{d+2}\)^{1+\frac2d}\quad\forall\,s\in(0,1)\,.
\]
With $c_d:=\frac12\,\big((d+2)/2\big)^{1+2/d}\,\omega_d^{2/d}$, we find
\[
\psi_{\!\lambda,M}(s)=c_d\,M^\frac4d\(s^{-\frac2d}-1\)
\]
and deduce from Lemma~\ref{Lem:Nash} that
\[
\|u(t,\cdot)\|_{\mathrm L^2(dx)}^2\le\(\|u(0,\cdot)\|_{\mathrm L^2(dx)}^{-\frac4d}+\frac t{c_d}\,\|u(0,\cdot)\|_{\mathrm L^1(dx)}^{-\frac4d}\)^{-\frac d2}\quad\forall\,t\in\R^+\,.
\]
This estimate is similar to the estimate that one would deduce from Nash's inequality in the case of the heat equation on $\R^d$. Notice that differentiating this estimate at $t=0$ gives a proof of Nash's inequality, with the same constant as in Nash's proof in~\cite{Nash58} (see~\cite{Bouin_2020} for a discussion of the optimal constant).
\end{example}

The assumption on $u$ in Lemma~\ref{Lem:Nash} is a coercivity estimate for the operator $\mathcal L$ with Fourier representation $-\lambda(|\xi|)/2$, which allows us to use a Bihari-LaSalle estimate, \emph{i.e.}, a nonlinear version of Gr\"onwall's lemma. For the main application in this paper, we have to rely on a hypocoercivity estimate, which is slightly more complicated.
\begin{lemma}\label{Lem:hypoNash} Assume that $s\mapsto\lambda(s)$ is a positive non-decreasing bounded function on $(0,+\infty)$ such that, with the above notation, $\lim_{s\to0_+}\psi_{\!\lambda,\mu}(s)=+\infty$ for all $\mu>0$. Let $u\in C(\R^+,\mathrm L^1\cap\mathrm L^2(dx))$ be such that, for some bounded continuous function $s\mapsto C(s)$ such that $C(s)\ge1$ for any $s>0$,
\[
|\hat u(t,\xi)|^2\le C(|\xi|)\,|\hat u(0,\xi)|^2\,e^{-\,\lambda(|\xi|)\,t}\quad\forall\,(t,\xi)\in\R^+\times\R^d
\]
and $\|u(t,\cdot)\|_{\mathrm L^1(dx)}\le M$ for some $M$ which does not depend on $t$. Then, for any $t\ge0$, we have
\begin{equation}\label{eq:lineest}
\|u(t,\cdot)\|_{\mathrm L^2(dx)}^2\le\Psi_{M,Q}(t)\,,
\end{equation}
where $Q:=\|u(0,\cdot)\|_{\mathrm L^2(dx)}$ and
\be{Psi}
\Psi_{M,Q}(t):=\inf_{R>0}\(\int_0^RC(s)\,e^{-\,\lambda(s)\,t}\,s^{d-1}\,ds\,\omega_d\,d\,M^2+\sup_{s\ge R}C(s)\,e^{-\,\lambda(R)\,t}\,Q^2\).
\ee
\end{lemma}
\begin{proof} For any $R>0$, we have
\[
\int_{|\xi|\le R}|\hat u(t,\xi)|^2\,d\xi\le\int_{|\xi|\le R}C(|\xi|)\,e^{-\,\lambda(|\xi|)\,t}\,d\xi\,\|\hat u(0,\cdot)\|_{\mathrm L^\infty(\R^d,d\xi)}^2
\]
with $\|\hat u(0,\cdot)\|_{\mathrm L^\infty(\R^d,d\xi)}\le\|u(0,\cdot)\|_{\mathrm L^1(\R^d,dx)}$ on the one hand, and
\[
\int_{|\xi|>R}|\hat u(t,\xi)|^2\,d\xi\le\sup_{s>R}C(s)\,e^{-\,\lambda(R)\,t}\|\hat u(0,\cdot)\|_{\mathrm L^2(\R^d,d\xi)}^2
\]
on the other hand. The result follows by optimizing on $R>0$.
\qed\end{proof}

The result of Lemma~\ref{Lem:hypoNash} is not as explicit as the result of Lemma~\ref{Lem:Nash}, but it is useful to investigate, for instance, the limit as $t\to+\infty$: if $\lim_{s\to0_+}C(s)=C(0)>0$ and $\lambda(s)=2\,s^2$ for any $s\in(0,1)$, then one can prove that
\[
\|u(t,\cdot)\|_{\mathrm L^2(dx)}^2\le O\(t^{-d/2}\)\quad\mbox{as}\quad t\to+\infty\,.
\]

\medskip In the spirit of~\cite{BDMMS}, let us draw some consequences for the solution of~\eqref{eq:model}.
\begin{theorem}\label{Thm:WholeSpace} If $f$ solves~\eqref{eq:model} for some nonnegative initial datum $f_0\in\mathrm L^2(\R^d\times\R^d,dx\,d\gamma)\cap\mathrm L^2\big(\R^d,d\gamma;\mathrm L^1(\R^d,dx)\big)$ and $\mathsf L=\mathsf L_1$ or $\mathsf L=\mathsf L_2$, then we have the estimate
\[
\|f(t,\cdot,\cdot)\|_{\mathrm L^2(\R^d\times\R^d,dx\,d\gamma)}^2\le(2\,\pi)^{-d}\,\Psi_{M,Q}(t)
\]
with $M=\|f_0\|_{\mathrm L^2\(\R^d,d\gamma;\mathrm L^1(\R^d,dx)\)}$, $Q=\|f_0\|_{\mathrm L^2(\R^d\times\R^d,dx\,d\gamma)}$, and $\Psi_{M,Q}(t)$ defined by~\eqref{Psi} using $C(s)=(2+\delta(s))/(2-\delta(s))$ and $\lambda(s)$, for any pair $(\delta,\lambda)$ of continuous functions on $(0,+\infty)$ taking values in $(0,2)\times(0,+\infty)$, with $s\mapsto\lambda(s)$ monotone non-decreasing, such that the \emph{entropy -- entropy production inequality}~\eqref{Opt} and the equivalence~\eqref{H-norm-xi} hold.
\end{theorem}
Here we abusively write $\lambda(\xi)=\lambda(s)$ and $\delta(\xi)=\lambda(s)$ with $s=|\xi|$. 
\begin{proof} We estimate $\|f(t,\cdot,\cdot)\|_{\mathrm L^2(\R^d\times\R^d,dx\,d\gamma)}^2$ using~\eqref{Fourier} and Plancherel's theorem
\[
\iint_{\R^d\times\R^d}|f(t,x,v)|^2\,dx\,d\gamma=\frac1{(2\,\pi)^d}\iint_{\R^d\times\R^d}|\hat f(t,\xi,v)|^2\,d\xi\,d\gamma\,.
\]
Applying the results of Theorem~\ref{theo:DMS2015} with
\[
C(\xi)=\frac{1+|\xi|^2+\delta(\xi)\,|\xi|}{1+|\xi|^2-\delta(\xi)\,|\xi|}\,,
\]
we learn that
\[
\int_{\R^d}|\hat f(t,\xi,v)|^2\,d\gamma\le C(\xi)\int_{\R^d}|\hat f_0(\xi,v)|^2\,d\gamma\,e^{-\lambda(\xi)\,t}\,.
\]
We can apply the same strategy as for Lemma~\ref{Lem:hypoNash}, with
\begin{align*}
&\iint_{B_R\times\R^d}C(\xi)\,|\hat f_0(\xi,v)|^2\,e^{-\lambda(\xi)\,t}\,d\xi\,d\gamma\le\int_{|\xi|\le R}C(\xi)\,e^{-\,\lambda(\xi)\,t}\,d\xi\,M^2\,,\\
&\iint_{B_R^c\times\R^d}C(\xi)\,|\hat f_0(\xi,v)|^2\,e^{-\lambda(\xi)\,t}\,d\xi\,d\gamma\le\sup_{\xi\in B_R^c}C(\xi)\,e^{-\lambda(R)\,t}\,Q^2\,,
\end{align*}
using $\sup_{\xi\in\R^d}|\hat f_0(\xi,v)|\le\int_{\R^d}f_0(x,v)\,dx$ for the first inequality, and the monotonicity of $\lambda$.\qed\end{proof}

In practice, any good estimate, for instance the estimate based on the functions $(\tilde\delta_2,\tilde\lambda_2)$ of Proposition~\ref{Prop:tilde}, provides us with explicit and constructive decay rates of the solution to~\eqref{eq:model} on $\R^d$. As a concluding remark, it has to be made clear that the method is not limited to the operators $\mathsf L_1$ and $\mathsf L_2$.

\part{The Goldstein--Taylor model}{The Goldstein--Taylor model\hypertarget{PartIII}{}}

\section{General setting and Fourier decomposition}\label{sec:GTgen}

Consider the \emph{two velocities Goldstein--Taylor (GT) model} (\emph{cf.}~\cite[\S~1.4]{DMS-2part}) with constant relaxation coefficient $\sigma>0$, position variable $x\in \mathcal X\subseteq \R$, and $t>0$:
\begin{align}\begin{aligned}\label{eq:GT}
\partial_t f_+(t,x)+ \partial_x f_+(t,x) &= \frac{\sigma}{2}\,\big(f_-(t,x)-f_+(t,x)\big)\,,\\
\partial_t f_-(t,x) - \partial_x f_-(t,x) &= - \frac{\sigma}{2}\,\big(f_-(t,x)-f_+(t,x)\big)\,,\\
f_{\pm}(x,0)&= f_{\pm,0}(x)\,,
\end{aligned}
\end{align}
where $f_{\pm}(t,x)$ are the density functions of finding a particle with a velocity $\pm 1$ in a position $x$ at time $t>0$ and $f_{\pm,0}\in \mathrm{L}^1_+(\mathcal X)$
is the initial configuration. This model is the prime example of discrete velocity BGK equations, as described in~\S~\ref{Sec:Abstract2}, Example~\ref{Ex2}, with $b=(1/2,1/2)^T$ and $V=\diag(1,-1)$.
We consider two situations for~$\mathcal X$, the one-dimensional torus and the real line, \emph{i.e.}, $\mathcal X\in\{\T, \R\}$.

Rewriting~\eqref{eq:GT} in the macroscopic variables of (mass and flux densities)
\[
u(t,x): = f_+(t,x) + f_-(t,x)\geq 0\,,\quad v(t,x):=f_+(t,x)-f_-(t,x)\,,
\]
leads to the transformed equations
\begin{align}
\begin{aligned}\label{eq:GTtrans}
\partial_t u(t,x) &= -\,\partial_x v(t,x)\,,\\
\partial_t v(t,x) &= -\,\partial_x u(t,x)-\sigma\,v(u,x)\,,
\end{aligned}
\end{align}
for $x\in\mathcal X$, $t\geq 0$. Integrating these equations along $\mathcal X$ directly shows that the total mass is conserved for all times, \emph{i.e.}\ $\int_\mathcal X u(t,x)\,dx \equiv \int_\mathcal X u(x,0)\,dx$, and that the total flux is decaying exponentially, \emph{i.e.}\ $\int_\mathcal X v(t,x)\,dx = e^{-\sigma\,t}\int_\mathcal X v(x,0)\,dx$ for $t\geq 0$.

A Fourier transformation in the space variable $x\in\mathcal X$ leads to ODEs of form~\eqref{abstr-eq-FT}, given explicitly as
\be{eq:transformedGT}
\partial_t\hat{y}(t,\xi)=-C(\xi,\sigma)\,\hat{y}(t,\xi)
\ee
with
\[
\hat{y}(t,\xi):=\begin{pmatrix}
\hat{u}(t,\xi)\\ \hat{v}(t,\xi)
\end{pmatrix}\quad\mbox{and}\quad C(\xi,\sigma):=\begin{pmatrix}
0 &i\,\xi \\
i\,\xi& \sigma
\end{pmatrix}\,,
\]
for the Fourier modes $\xi\in\Z$ in the case of $\mathcal X=\T$, and $\xi\in\R$ for $\mathcal X = \R$.

The matrix $C(\xi,\sigma)$ from~\eqref{eq:transformedGT} has the eigenvalues
$$\lambda_{\pm}(\xi,\sigma):= \frac{\sigma}{2} \pm \sqrt{\frac{\sigma^2}{4}-\xi^2}$$
and hence its \emph{modal spectral gap} is given by
\begin{equation}\label{eq:spectralgap}
\mu(\xi,\sigma):= \re\left({\frac{\sigma}{2} - \sqrt{\frac{\sigma^2}{4}-\xi^2}}\right)\,,\quad \xi\neq 0\,.
\end{equation}
For $\mathcal X = \R$, the modal spectral gap takes all values in the intervall $(0,\sigma/2]$ with $\lim_{\xi\to0}\mu(\xi,\sigma)=0$. To obtain decay estimates with the sharp decay rate of solutions $y(t,\xi)$ to~\eqref{eq:GTtrans} it is therefore important to achieve precise estimates of the decay behavior as $\xi\to 0$. For $\mathcal X=\T$, the spectral gap for solutions to~\eqref{eq:GTtrans} corresponds to the \emph{uniform-in-$\Z$ spectral gap}, \emph{i.e.}
\begin{equation*}\label{mubar}
\overline{\mu}(\sigma):=\min_{\xi\in\Z\setminus\{0\}} \mu(\xi,\sigma)\,.
\end{equation*}
The set of modal spectral gaps which coincide with the uniform spectral gap is denoted by~$\Xi(\sigma)$ and depends on the values of $\sigma>0$:
\begin{itemize}
\item For $\sigma\in(0,2]$ it follows that 
$$\overline{\mu}(\sigma)= \frac{\sigma}{2}\,, \quad \Xi(\sigma) = \Z \setminus\{0\}\,.$$
\item For $\sigma > 2$ the lowest modes determine the uniform-in-$\Z$ spectral gap,
\begin{equation}\label{eq:mubar1}
\omu(\sigma)=\mu(\pm 1,\sigma) = \frac{\sigma}{2}-\sqrt{\frac{\sigma^2}{4}-1}\,, \quad \Xi = \{-1,1\}\,.
\end{equation} 
\end{itemize}

\medskip Now, we consider the two hypocoercivity methods from~\S~\ref{Sec:Abstract1} and~\S~\ref{Sec:Abstract2} for solutions $\hat{y}(t,\xi)$ of~\eqref{eq:transformedGT} for fixed but arbitrary modes $\xi$.

\medskip\noindent{\bf Approach of~\S~\ref{Sec:Abstract2}}: For equations of form~\eqref{abstr-eq-FT}, we consider the modal Lyapunov functionals $\|\hat{y}(t,\xi)\|^2_{P(\xi,\sigma)}$ with deformation matrices $P(\xi,\sigma)$. These functionals satisfy the explicit estimates of form~\eqref{eq:modaldecay}, which go as follows:
\begin{itemize}
\item 
For fixed $|\xi|\neq \sigma/2$, $|\xi|>0$ the matrix $C(\xi,\sigma)$ is not defective and it follows from Lemma~\ref{lemma:Pdefinition} that
\begin{equation}\label{eq:Pdecay}
\|\hat{y}(t,\xi)\|^2_{P(\xi,\sigma)} \leq e^{-2\,\mu(\xi,\sigma)\,t}\,\|\hat{y}(\xi,0)\|^2_{P(\xi,\sigma)}\,,
\end{equation}
with $P(\xi,\sigma)=P^{(1)}(\xi,\sigma)$ for $|\xi|>\sigma/2$ and $P(\xi,\sigma)=P^{(2)}(\xi,\sigma)$ for $|\xi|<\sigma/2$, where
\begin{equation}\label{eq:P12}
P^{(1)}(\xi,\sigma):=
\begin{pmatrix}
1 & -\frac{i\,\sigma}{2\,\xi}\\
\frac{i\,\sigma}{2\,\xi}& 1
\end{pmatrix}\,, \quad 
P^{(2)}(\xi,\sigma):=
\begin{pmatrix}
1 & -\frac{2\,i\,\xi }{\sigma}\\
\frac{2\,i\,\xi }{\sigma}& 1
\end{pmatrix}\,.
\end{equation}

\item For $|\xi|=\sigma/2$ the matrix $C(\xi,\sigma)$ is defective. Then, due to~\cite[Lemma~4.3]{Arnold2014}, for any $\eps >0$ there exists an $\eps$-dependent matrix that yields the purely exponential decay $\mu(\sigma/2) - \varepsilon$. For later purposes it will be sufficient to investigate the case $\sigma = 2$ with $\xi=1$, see~\S~\ref{subsec:def}. Hence, we will not state the general form here.
\end{itemize}

\medskip\noindent{\bf Approach of~\S~\ref{Sec:Abstract1}}: With notation from Theorem~\ref{theo:DMS2015}, the Goldstein--Taylor equation in Fourier modes~\eqref{eq:transformedGT} can be written as
\begin{equation*}
\partial_t \hat{y}(t,\xi) = \big(\LL(\sigma) - \TT(\xi)\big)\,\hat{y}(t,\xi)\,.
\end{equation*}
The Hermitian collision matrix and the anti-Hermitian transport matrix are, respectively, given as
\begin{equation*}
\LL(\sigma) :=
\begin{pmatrix}
0& 0\\
0& -\,\sigma
\end{pmatrix}\,,\quad 
\TT(\xi):=
\begin{pmatrix}
0& i\,\xi\\
i\,\xi& 0
\end{pmatrix}\,.
\end{equation*}
The projection on the space of local-in-$x$ equilibria (satisfying $\LL(\sigma) \Pi=0$) is given by the matrix
\begin{equation*}
\Pi:=
\begin{pmatrix}
1~&0\\
0~&0
\end{pmatrix}\,.
\end{equation*}
We introduce the operator $\AA(\xi)$ as in~\eqref{A} for each mode $\xi$:
\begin{equation*}\label{eq:defAk}
\AA(\xi):= \Big(\mathrm{Id} + \big(\TT(\xi)\,\Pi\big)^*\,\TT(\xi)\,\Pi\Big)^{-1}\,\big(\TT(\xi)\,\Pi\big)^* = 
\begin{pmatrix}
0& -\frac{i\,\xi}{1 + \xi^2}\\
0&0
\end{pmatrix}\,.
\end{equation*}
The modal Lyapunov functional~\eqref{H} is given as
\begin{align*}
\HH_1(\xi,\delta)[\hat{y}(\xi)]&:=\frac12\,\|\hat{y}(\xi)\|^2 + \delta\, \re \big( \hat{y}(\xi)^*\, \AA(\xi)\,\hat{y}(\xi) \big)\\
&=\frac12\,\|\hat{y}(\xi)\|^2 + \delta \hat{y}(\xi)^* \AA_H(\xi)\,\hat{y}(\xi) \\
&=\frac12\,\hat{y}(\xi)^*\begin{pmatrix}
1& -\frac{i\,\xi\,\delta }{1+\xi^2}\\
\frac{i\,\xi\,\delta }{1+\xi^2}& 1
\end{pmatrix}\,\hat{y}(\xi)\,,\numberthis\label{eq:Hkdef} 
\end{align*}
where we denote the Hermitian part of the matrix $\AA$ by $\AA_H:= \frac{1}{2}(\AA+ \AA^*)$.

\section{Comparison of the two hypocoercivity methods\texorpdfstring{ for $\mathcal X = \T$}{}}\label{sec:T}
In the next step we shall assemble, for both hypocoercivity methods, the modal Lyapunov functionals to form a global one. When appropriately optimizing both of these functionals, we shall see that they actually coincide and achieve optimal decay estimates in the class of all quadratic forms.

\subsection{The optimal global Lyapunov functional}\label{subsec:functional}

We start by applying the strategies outlined in~\S~\ref{Sec:Abstract2} to assemble a global $\mathrm L^2(\T)$ functional. For simplicity, let us first assume that the matrix $C(\xi,\sigma)$, $\xi\in\Z$ of~\eqref{eq:transformedGT} is diagonalizable for all modes, \emph{i.e.}\ $\sigma\not\in2\,\Z$. A brief discussion of the defective cases is deferred to the end of this section. 

We first consider Strategy \hyperlink{strat1}{1} of~\S~\ref{Sec:Abstract2} that leads to the functional 
\begin{equation}
\HH_2[y] := \sum_{|\xi|>\sigma/2}\,\|\hat y(\xi)\|^2_{P^{(1)}(\xi)} + \sum_{|\xi|< \sigma/2}\,\|\hat y(\xi)\|^2_{P^{(2)}(\xi)} \,,\quad y\in (\mathrm L^2(\T))^2\,,
\end{equation}
according to definition~\eqref{def-H}. Assuming that the system has total mass $0$, \emph{i.e.}\\ $\int_\T u(x,0)\,dx = 0$, we obtain that solutions~$y(t)$ of~\eqref{eq:GTtrans},~\eqref{eq:transformedGT} satisfy the estimate:
\begin{equation}\label{eq:decayoc}
\|y(t)\|^2 \leq \overline{c}_P\,e^{-2\,\omu(\sigma)\,t}\,\|y(0)\|^2\,,
\end{equation}
where
\begin{equation}\label{eq:cbar} 
\bar c_P:=\max\left\{\sup_{|\xi|>\sigma/2} \left[\mbox{cond}\(P^{(1)}(\xi)\)\right], \sup_{|\xi|<\sigma/2} \left[\mbox{cond}\(P^{(2)}(\xi)\)\right]\right\}\,.
\end{equation}

To improve upon the multiplicative constant $\overline{c}_P$ in~\eqref{eq:decayoc}, we continue with Strategy~\hyperlink{strat2}{2} of~\S~\ref{Sec:Abstract2}.
\begin{itemize}
\item 
For the case $\sigma<2$ the functional $\HH_2$ and~\eqref{eq:cbar} directly yield the optimal multiplicative constant. With notation from~\S~\ref{Sec:Abstract2} this follows from $\Xi = \Z\setminus\{0\}$ and $ \overline{c}_P= c_\Xi=\mbox{cond}\big(P^{(1)}(\pm 1)\big) = (2 + \sigma)/(2-\sigma)$. In this case the two eigenvalues of the ODE system matrix $C(\xi)$ are distinct and form a complex conjugate pair. Hence, the multiplicative constant $\overline{c}_P$ is the optimal constant within the family of form~\eqref{eq:decayoc}, as has been shown in~\cite[Theorem 3.7]{AAS19}.\smallskip

\item 
For the case $\sigma>2$, $\sigma\not\in2\,\Z$, the lowest modes have the slowest decay: $\Xi = \{-1,1\}$ with $c_\Xi = \mbox{cond}\big(P^{(2)}(\pm 1)\big) = (\sigma + 2)/(\sigma - 2)$. The multiplicative constant $c_\Xi$ is not the smallest possible multiplicative constant in~\eqref{eq:decayoc}. However, according to~\cite[Theorem 4.1]{AAS19} it is the best possible multiplicative constant achievable by Lyapunov functionals that are quadratic forms. As $\sigma>2$ it follows that $\bar c_P > c_\Xi$, and hence we replace the functionals $\|\cdot\|^2_{\kern2pt P^{(1)}(\xi)}$ and $\|\cdot\|^2_{\kern2pt P^{(2)}(\xi)}$ for the faster  decaying modes $\xi\not\in\Xi$, $\xi\neq 0$. Let us define 
\begin{equation*}\label{eq:oP1}
\overline{P}^{(1)}(\xi):=\begin{pmatrix}
1& -\frac{2\,i }{\sigma\xi}\\
\frac{2\,i }{\sigma\xi}& 1
\end{pmatrix},\quad \xi\not\in\Xi, \xi\neq 0,
\end{equation*}
 and notice that $\overline{P}^{(1)}(\xi)$ satisfies the matrix inequality
\begin{equation}\label{eq:barPineq}
C^*(\xi)\,\overline{P}^{(1)}(\xi) + \overline{P}^{(1)}(\xi)\,C(\xi) \geq 2\,\omu\, \overline{P}^{(1)}(\xi)\,,\quad \xi\not\in \Xi\,,\;\xi\neq 0\, ,
\end{equation}
where $\overline{\mu}$ is the explicitly given uniform-in-$\Z$ spectral gap \eqref{eq:mubar1}.
Furthermore, as $\mbox{cond}\big(\overline{P}^{(1)}(\xi)\big) \leq \mbox{cond}\big(\overline{P}^{(1)}(\pm1)\big) = (\sigma+2)/(\sigma-2) $, it satisfies the estimate $\mbox{cond}\big(\overline{P}^{(1)}(\xi)\big) \leq c_{\Xi}(\sigma)$. Thus, the choice $\|\cdot\|_{\kern2pt\overline{P}^{(1)}(\xi)}$ for $\xi\not\in\Xi(\sigma)$ and $\xi\neq 0$ leads (via \eqref{def-H2}) to the global functional for $y\in (\mathrm L^2(\T))^2$, given as 
\[
\widetilde{\HH}_2[y]:= \sum_{\xi\in\Xi}\|\hat{y}(\xi)\|_{\kern2ptP^{(2)}(\xi)}^2 + \sum_{\xi\not\in\Xi, \xi\neq 0}\|\hat{y}(\xi)\|_{\kern2pt\overline{P}^{(1)}(\xi)}^2
= \sum_{\xi\in\Z\setminus\{0\}}\|\hat{y}(\xi)\|_{\kern2pt\overline{P}^{(1)}(\xi)}^2\,, 
\]
where the equality follows as $P^{(2)}(\pm1) = \overline{P}^{(1)}(\pm1)$.
$\widetilde{\HH}_2$ yields decay with sharp rate $2\,\omu(\sigma)$ given by~\eqref{eq:mubar1} and, within the family of quadratic forms, the optimal multiplicative constant $c_\Xi(\sigma)$ in~\eqref{eq:decayoc}.
\end{itemize}

In summary, for arbitrary $\sigma>0$, $\sigma\not\in2\,\Z$, \strattwo~of~\S~\ref{Sec:Abstract2} yields the global Lyapunov functional
\begin{equation}\label{eq:GTH2}
\widetilde{\HH}_2(\sigma)[y]:= \sum_{\xi\in\Z\setminus\{0\}}\|\hat{y}(\xi)\|_{\kern2pt\overline{P}(\xi,\theta(\sigma))}^2\,, \quad y\in (\mathrm L^2(\T))^2\,,
\end{equation}
where
\begin{equation}\label{eq:Ptheta}
\overline{P}(\xi, \theta):=\begin{pmatrix}
1& -\frac{i\,\theta }{2\,\xi}\\
\frac{i\,\theta}{2\,\xi}& 1
\end{pmatrix}\,, \quad \theta\pa{\sigma}:=
\begin{cases} \sigma\,, & 0<\sigma<2\,,\\
\frac{4}{\sigma}\,, & \sigma>2\,.
\end{cases}
\end{equation}

Next, we turn to the method of~\S~\ref{Sec:Abstract1} and derive another global $\mathrm L^2(\T)$ functional that is based on the modal functionals~\eqref{eq:Hkdef}:
\begin{equation*}
\HH_1(\delta)[y] := \sum_{\xi\in\Z\setminus\{0\}} \HH_1(\xi,\delta)[\hat{y}(\xi)]\,, \quad y\in (\mathrm L^2(\T))^2\,.
\end{equation*}
In~\cite[\S~1.4]{DMS-2part} the parameter $\delta\in(0,2)$ was chosen independent of $\xi$.
But optimizing the resulting decay rate of $\HH_1(\delta)$ w.r.t.\ the parameter $\delta\in(0,2)$ yields non-sharp decay rates (as derived in~\cite[\S~1.4]{DMS-2part} for $\lambda_m = \sigma=1$). Hence, we shall optimize here each modal functional $\HH_1(\xi,\delta(\xi))$ w.r.t.\ the parameter $\delta$.

For $y\in (\mathrm L^2(\T))^2$ and arbitrary $\sigma>0$, $\sigma\not\in2\,\Z$, the resulting functional is given as
\begin{equation}\label{eq:H1tilde}
\widetilde{\HH}_1(\sigma)[y] := 2\sum_{\xi\in\Z\setminus\{0\}} \HH_1\(\xi,\overline{\delta}(\xi,\sigma)\)[\hat{y}(\xi)]\,,
\end{equation}
with the optimal parameter $\overline{\delta}(\xi,\sigma):=\frac{\theta(\sigma)\,(1+\xi^2)}{2\,\xi^2}\in (0,2)$ and $\theta(\sigma)$ defined in~\eqref{eq:Ptheta}. The following theorem relates $\widetilde{\HH}_1$ to the previously defined functional $ \widetilde{\HH}_2$, given respectively by~\eqref{eq:H1tilde} and~\eqref{eq:GTH2}.
\begin{theorem}\label{thm:T} For $y\in (\mathrm L^2(\T))^2$ and arbitrary $\sigma>0$, $\sigma\not\in2\,\Z$ it follows that
\begin{equation*}
\widetilde{\HH}_1(\sigma)[y] = \widetilde{\HH}_2(\sigma)[y]\,.
\end{equation*}
\end{theorem}
\begin{proof} Thanks to previous considerations, the proof is now straightfoward.
For each mode $\xi\in\Z\setminus\{0\}$, the identity 
\begin{equation}\label{eq:modalfun}
2\,\HH_1\(\xi,\overline{\delta}(\xi,\sigma)\)[\hat{y}] = \|\hat{y}\|_{\kern2pt\overline{P}(\xi,\theta(\sigma))}^2\,,\quad \hat{y}\in\C^2
\end{equation}
follows by setting $\delta = \overline{\delta}(\xi,\theta)$ in~\eqref{eq:Hkdef}.
\end{proof}

\subsection{The defective cases}\label{subsec:def}
For $\sigma \neq 2$ the defective modes $|\xi| = \sigma/2$ do not exhibit the slowest decay of all modes, \emph{i.e.} $\xi \not\in\Xi$ as defined in~\S~\ref{Sec:Abstract2}. The functional $\|\cdot\|^2_{\kern2pt\overline{P}^{(1)}(\xi)}$ yields the sufficient decay rate $2\,\bar\mu(\sigma)$, along with multiplicative constants that are small enough, \emph{i.e.}, $\mbox{cond}\big(\overline{P}^{(1)}(\xi)\big) \leq c_{\Xi}(\sigma)$. It follows that Strategy \hyperlink{strat2}{2} of~\S~\ref{Sec:Abstract2} again yields the functional~$\widetilde{\HH}_2$ as defined in~\eqref{eq:GTH2}. 

The case $\sigma = 2$ is the only case where the defective modes correspond to the slowest modal decay, \emph{i.e.}\ $\xi = \pm 1\in\Xi$. Then, for arbitrarily small $\eps>0$ the modified norm $\|\cdot\|^2_{\kern2pt\overline{P}(\xi,\theta_\eps)}$, defined in~\eqref{eq:Ptheta}, with
\be{thetaeps}
\theta_\eps:=2\,\frac{2-\varepsilon^2}{2+\varepsilon^2}\,,
\ee
yields the exponential decay rate $2\,\big(\omu(2)-\eps\big)$ for all modes $|\xi|\neq 0$. Due to the lack of an eigenvector basis in the defective case, constructing the matrix $\overline{P}(\xi,\theta_\eps)$ results in a decay estimate of form~\eqref{eq:decayoc} with multiplicative constant $c_\eps = \sqrt{2}/\eps$. The blow-up $\lim_{\eps \to 0_+}c_\eps=+\infty$ reflects the fact that the true decay behaviour of solutions in this defective setting is not purely exponential with rate $2\,\overline{\mu}(2)$, but rather exponential times a polynomial in time $t$. An approach based on more involved time-dependent Lyapunov functionals yields estimates with the sharp defective decay behaviour. As the time-dependent construction is besides our focus, we simply refer to~\cite{AJW} for further details.

\subsection{Decay results for \texorpdfstring{the case $\mathcal X = \T$}{the Torus}}
In this subsection we start by refining the general strategy of~\S~\ref{Sec:Mode-by-mode} to extract the sharp decay rate for the GT model from the functional $\widetilde{\HH}_1$. Subsequently, we conclude the torus case by expressing the global Lyapunov functional in the spatial variable.

In Theorem~\ref{thm:T} above, we establish that both functional constructions (as described in~\S~\hyperref[Sec:Abstract1]{I}) coincide for the GT equation if one chooses the appropriate parameter $\delta(\xi)$ for $\widetilde{\HH}_1$. Now, we compare both approaches of~\S~\hyperref[Sec:Abstract1]{I} to extract explicit decay rates from the functional.

The \strattwo\, of~\S~\ref{Sec:Abstract2} for $\widetilde{\HH}_2$ is based on modal matrix inequalities,~\eqref{eq:barPineq}, which prove the sharp global decay rate $2\,\overline{\mu}$ as already discussed in~\S~\ref{subsec:functional}.

The general method of~\S~\ref{Sec:Abstract1} for $\widetilde{\HH}_1$ (and its improvements of~\S~\ref{part:II}) is to estimate the entropy -- entropy production inequality $\mathrm D[y]-\lambda\,\widetilde{\HH}_1[y]\geq 0$ in terms of $\|(\mathrm{Id}-\Pi)\,y\|$ and $\|\Pi y\|$ that are then optimized for $\lambda$. As assumed in~\S~\ref{part:II}, we restrict our discussion to $\lambda_m = 1$ which requires the relaxation rate $\sigma = 1$. Applying Proposition~\ref{Prop:tilde} to the modal equation~\eqref{eq:transformedGT} for $\xi=\pm1$ yields the decay estimates
$$\HH_1\(\pm 1,\tilde{\delta}_2(1)\)[\hat{y}(\pm1,t)] \leq e^{-\tilde{\lambda}(1)\,t}\,\HH_1\(\pm 1,\tilde{\delta}_2(1)\)[\hat{y}(\pm1,0)]$$
with non-optimal modal decay rate $\tilde{\lambda}(1) \approx 0.165$ and parameter $\tilde{\delta}_2(1) \approx 0.325$. Higher modes, $|\xi|>1$, yield higher decay rates (\emph{cf.}\ Lemma~\ref{Lem:tilde}), but as
$$\inf_{\xi\in\Z\setminus\{0\}} \tilde{\lambda}(|\xi|) = \tilde{\lambda}(1) < 2\,\overline{\mu}(1) = 1\,,$$
the optimal global rate cannot be recovered. One cause for not reaching the sharp rate is that Proposition~\ref{Prop:tilde} approximates the entropy - entropy production inequality condition to obtain readable formulas (via the discriminant $\tilde{h}_2$ as defined in~\eqref{eq:h2t}). But even omitting approximations when optimizing $\delta$ does not yield sharp decay rates $\lambda(|\xi|)=1$ for our example. 

This is not surprising as~\S~\ref{Sec:Further} provides explicit estimates with a general hypocoercive setting in mind. In order to obtain sharp decay rates for the GT model, we sacrifice this generality and refine the strategy for the simple structure at hand. The reduction of the continuous velocity space $v\in\R$ (as defined in~\S~\ref{Sec:Intro}) to two discrete velocities in the GT setting allows the following modifications: With the notation of~\S~\ref{sec:GTgen} it holds that 
\begin{equation*}
\AA(\xi)\,\TT(\xi)(\mathrm{Id}-\Pi)\hat{y} = 0\,, \quad \hat{y}\in\C^2\,.
\end{equation*} 
Thus, the constant $C_M$ as defined in~\eqref{Ex} improves to $C_M = \frac{|\xi|}{1+|\xi|^2}$. Additionally, as
\begin{equation*}
\AA(\xi)\,(\LL+\lambda\,\mathrm{Id}) = 
\begin{pmatrix}
0& \frac{i\,\xi\,(1-\lambda)}{1+\xi^2}\\
0 & 0
\end{pmatrix}\,,
\end{equation*}
it follows that
$$\Big|\re\big\langle \AA(\xi)\,(\LL+\lambda\,\mathrm{Id})\hat{y}(\xi)\,, \hat{y}(\xi) \big\rangle\Big| \leq \frac{|\xi|}{1+|\xi|^2}\,(1-\lambda)\,X\,Y\,,$$
for $0\leq \lambda\leq 1$, where $X:=\|(\mathrm{Id}-\Pi)\hat{y}\|= \|v\|$ and $Y:=\|\Pi\hat{y}\|=\|u\|.$
Then, as a refinement of~\eqref{eq:Q2} for the GT equation with $\HH_1(\xi,\delta)$ from~\eqref{eq:Hkdef} it follows that
\begin{multline*}
\mathsf D(\xi,\delta)[\hat{y}] - \lambda\,\HH_1(\xi,\delta)[\hat{y}] \\
\ge \left(1 - \frac{\delta\,\xi^2}{1+\xi^2}- \frac{\lambda}{2}\right)\,X^2 - \delta\,\re \langle \AA\,(\LL-\lambda\,\mathrm{Id})\,F, F\rangle + \left(\frac{\lambda\,\xi^2}{1+\xi^2}-\frac{\lambda}{2} \right)\,Y^2 \\
\ge
\left(1 - \frac{\delta\,\xi^2}{1+\xi^2}- \frac{\lambda}{2}\right)\,X^2 - \frac{\delta\,|\xi|\,(1-\lambda)}{1+\xi^2}\,X\,Y+ \left(\frac{\lambda\,\xi^2}{1+\xi^2}-\frac{\lambda}{2} \right)\,Y^2\,.
\end{multline*}
The \emph{refined discriminant condition} is then given by the non-positivity of
\begin{equation*}
h_{\textrm{GT}}(\delta,\lambda):= \frac{\delta^2\,\xi^2}{(1+\xi^2)^2}\,(1-\lambda)^2 - 4\left(1-\frac{\delta\,\xi^2}{1+\xi^2}-\frac{\lambda}{2}\right)\left(\frac{\delta\,\xi^2}{1+\xi^2}-\frac{\lambda}{2}\right).
\end{equation*}
It can be verified directly that $\overline{\delta}(\xi): = \frac{1+\xi^2}{2\,\xi^2}$ for $\xi\neq 0$ yields $h_{\textrm{GT}}\big(\overline{\delta}(\xi),1\big)=0$. Hence we recover the sharp exponential decay rate $\lambda(|\xi|) = 2\,\overline{\mu}(1) = 1$ for the modal equations~\eqref{eq:transformedGT} for all $\xi\in\Z\setminus\{0\}$ with $\sigma=1$.

With this we have shown that refining the method of \S~\ref{Sec:Mode-by-mode} for the GT model (with $\sigma = 1$) allows us to recover the sharp global decay rate from the global functional $\widetilde{\HH}_1$, as defined in \eqref{eq:H1tilde}.

In~\S~\ref{sec:T} above we show that both hypocoercive methods from~\S~\hyperref[Sec:Abstract1]{I} lead to the same global Lyapunov functional for arbitrary $\sigma>0$. We conclude this subsection by leaving the modal formulation behind and expressing this global functional in the spatial variable. 

In~\cite{AESW} the authors define an explicit spatial Lyapunov functional that yields the sharp, purely exponential decay rates and best possible multiplicative constant (reachable via quadratic forms) for each $\sigma>0$:
\begin{definition}\label{def:entropy}
Let $u$, $v\in \mathrm L^2\pa{\T}$ be real-valued and let $\theta \in (0,2)$ be given. We then define the functional $\mathsf E_\theta[u,v]$ as
\begin{equation*}\label{eq:def_entropy}
\mathsf E_\theta[u,v]:=\norm{u}^2_{\mathrm L^2(\T)}+\norm{v}^2_{\mathrm L^2(\T)} -\frac{\theta}{2\pi}\int_{0}^{2\pi} v\,\partial_x^{-1}u\,dx\,.
\end{equation*}
Here, the {\it anti-derivative} of $u$ is defined as
\begin{equation}\label{eq:antiderivative}
\partial_x^{-1}u (x):=\int_{0}^{x}u\,dy - \pa{\int_{0}^{x}u(y)\,dy }_{\mathrm{avg}}\,,
\end{equation}
where $u_{\avg}:= \frac{1}{2\pi} \int_0^{2\pi}u\,dx=\hat{u}(0).$
\end{definition}
\begin{theorem}
For $u$, $v\in \mathrm L^2(\T)$ and arbitrary $\sigma>0$, $\sigma\not\in2\,\Z$ it follows that
\begin{equation}
\widetilde{\HH}_1(\sigma)\left[u-u_{\avg},v\right] = \mathsf E_{\theta(\sigma)}\left[u-u_{\avg},v\right]\,,
\end{equation}
with $\widetilde{\HH}_1$ defined in~\eqref{eq:H1tilde}.
\end{theorem}
\begin{proof} As in~\cite[\S~4.3]{AESW}, we use Parseval's identity and the fact that $(i\,k)^{-1}$ is the (discrete) Fourier symbol of $\ip$ as defined in~\eqref{eq:antiderivative}. For the total entropy of arbitrary $y: = (u,v)^T\in (\mathrm L^2(\T))^2$, with $u_{\avg} = 0$, we deduce from~\eqref{eq:modalfun} that
\begin{align*}
\widetilde{\HH}_1(\sigma)[y] + \|\hat{v}(0)\|^2 &= \sum_{k\in\Z\setminus\{0\}}\|\hat{y}(\xi)\|_{\kern2pt\overline{P}(\xi,\theta(\sigma))}^2 +\|\hat{v}(0)\|^2 \\
&= \frac{1}{2\pi}\int_0^{2\pi} \Big(|u|^2 + |v|^2 -\theta(\sigma)\,v\,\ip u\Big)\,dx\label{eq:spatialE}\\
&= \mathsf E_{\theta(\sigma)}[u,v]\,.
\end{align*}~
\end{proof}

In the following result we recall from~\cite[Theorem 2.2.a]{AESW} the optimal exponential decay for $y(t)$ to the steady state $ y_\infty = (u_{\avg}, 0)^T $, both in the functional $ \mathsf E_\theta $ and in the Euclidean norm. {\it Mild solution} refers to the terminology of semigroup theory~\cite{Pa92}.
\begin{theorem}\label{thm:AESW}
Let $(u,v)\in C([0,\infty);(\mathrm L^2\pa{\T})^2)$ be a mild real valued solution to~\eqref{eq:GTtrans} with initial datum $u_0$, $v_0\in \mathrm L^2\pa{\T}$ and define $u_{\mathrm{avg}}:=\frac{1}{2\pi} \int_0^{2\pi} u_0(x)\,dx$.\\
$\bullet$ If $\sigma\not=2$ then
$$\mathsf E_{\theta(\sigma)}[u(t)-u_{\mathrm{avg}},v(t)] \leq \mathsf E_{\theta(\sigma)}[u_0-u_{\mathrm{avg}},v_0]\,e^{-2\,\mu\pa{\sigma}\,t}\quad\forall\,t\ge0\,,$$
where 
$$\theta\pa{\sigma}:=\begin{cases} \sigma\,, & 0<\sigma<2 \\ \frac{4}{\sigma}\,, & \sigma>2 \end{cases}\,, \quad \mu\pa{\sigma}:=\begin{cases} \frac{\sigma}{2}\,, & 0<\sigma<2 \\ \frac{\sigma}{2}-\sqrt{\frac{\sigma^2}{4}-1}\,, & \sigma>2 \end{cases}\,.$$
Consequently we obtain the decay estimate
\begin{equation*}\label{eq:main_for_f_pm_constant}
\begin{gathered}
\left\|f(t)-\binom{f_\infty}{f_\infty}\right\|_{\mathrm L^2(\T)} \leq \mathscr{C}_{\sigma} \left\|f_0-\binom{f_\infty}{f_\infty}\right\|_{\mathrm L^2(\T)}\,e^{-\mu\pa{\sigma}\,t}\quad\forall\,t\ge0\,,
\end{gathered}
\end{equation*}
where the decay rate $\mu(\sigma)$ is sharp and
\begin{equation*}\label{eq:C_sigma_and_f_infty}
\mathscr{C}_{\sigma}:=\sqrt{\frac{2+\sigma}{|2-\sigma|}}\,,\quad f(t):=\binom{f_+(t)}{f_-(t)}\,, \quad f_{\infty}=\frac12\,u_{\mathrm{avg}}\,,\quad f_0:=\binom{f_{+,0}}{f_{-,0}}\,.
\end{equation*}
$\bullet$   If $\sigma=2$ then for any $0<\varepsilon<1$
$$\mathsf E_{\theta_\eps}[u(t)-u_{\mathrm{avg}},v(t)] \leq \mathsf E_{\theta_\eps}[u_0-u_{\mathrm{avg}},v_0]\,e^{-2\,\pa{1-\varepsilon}\,t}\quad\forall\,t\ge0$$
with $\theta_\eps$ defined as in~\eqref{thetaeps} and we have that 
\begin{equation*}\label{eq:main_for_f_pm_constant_ep}
\begin{gathered}
\left\|f(t)-\binom{f_\infty}{f_\infty}\right\|_{\mathrm L^2(\T)} \leq \frac{\sqrt2}{\varepsilon}\, \left\|f_0-\binom{f_\infty}{f_\infty}\right\|_{\mathrm L^2(\T)}\,e^{-\pa{1-\varepsilon}\,t}\quad\forall\,t\ge0\,.
\end{gathered}
\end{equation*}
\end{theorem}

\section{Decay results for \texorpdfstring{the case $\mathcal X=\R$}{the real line}} \label{sec:R}

We consider the GT model with position $x$ on the real line and prove two global decay estimates with sharp algebraic rate. Our first goal is to obtain modal decay estimates of general form~\eqref{eq:lineest} with modal constants $C(|\xi|)$ as small as possible. As we discuss below, a straightforward application of Lemma~\ref{Lem:hypoNash} with \stratone~of~\S~\ref{Sec:Abstract2} is not possible due to the appearance of a defective eigenvalue in the modal equation~\eqref{eq:transformedGT}. To avoid this difficulty we shall use a non-sharp decay estimate as input to apply Lemma~\ref{Lem:hypoNash}. Our second goal is to construct a simple spatial functional that closely approximates our first result. To achieve this, we construct modal Lyapunov functionals that yield slightly less precise estimates but have the advantage of representing a more convenient pseudo-differential operator.
\medskip

To simplify notation, we assume that $\sigma = 1$ in the present section. This is no restriction, as the general case $\sigma>0$ in~\eqref{eq:GT} can always be reduced to the normalized one thanks to the rescaling $\tilde{t}= \sigma\,t$, $\tilde{x} = \sigma\,x$.

A natural approach to obtain a decay estimate for $\mathcal X = \R$ is an application of Lemma~\ref{Lem:hypoNash} to the decay estimates~\eqref{eq:Pdecay} with the matrices $P^{(1)}$ and $P^{(2)}$ of~\eqref{eq:P12} for $\sigma = 1$. The extension of Strategy \hyperlink{strat1}{1} of~\S~\ref{Sec:Abstract2} to $\xi\in\R$ leads to~\eqref{eq:decayoc}. But as $\operatorname{cond}\big(P^{(1)}(\xi)\big)\to \infty$ and $\operatorname{cond}\big(P^{(2)}(\xi)\big)\to \infty$ for $|\xi|\to 1/2$, it follows that the multiplicative constants in~\eqref{eq:decayoc} become unbounded. This is due to the defective limit of the modal equation~\eqref{eq:transformedGT} at $|\xi|=1/2$. The modal Lyapunov functionals with sharp rate depend on the eigenspace structure, which has a discontinuity at $|\xi|=1/2$ and the decay rates are not purely exponential there. Hence, we cannot directly use Lemma~\ref{Lem:hypoNash} with sharp rates.

Therefore, the natural and in fact sharper approach is to start with the exact modal decay function (instead of an exponential approximation): for $2\times2$ ODE systems, this decay function was given in~\cite[Proposition~4.2]{AAS19}:
\begin{equation}\label{eq:hplus}
\|\hat y(t,\xi)\|_2^2 \le h_+(t,\xi)\,\|\hat y(0,\xi)\|_2^2 \quad\forall\,t\ge0\,,
\end{equation}
where $h_+(t,\xi)$, the squared propagator norm associated with~\eqref{eq:transformedGT}, is explicitly given in~\cite{AAS19}. Since this function is continuous at the defective point $\xi=1/2$ for all $t\ge0$ (see Fig.~\ref{fig:hplus}), one could easily extend Lemma~\ref{Lem:hypoNash} to this setting. But, since $h_+(t,\xi)$ is a quite involved function, the minimization w.r.t.\ $R$ (as in~\eqref{Psi}) could only be carried out numerically. In order to come up with an explicit decay estimate, we shall therefore rather approximate the modal (exponential) decay estimates that are used as a starting point for Lemma~\ref{Lem:hypoNash}.
\begin{figure}[ht]
\begin{center}
\hspace*{-10pt}\begin{picture}(12,7.5)
\put(2.05,0.13){\includegraphics[width=8cm]{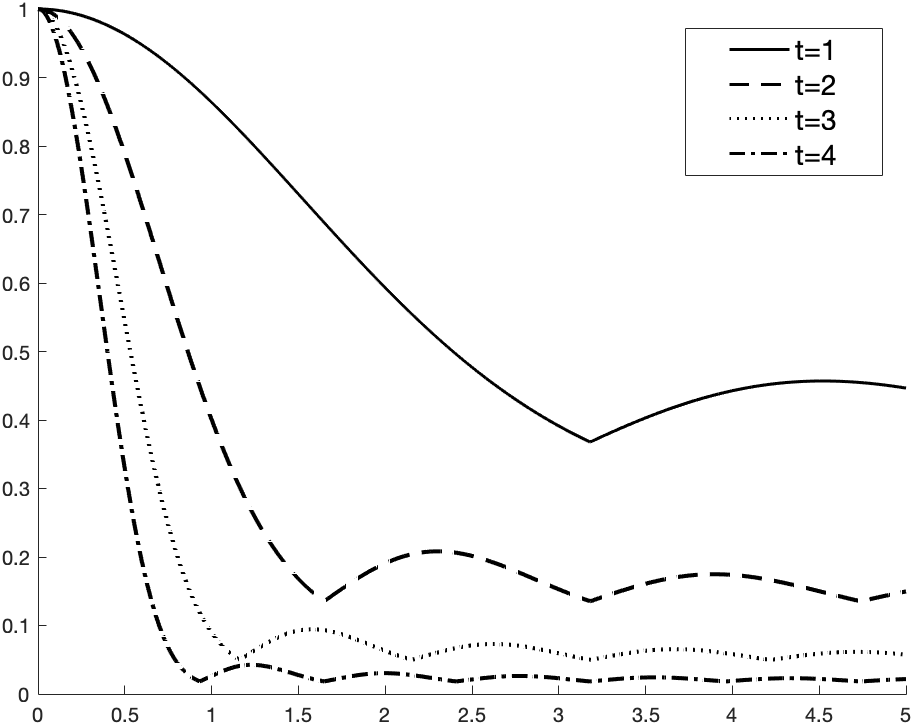}}
\put(10.5,0.7){$s$}
\put(2.3,7.2){$h_+(t,s)$}
\thicklines
\put(2.35,0.38){\vector(1,0){8.25}}
\put(2.37,0.3){\vector(0,1){6.6}}
\end{picture}
\caption{The mapping $|\xi|=s\mapsto h_+(t,s)$ shows the continuous modal dependency of the squared propagator norm of $C(\xi,1)$ for fixed times $t$. Note that the kinks are no numerical artefact.}
\label{fig:hplus}
\end{center}
\end{figure}
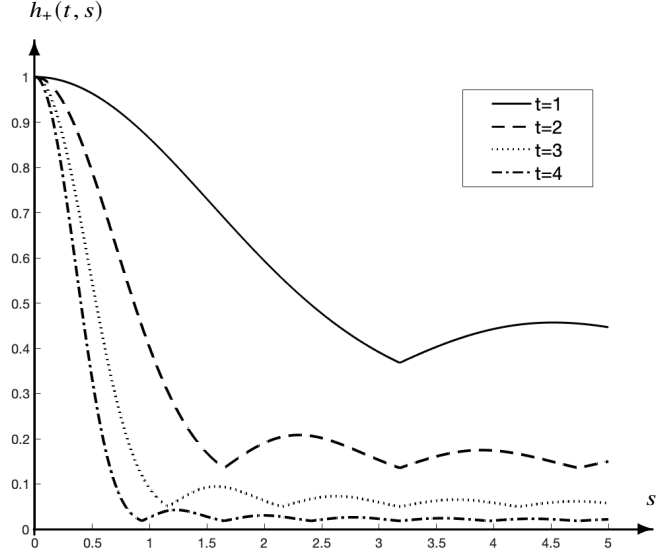
We now approximate the decay estimate for large frequencies $|\xi|$, but keep the sharp estimates for $|\xi|$ small. 
\begin{lemma}\label{lem:PR}
Assume that $R\in(0,1/2)$ and let $\hat{y}(t,\xi)$ be a solution to~\eqref{eq:transformedGT}. Then,
\begin{equation}\label{eq:GTReuclid}
\|\hat{y}(t,\xi)\|^2 \leq c(\xi)\,e^{-\lambda(\xi)\,t}\,\|\hat{y}(0,\xi)\|^2\,, \quad\forall\,\xi\in\R\setminus\{0\}\,,\quad \forall\,t\geq 0\,,
\end{equation}
with $c(\xi) =
\begin{cases}
\frac{1+2\,|\xi|}{1-2\,|\xi|}\,,& |\xi|<R\,,\\
\frac{|\xi| + 2\,R^2}{|\xi|-2\,R^2}\,, & |\xi|\geq R\,,
\end{cases}
\quad
\lambda(\xi) =
\begin{cases}
2\,\mu(\xi)\,,& |\xi|<R\,,\\
2\,\mu(R)\,, & |\xi|\geq R\,.
\end{cases}$
\end{lemma}
\begin{proof} For every $|\xi| < R$, the modal functional $\|\cdot\|^2_{P^{(2)}(\xi)}$, as defined in~\eqref{eq:P12}, yields the sharp modal decay
\[
2\,\mu(\xi):= 2\,\mu(\xi,1)=1-\sqrt{1-4\,\xi^2}\,,
\]
given by~\eqref{eq:spectralgap}. The condition number of $P^{(2)}(\xi)$ is given as 
\begin{equation}\label{eq:CbelowR}
c(\xi):=\operatorname{cond}\(P^{(2)}(\xi)\) = \frac{1+2\,|\xi|}{1-2\,|\xi|} \leq \frac{1 + 2\,R}{1-2\,R}\,, \quad |\xi|<R\,.
\end{equation}

For $|\xi|\geq R$ we use the rescaled version of $\|\cdot\|^2_{\kern2pt\overline{P}^{(1)}(\xi)}$ (from \S~\ref{eq:oP1}, but now for $\mathcal X = \R$), given as $\|\cdot\|^2_{\kern2pt\overline{P}(\xi)}$, with the matrix
\begin{equation}
\overline{P}(\xi) : =
\begin{pmatrix}
1 & -\frac{2\,i\,R^2}{\xi}\\
\frac{2\,i\,R^2}{\xi}& 1
\end{pmatrix}\,.
\end{equation}
As this matrix satisfies the inequality
\begin{equation}
C^*(\xi) \overline{P}(\xi) + \overline{P}(\xi) C(\xi) \geq 2\mu(R) \overline{P}(\xi),\quad |\xi|\geq R,
\end{equation}
the functional $\|\cdot\|^2_{\overline{P}(\xi)}$ yields an exponential decay $2\,\mu(R)$ for all modes $|\xi|\geq R$.
The condition number of $\overline{P}(\xi)$ is given as:
\begin{equation}\label{eq:CaboveR}
c(\xi):=\operatorname{cond}\(\overline{P}(\xi)\) = \frac{|\xi| + 2\,R^2}{|\xi|-2\,R^2}\leq \operatorname{cond}\(\overline{P}(R)\) = \frac{1 + 2\,R}{1-2\,R}\,, \quad |\xi |\geq R\,,
\end{equation}
from which the desired result follows.
\end{proof}
We can now apply Lemma~\ref{Lem:hypoNash} and Lemma~\ref{lem:PR} to obtain following global decay estimate.
\begin{proposition}\label{prop:GTR}
Let $y:=(u,v)^T$ be a solution of the Goldstein--Taylor equation~\eqref{eq:GTtrans} on $\R$ with $\sigma=1$ and initial datum $y_0:=(u_0,v_0)^T$, such that $u_0$, $v_0 \in \mathrm L^1(\R)\cap \mathrm L^2(\R)$. Let the modal spectral gap, defined in~\eqref{eq:spectralgap}, be denoted as $\mu(\xi):= \mu(\xi,1)$. Then, for any $t\geq 0$ it follows that
\[
\|y(t)\|^2_{\mathrm L^2(\R)} \leq \inf_{0<R<\frac12} \frac{1 + 2\,R}{1-2\,R} \left(2 \min\left\{\frac{B(t,R)}{\sqrt{t}}, R\right\} \, \|y_0\|_{\mathrm L^1(\R)}^2 + e^{-\,2\,\mu(R)\,t}\,\|y_0\|^2_{\mathrm L^2(\R)}\right)
\]
with $B(t,R): = \sqrt{t} \int_0^R e^{-2\,\mu(s)\,t} ds\in \big[0,\sqrt{\pi/8}\,\big)$.
\end{proposition}
\begin{proof}
Applying Lemma~\ref{Lem:hypoNash} to the modal decay estimates~\eqref{eq:GTReuclid} and taking into account the estimates~\eqref{eq:CbelowR} and~\eqref{eq:CaboveR} leads to
the decay result where, for $B(t,R)$, we use the estimate $\mu(|\xi|)/\xi^2 = \big(1/2-\sqrt{1/4- \xi^2}\,\big)/\xi^2 \geq 1$ for $0\leq |\xi|\leq 1/2$. The bound on $B$ follows from
$B(t,R)\leq \sqrt{t}\int_0^\infty e^{-2\,\xi^2\,t}\,d\xi$.\qed\end{proof}
\begin{remark}
The decay result of Proposition~\ref{prop:GTR} is neither explicit in the optimization with respect to $R$ (for fixed $t$), nor optimal, as this would require an approach starting from~\eqref{eq:hplus}. It is however the best possible estimate of form~\eqref{eq:lineest} achievable with quadratic forms for each mode. This follows, as for one, the modal functionals for $|\xi|\leq R$, $\xi\neq 0$ are optimal for quadratic forms (\emph{cf.}\ the discussion on $P^{(2)}(\xi)$ in \S~\ref{subsec:functional}). Additionally, the modal functionals for $|\xi|\geq R$ are sufficient (in light of Lemma~\ref{Lem:hypoNash}) as they yield the sufficient decay $2\,\overline{\mu}(R)$ and the sufficient multiplicative constants $\sup_{|\xi|\geq R} c(\xi) = \sup_{|\xi|< R} c(\xi)=(1+2\,R)/(1-2\,R) $. In analogy to~\S~\ref{Sec:Abstract2} the decay stated in Proposition~\ref{prop:GTR} results from the global functional
\[
\widehat{\HH}_2(R)[y]: = \int_{(-R,R)} \|\hat{y}(\xi)\|^2_{P^{(2)}(\xi)}\,d\xi+ \int_{|\xi|\geq R} \|\hat{y}(\xi)\|^2_{\kern2pt\overline{P}(\xi)}\,d\xi\,.
\]
\end{remark}
\medskip

As our final result, we shall consider an alternative modal functional for the GT equation on $\R$ that translates into a convenient representation in the spatial variable. The trade-off is a less accurate global decay estimate.

The result of Proposition~\ref{prop:GTR} was based on the modal Lyapunov functional $\|\cdot\|^2_{\kern2pt\overline{P}(\xi)}$ for large modes and $\|\cdot\|^2_{P^{(2)}(\xi)}$ for small modes. Now, we replace both functionals by the single norm $\|\cdot\|_{\tilde{P}(\xi)}^2$ with the positive definite Hermitian matrix
\begin{equation}\label{eq:Ptilde}
\tilde{P}(\xi) := 
\begin{pmatrix}
1 & -\frac{2\,i\,\xi}{1 + 4\,\xi^2}\\
\frac{2\,i\,\xi}{1 + 4\,\xi^2} & 1
\end{pmatrix}\,,\quad \xi\neq 0\,,
\end{equation}
which asymptotically approximates the matrices from~\eqref{eq:P12} which yield sharp modal decay. For the off-diagonal matrix elements we have
\begin{align*}
\tilde{P}_{12}(\xi) - P^{(1)}_{12}(\xi) &= o\(P^{(1)}_{12}(\xi)\)\quad \text{ as }\quad |\xi| \to+\infty\,,\\
\tilde{P}_{12}(\xi) - P^{(2)}_{12}(\xi)&= o\(P^{(2)}_{12}(\xi)\)\quad\text{ as }\quad |\xi| \to 0\,.
\end{align*}
It satisfies the matrix inequality~\eqref{matrixestimate1} with $P=\tilde{P}(\xi)$ and the spectral gap $\mu$ replaced~by
$$\tilde{\mu}(\xi)=\frac12\left(1-\frac{1}{\sqrt{1+4\,\xi^2\,(1+4\,\xi^2)}}\right)\,.$$
The rate $2\,\tilde{\mu}(\xi)$ is an approximation to the sharp decay rate $2\,\mu(\xi)$ of fifth order for modes $\xi$ close to 0, see Fig.~\ref{fig:mutilde}.
\begin{figure}[ht]
\begin{center}
\hspace*{-10pt}\begin{picture}(12,9)
\put(2,0){\includegraphics[width=8cm]{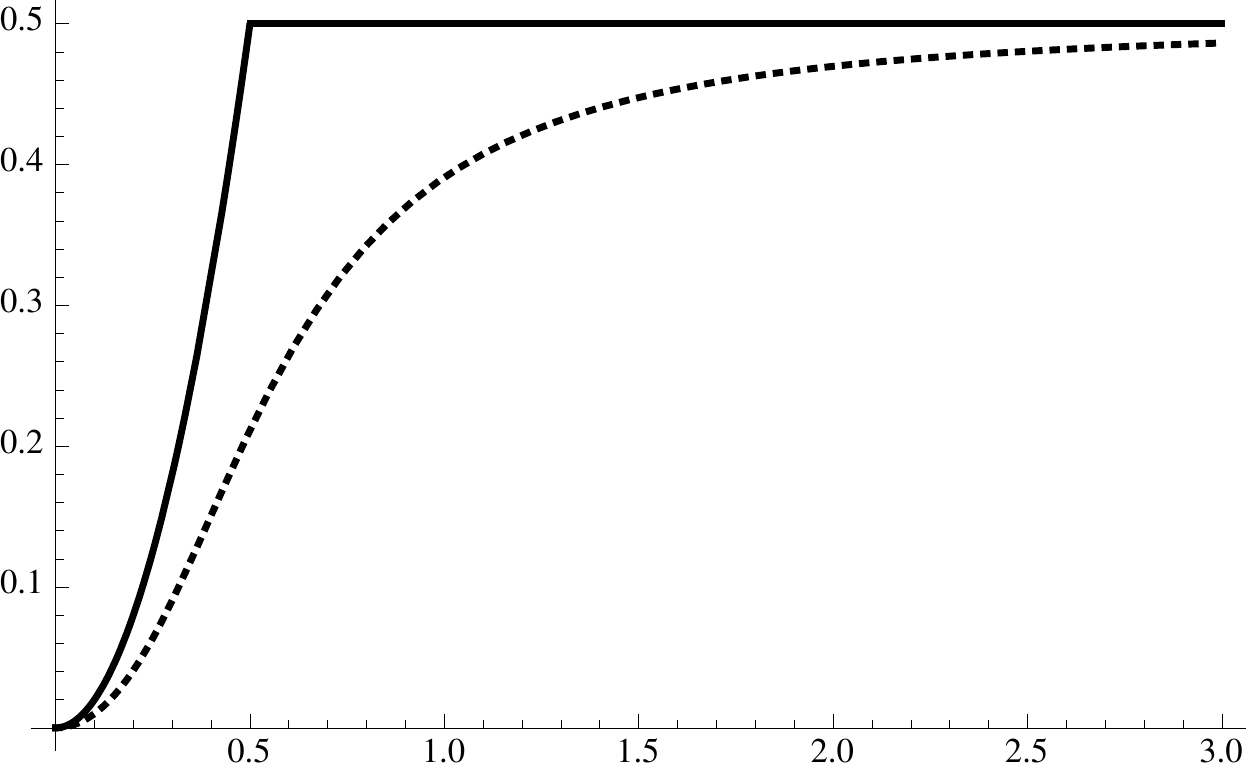}}
\put(10.5,0.7){$s=|\xi|$}
\put(2.75,4){$\mu(s)$}
\put(4.35,3){$\tilde\mu(s)$}
\thicklines
\put(2.1,0.33){\vector(1,0){8.5}}
\put(2.36,0.1){\vector(0,1){5.6}}
\end{picture}
\caption{\label{fig:mutilde} Exponential decay rate $\tilde{\mu}$ in comparison to the sharp exponential rate $\mu$, shown as functions of the spatial frequency $s=|\xi|$.}
\end{center}
\end{figure}
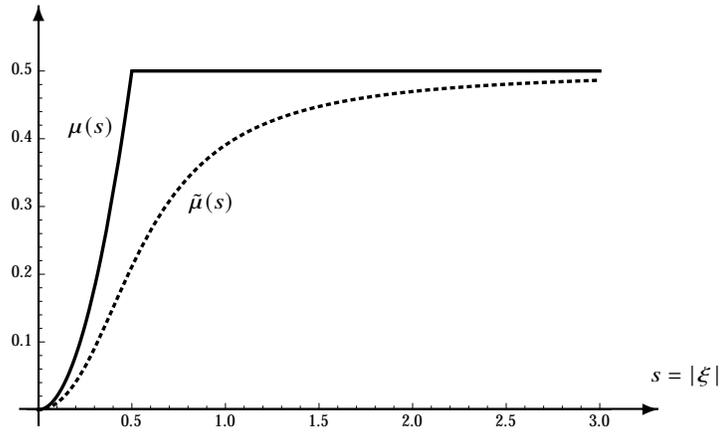
The condition number of $\tilde{P}(\xi)$ is given by
\begin{equation}\label{eq:condPtilde}
\tilde{c}(\xi):=\operatorname{cond}\(\tilde{P}(\xi)\):= \frac{1 + 2\,|\xi|+4\,\xi^2}{1 - 2\,|\xi| + 4\,\xi^2}\,,
\end{equation}
and hence we arrive at the modal decay estimates for $\xi\neq 0$:
\begin{equation}\label{eq:GTReuclidtilde}
\|\hat{y}(t,\xi)\|^2 \leq \tilde{c}(\xi)\,e^{-2\,\tilde{\mu}(\xi)\,t}\,\|\hat{y}(0,\xi)\|^2\,, \quad t\geq 0\,.
\end{equation}
We define the global Lyapunov functional
\begin{equation*}\label{eq:H3}
\HH_3[y]:= \int_\R \|\hat{y}(\xi)\|^2_{\tilde{P}(\xi)}\,d\xi\,.
\end{equation*}
As we shall see now, this can be rewritten in $x$-space (without resorting to the $\xi$-modes) in terms of a fairly simple pseudo-differential operator, similar to the functional $E_\theta[y]$ from Definition~\ref{def:entropy}. Moreover, it is easily related to the functional $\HH_1[y]$ from \S~\ref{Sec:Abstract1}: on the symbol level it holds that functional $\HH_3(\xi) = 2\,\HH_1(2\,\xi,\delta = 1)$, see~\eqref{eq:Ptilde},~\eqref{eq:Hkdef}.
\begin{proposition}~
\begin{enumerate}[label=\alph*)]
\item 
For $u$, $v\in \mathrm L^2(\R)$, the functional $\HH_3$ can be expressed as
\begin{equation*}
\HH_{3}[u,v]=\|u\|^2_{\mathrm L^2(\R)} + \|v\|^2_{\mathrm L^2(\R)} -4 \int_\R u(x) \,\partial_x\,\big(1 - 4\,\partial_x^2\big)^{-1}\,v(x)\,dx\,.
\end{equation*}
\item
Let $y:=(u,v)^T\in C\big([0,\infty);(\mathrm L^2\pa{\R})^2\big)$ be a mild real valued solution to~\eqref{eq:GTtrans} with $\sigma =1$ and initial datum $u_0,\, v_0\in \mathrm L^1(\R)\cap \mathrm L^2(\R)$.
Then, the functional $\HH_3$ yields the decay estimate 
\begin{align*}
&\|y(t,x)\|_{\mathrm L^2(\R)}^2\\
&\leq\inf_{0<R\leq \tfrac{\sqrt{5}-1}{4}} \left(\tfrac{1 + 2\,R + 4\,R^2}{1 -2\,R + 4\,R^2}\min\left\{2\,R, \sqrt{\tfrac{\pi}{2\,t}} \right\}\|y_0\|_{\mathrm L^1(\R)}^2 + 3\,e^{-2\,\tilde{\mu}(R)\,t}\,\|y_0\|_{\mathrm L^2(\R)}^2\right)\,.
\end{align*}
\end{enumerate}
\end{proposition}
\begin{proof}With Plancherel's identity it follows that
\begin{align*}
\HH_3[y]&=\int_\R \|\hat{y}(\xi)\|^2_{\tilde{P}(\xi)} \,d\xi= \|u\|^2_{\mathrm L^2(\R)} + \|v\|^2_{\mathrm L^2(\R)} + 2\,\re \(\int_{\R} 2\,i\,\xi\,\hat{u}(\xi)\,\frac{\overline{\hat{v}(\xi)}}{1+4\,\xi^2} \,d\xi\)\\
&=\|u\|^2_{\mathrm L^2(\R)} + \|v\|^2_{\mathrm L^2(\R)} -4 \int_\R  u(x)\,\partial_x\big(1 -4\, \partial_x^2\big)^{-1}v(x)\,dx\,.
\end{align*}
To prove the decay estimate, we apply Lemma~\ref{Lem:hypoNash} to~\eqref{eq:GTReuclidtilde}. The multiplicative constant~$\tilde{c}(\xi)$ from~\eqref{eq:condPtilde} is monotonously increasing for $\xi\in[0,1/2]$ to its global maximum $\operatorname{cond}\(\Pt(1/2)\)= 3$. For the integral in \eqref{Psi} with $\tilde{c}(\xi)$, we estimate
\begin{align*}
\int_{|\xi|\leq R}\tilde{c}(\xi)\,e^{-2\,\tilde{\mu}(\xi)\,t}\,d\xi \leq\tilde{c}(R) \int_{|\xi|\le R} e^{- \xi^2 \alpha(\xi)\,t}\,d\xi
\end{align*}
with
\begin{equation*}
\alpha(\xi ) := \frac{1}{\xi^2}\(1-\frac{1}{\sqrt{1+4\,\xi^2\,(1+4\,\xi^2)}}\)\,.
\end{equation*}
One easily sees that $\alpha$ has a local minimum at $\xi=0$ with $\alpha(0) = \alpha(\xi_1) = 1$, $\xi_1 = \(\sqrt{5}-1\)/4 \approx 0.3$, \emph{i.e.}\ for $0<R<\big(\sqrt{5}-1\big)/4 $ it holds that $\alpha(\xi)\geq 1$ for all $|\xi|<R$. Thus, for $0<R<\big(\sqrt{5}-1\big)/4 $ and $t>0$, the desired result follows.
\end{proof}

 \bigskip
\begin{acknowledgement}

This work has been partially supported by the Project EFI (ANR-17-CE40-0030) of the French National Research Agency (ANR) and the Amadeus project \emph{Hypocoercivity} no.~39453PH. J.D.~and C.S.~thank E.~Bouin, S.~Mischler and C.~Mouhot for stimulating discussions that took place during the preparation of~\cite{BDMMS}: some questions raised at this occasion are the origin for this contribution. A.A., C.S., and T.W.~were partially supported by the FWF (Austrian Science Fund) funded SFB F65.
\par\noindent{\small\copyright~2020 by the authors. This paper may be reproduced, in its entirety, for non-commercial purposes.}

\end{acknowledgement}

\newpage
\bibliographystyle{spmpsci}
\bibliography{ADSW2020}

\begin{thebibliography}{10}
\providecommand{\url}[1]{{#1}}
\providecommand{\urlprefix}{URL }
\expandafter\ifx\csname urlstyle\endcsname\relax
  \providecommand{\doi}[1]{DOI~\discretionary{}{}{}#1}\else
  \providecommand{\doi}{DOI~\discretionary{}{}{}\begingroup
  \urlstyle{rm}\Url}\fi

\bibitem{AAC20}
Achleitner, F., Arnold, A., Carlen, E.: The hypocoercivity index for the short
  and large time behavior of {ODE}s.
\newblock In preparation  (2020)

\bibitem{AAC}
Achleitner, F., Arnold, A., Carlen, E.A.: On linear hypocoercive {BGK} models.
\newblock In: From Particle Systems to Partial Differential Equations {III},
  pp. 1--37. Springer International Publishing (2016).
\newblock \urlprefix\url{https://doi.org/10.1007/978-3-319-32144-8_1}

\bibitem{Achleitner2018}
Achleitner, F., Arnold, A., Carlen, E.A.: On multi-dimensional hypocoercive
  {BGK} models.
\newblock Kinet. Relat. Models \textbf{11}(4), 953--1009 (2018).
\newblock \urlprefix\url{https://doi.org/10.3934/krm.2018038}

\bibitem{AAS19}
Achleitner, F., Arnold, A., Signorello, B.: On optimal decay estimates for
  {ODE}s and {PDE}s with modal decomposition.
\newblock In: Stochastic dynamics out of equilibrium, \emph{Springer Proc.
  Math. Stat.}, vol. 282, pp. 241--264. Springer, Cham (2019).
\newblock \urlprefix\url{https://doi.org/10.1007/978-3-030-15096-9_6}

\bibitem{addala2019l2hypocoercivity}
Addala, L., Dolbeault, J., Li, X., Tayeb, M.L.: {$L^2$}-hypocoercivity and
  large time asymptotics of the linearized {Vlasov-Poisson-Fokker-Planck}
  system.
\newblock Preprint arXiv  (2019).
\newblock \urlprefix\url{https://arxiv.org/abs/1909.12762}

\bibitem{armstrong2019variational}
Armstrong, S., Mourrat, J.C.: Variational methods for the kinetic
  {Fokker-Planck} equation.
\newblock Preprint arXiv  (2019).
\newblock \urlprefix\url{https://arxiv.org/abs/1409.5425}

\bibitem{AESW}
Arnold, A., Einav, A., Signorello, B., W{\"o}hrer, T.: Large time convergence
  of the non-homogeneous {Goldstein-Taylor} equation (2020).
\newblock \urlprefix\url{https://arxiv.org/abs/2007.11792}

\bibitem{ArnEinSigWoe}
Arnold, A., Einav, A., W\"{o}hrer, T.: On the rates of decay to equilibrium in
  degenerate and defective {Fokker-Planck} equations.
\newblock Journal of Differential Equations \textbf{264}(11), 6843 -- 6872
  (2018).
\newblock \urlprefix\url{https://doi.org/10.1016/j.jde.2018.01.052}

\bibitem{Arnold2014}
Arnold, A., Erb, J.: Sharp entropy decay for hypocoercive and non-symmetric
  {Fokker-Planck} equations with linear drift.
\newblock Preprint arXiv  (2014).
\newblock \urlprefix\url{https://arxiv.org/abs/1409.5425}

\bibitem{AJW}
Arnold, A., Jin, S., W\"{o}hrer, T.: Sharp decay estimates in local sensitivity
  analysis for evolution equations with uncertainties: from {ODE}s to linear
  kinetic equations.
\newblock J. Differential Equations \textbf{268}(3), 1156--1204 (2020).
\newblock \urlprefix\url{https://doi.org/10.1016/j.jde.2019.08.047}

\bibitem{ASS20}
Arnold, A., Schmeiser, C., Signorello, B.: Propagator norm and sharp decay
  estimates for {Fokker-Planck} equations with linear drift.
\newblock Preprint arXiv  (2020).
\newblock \urlprefix\url{https://arxiv.org/abs/2003.01405}

\bibitem{BS}
Bernard, {\'{E}}., Salvarani, F.: Optimal estimate of the spectral gap for the
  degenerate {Goldstein-Taylor} model.
\newblock Journal of Statistical Physics \textbf{153}(2), 363--375 (2013).
\newblock \urlprefix\url{https://doi.org/10.1007/s10955-013-0825-6}

\bibitem{BScorr}
Bernard, {\'{E}}., Salvarani, F.: Correction to: Optimal estimate of the
  spectral gap for the degenerate {Goldstein-Taylor} model.
\newblock Journal of Statistical Physics \textbf{181}(4), 1--2 (2020).
\newblock \urlprefix\url{https://doi.org/10.1007/s10955-020-02631-y}

\bibitem{BDLS}
Bouin, E., Dolbeault, J., Lafleche, L., Schmeiser, C.: Hypocoercivity and
  sub-exponential local equilibria.
\newblock Monatshefte f{\"u}r Mathematik  (2020).
\newblock \urlprefix\url{https://doi.org/10.1007/s00605-020-01483-8}

\bibitem{BDMMS}
Bouin, E., Dolbeault, J., Mischler, S., Mouhot, C., Schmeiser, C.:
  Hypocoercivity without confinement.
\newblock Pure and Applied Analysis \textbf{2}(2), 203--232 (2020).
\newblock \urlprefix\url{https://doi.org/10.2140/paa.2020.2.203}

\bibitem{BDS-very-weak}
Bouin, E., Dolbeault, J., Schmeiser, C.: Diffusion and kinetic transport with
  very weak confinement.
\newblock Kinetic \& Related Models \textbf{13}(2), 345--371 (2020).
\newblock \urlprefix\url{https://doi.org/10.3934/krm.2020012}

\bibitem{Bouin_2020}
Bouin, E., Dolbeault, J., Schmeiser, C.: A variational proof of {N}ash's
  inequality.
\newblock Rendiconti Lincei -- Matematica e Applicazioni \textbf{31}(1),
  211--223 (2020).
\newblock \urlprefix\url{https://doi.org/10.4171/rlm/886}

\bibitem{CRS}
Calvez, V., Raoul, G.: Confinement by biased velocity jumps: Aggregation of
  \emph{escherichia coli}.
\newblock Kinetic and Related Models \textbf{8}, 651 (2015).
\newblock \urlprefix\url{https://doi.org/10.3934/krm.2015.8.651}

\bibitem{Dolbeault22102012}
Dolbeault, J., Klar, A., Mouhot, C., Schmeiser, C.: Exponential rate of
  convergence to equilibrium for a model describing fiber lay-down processes.
\newblock Applied Mathematics Research eXpress  (2012).
\newblock \urlprefix\url{https://doi.org/10.1093/amrx/abs015}

\bibitem{Dolbeault2009511}
Dolbeault, J., Mouhot, C., Schmeiser, C.: Hypocoercivity for kinetic equations
  with linear relaxation terms.
\newblock Comptes Rendus Math{\'e}matique \textbf{347}(9-10), 511 -- 516
  (2009).
\newblock \urlprefix\url{https://doi.org/10.1016/j.crma.2009.02.025}

\bibitem{DMS-2part}
Dolbeault, J., Mouhot, C., Schmeiser, C.: Hypocoercivity for linear kinetic
  equations conserving mass.
\newblock Transactions of the American Mathematical Society \textbf{367}(6),
  3807--3828 (2015).
\newblock \urlprefix\url{https://doi.org/10.1090/s0002-9947-2015-06012-7}

\bibitem{FavSch}
Favre, G., Schmeiser, C.: Hypocoercivity and fast reaction limit for linear
  reaction networks with kinetic transport.
\newblock Journal of Statistical Physics \textbf{178}(6), 1319--1335 (2020).
\newblock \urlprefix\url{https://doi.org/10.1007/s10955-020-02503-5}

\bibitem{FePrTa17}
Fellner, K., , Prager, W., Tang, B.Q.: The entropy method for
  reaction-diffusion systems without detailed balance: First order chemical
  reaction networks.
\newblock Kinetic {\&} Related Models \textbf{10}(4), 1055--1087 (2017).
\newblock \urlprefix\url{https://doi.org/10.3934/krm.2017042}

\bibitem{goudon:hal-01421710}
Goudon, T., Alonso, R.J., Vavasseur, A.: {Damping of particles interacting with
  a vibrating medium}.
\newblock {Annales de l'Institut Henri Poincar{\'e} (C) Non Linear Analysis}
  (2016).
\newblock \urlprefix\url{https://doi.org/10.1016/j.anihpc.2016.12.005}

\bibitem{Herau}
H{\'e}rau, F.: Hypocoercivity and exponential time decay for the linear
  inhomogeneous relaxation {B}oltzmann equation.
\newblock Asymptot. Anal. \textbf{46}(3-4), 349--359 (2006).
\newblock
  \urlprefix\url{https://content.iospress.com/articles/asymptotic-analysis/asy741}

\bibitem{HJ}
Horn, R.A., Johnson, C.R.: Matrix analysis, second edn.
\newblock Cambridge University Press, Cambridge (2013).
\newblock \urlprefix\url{https://doi.org/10.1017/CBO9780511810817}

\bibitem{MR1057534}
Kawashima, S.: The {B}oltzmann equation and thirteen moments.
\newblock Japan J. Appl. Math. \textbf{7}(2), 301--320 (1990).
\newblock \urlprefix\url{https://doi.org/10.1007/BF03167846}

\bibitem{Mouhot-Neumann}
Mouhot, C., Neumann, L.: Quantitative perturbative study of convergence to
  equilibrium for collisional kinetic models in the torus.
\newblock Nonlinearity \textbf{19}(4), 969--998 (2006).
\newblock \urlprefix\url{https://doi.org/10.1088/0951-7715/19/4/011}

\bibitem{Nash58}
Nash, J.: Continuity of solutions of parabolic and elliptic equations.
\newblock Amer. J. Math. \textbf{80}, 931--954 (1958).
\newblock \urlprefix\url{https://doi.org/10.2307/2372841}

\bibitem{NeuSch}
Neumann, L., Schmeiser, C.: A kinetic reaction model: Decay to equilibrium and
  macroscopic limit.
\newblock Kinetic and Related Models \textbf{9}, 571 (2016).
\newblock \urlprefix\url{https://doi.org/10.3934/krm.2016007}

\bibitem{Pa92}
Pazy, A.: Semigroups of linear operators and applications to partial
  differential equations, \emph{Applied Mathematical Sciences}, vol.~44.
\newblock Springer-Verlag, New York (1983).
\newblock \urlprefix\url{https://doi.org/10.1007/978-1-4612-5561-1}

\bibitem{shizuta1985}
Shizuta, Y., Kawashima, S.: Systems of equations of hyperbolic-parabolic type
  with applications to the discrete {B}oltzmann equation.
\newblock Hokkaido Math. J. \textbf{14}(2), 249--275 (1985).
\newblock \urlprefix\url{https://doi.org/10.14492/hokmj/1381757663}

\bibitem{MR2927622}
Ueda, Y., Duan, R., Kawashima, S.: Decay structure for symmetric hyperbolic
  systems with non-symmetric relaxation and its application.
\newblock Arch. Ration. Mech. Anal. \textbf{205}(1), 239--266 (2012).
\newblock \urlprefix\url{https://doi.org/10.1007/s00205-012-0508-5}

\bibitem{Mem-villani}
Villani, C.: Hypocoercivity.
\newblock Mem. Amer. Math. Soc. \textbf{202}(950), iv+141 (2009).
\newblock \urlprefix\url{https://doi.org/10.1090/S0065-9266-09-00567-5}

\end{thebibliography}

\end{document}